\documentclass[10pt]{article}
\usepackage{geometry} 
\usepackage{amsthm}
\usepackage{hyperref}
\usepackage{mathtools}
\usepackage{tikz}  
\tikzstyle{int}=[draw, fill=blue!20, minimum size=2em]
\tikzstyle{init} = [pin edge={to-,thin,black}]
\usepackage{comment}
\usetikzlibrary{arrows}
\usepackage{rotating}
\usetikzlibrary{shapes,quotes,calc,angles,arrows,through,intersections,shadings}
\usepackage{tkz-euclide}
\usepackage{mathrsfs}
\newcommand{\R}{\mathbb{R}}

\newcommand{\N}{\mathbb{N}}

\newcommand{\1}{\mathds{1}}

\usepackage{enumerate}  
\usepackage{amssymb}
\usepackage{dsfont}
\usepackage{esvect}
\usepackage{bm}
\usepackage{tikz-cd}
\usepackage{graphicx}
\newcommand{\vertiii}[1]{{\left\vert\kern-0.25ex\left\vert\kern-0.25ex\left\vert #1 
    \right\vert\kern-0.25ex\right\vert\kern-0.25ex\right\vert}}
\graphicspath{ {images/} }

\newtheorem{definition}{Definition}[section]

\newtheorem{prop}{Proposition}[section]
\newtheorem{lem}{Lemma}[section]
\newtheorem{thm}{Theorem}[section]
\newtheorem{rmk}{Remark}[section]
\newtheorem*{statement}{Statement of Main Result}
\numberwithin{equation}{section}
\title{A Rigorous Derivation of a Boltzmann System for a Mixture of Hard-Sphere Gases}
\author{Ioakeim Ampatzoglou, Joseph K. Miller, and Nata\v{s}a Pavlovi\'{c}}
\date{\today}
\begin{document}
\maketitle 
\begin{abstract}
In this paper, we rigorously derive a Boltzmann equation for mixtures from the many body dynamics of two types of hard sphere gases. We prove that the microscopic dynamics of two gases with different masses and diameters is well defined, and introduce the concept of a two parameter BBGKY hierarchy to handle the non-symmetric interaction of these gases. As a corollary of the derivation, we prove Boltzmann's \emph{propagation of chaos} assumption for the case of a mixtures of gases. 
\end{abstract}
\tableofcontents 
\section{Introduction}\label{introduction}
Much effort has been put into studying the dynamics of a collection of interacting gases. Gas mixtures such as helium and xenon were studied as a possible coolant for nuclear reactors in spacecraft \cite{HV07}. Similar mixtures have also undergone extensive analysis as possible coolants for thermoacoustic refrigerators \cite{CPN11}. Sound propagation in binary mixtures \cite{FM04}, and hypersonic shockwave analysis for aerospace applications \cite{TA09,TA10} have also been studied. \\

In the case of two gases evolving in $\R^d$, the phase space of a single particle takes the form $\R^d \times \R^d$. Call one gas type A and the other type B. If $g_0(x,v):\R^{d}\times \R^d \rightarrow \R$ is an initial density distribution on phase space of the type A gas, and $h_0(x,v): \R^d \times \R^d \rightarrow \R$ the distribution of the type B gas, the evolution of the two gases is modeled by the Boltzmann system for mixtures: 
\begin{equation}\label{PDE-intro}
\begin{cases}
\partial_t g + v \cdot \nabla_x g = c_1 Q_{1,1}(g,g) + c_{1,2}Q_{1,2}(g,h) \\
\partial_t h + v \cdot \nabla_x h = c_2 Q_{2,2}(h,h) + c_{2,1} Q_{2,1}(h,g)\\
g(0,x,v) = g_0(x,v), \qquad h(0,x,v) = h_0(x,v).
\end{cases}
\end{equation}
Here, $Q_{i,j}$ are integral, bilinear operators called the collision kernels, and $c_1, c_2, c_{2,1}, c_{1,2}\in (0,\infty)$ are constants. They encode all the information about the possible collisions between two particles. This system of equations was initially studied starting in the 1950's by Chapman and Cowling \cite{CC52}, and later by Hamel in the 1960's \cite{Ham65}. The system \eqref{PDE-intro} can be seen as a generalization of the standard Boltzmann equation, which has its roots in the works of Maxwell in 1867 \cite{Max67} and Boltzmann in 1872 \cite{Bol72}. It is interesting to note that although Maxwell did not study the Boltzmann system \eqref{PDE-intro} explicitly, he did consider collisions between particles of disparate masses \cite{Max67}. \\

Much work has been done mathematically to study the solutions of equation \eqref{PDE-intro}. Global well posedness for mild solutions has been proven in the case of inverse power molecular interactions and small data \cite{HNY07}, and in the case of hard potentials also with small data \cite{HN10}. Stability in various formulations has been extensively studied \cite{BD16,HNY07,HN10}, in addition to numerous numerical schemes \cite{Kos09,TA09,ADKT14}. Furthermore, rigorous connections between \eqref{PDE-intro} and the compressible Navier-Stokes equations for mixtures of fluids have been established \cite{BD16}. Explicit solutions to the space homogeneous variant of \eqref{PDE-intro} have been studied \cite{BG06} in addition to recent proofs of global well posedness and propagation of moments \cite{GP20,CGP20}. \\ 

However, despite the mathematical progress on the subject of the system \eqref{PDE-intro}, no work has been done on rigorously deriving the system from a system of hard spheres. Mathematical derivation results for the single type hard sphere system trace back to the pioneering work of Alexander \cite{Ale75,Ale76}, Lanford \cite{Lan75}, King \cite{Kin75}. Gallagher, Saint-Raymond, and Texier refined and extended the derivation of a single gas Boltzmann equation for hard spheres and short range potentials \cite{GSRT13}. More recently, the first and last authors of this paper considered more complex interactions which model dense gases with ternary interactions \cite{AP19a} and binary-ternary interactions \cite{AP19b}. It is of relevance to note that each of these derivation results considers a Boltzmann like equation for a single type of particle. In addition to the theoretical framework developed above for a single type gas, new techniques for multiple gases are needed to keep track of the evolution and correlation of one type of gas to the other. This is exactly what we do in this paper. \\

More precisely, in this paper, we derive the Boltzmann system \eqref{PDE-intro} from a mixture of finitely many hard spheres. In order to do this, we consider $N_1$ hard spheres of mass $M_1$ and diameter $\epsilon_1$ mixed with a system of $N_2$ hard spheres of mass $M_2$ and diameter $\epsilon_2$. For the $i$th particle of mass $M_1$, we denote its center by $x_i$ and its velocity by $v_i$. Similarly, for the $i$th particle of mass $M_2$, we denote its center by $y_i$ and its velocity by $w_i$. For notational simplicity, we will write the vector of all positions and velocities by
$$
Z_{(N_1,N_2)} = (x_1, \dots, x_{N_1},y_1, \dots, y_{N_2}, v_1, \dots, v_{N_1}, w_1, \dots, w_{N_2}).
$$
The natural phase space for this collection of particles is  
$$
\mathcal{D}_{(\epsilon_1, \epsilon_2)}^{(N_1,N_2)}: = \left\{ Z_{(N_1, N_2)} : \, \, 
\begin{aligned} 
& \forall i\neq j,\,\, |x_i - x_j| \geq \epsilon_1, \, \, |y_i - y_j| \geq \epsilon_2,  \\
& \forall i,j, \, \, |x_i -y_j| \geq \frac{\epsilon_1 + \epsilon_2}{2}
\end{aligned} \right\}.
$$ 
To model a mixture of gases, we assume that each particle is a hard sphere that evolves according to Newton's laws. That is, if non-collisional, we assume that the particles perform rectilinear motion, i.e. 
\begin{equation}\label{Newton}
\begin{cases}
\dot{x_i} = v_i, & \dot{y_j} = w_j \\
\dot{v_i} = 0, & \dot{w_j} = 0.\\
(x_i(0), v_i(0)) = (x_{i,0},v_{i,0}), & (y_i(0),w_i(0)) = (y_{i,0}, w_{i,0}).
\end{cases}
\end{equation}
If there exists exactly one collisional pair of particles, we assume the collision is completely elastic, i.e. energy and momentum are conserved under collisions. Consequently, we have the following collisional laws.
\begin{itemize}
\item[1.] If $(x_i, v_i)$ and $(x_j, v_j)$ are such that $|x_i -x_j| = \epsilon_1$, then the pre-collisional velocities $(v_i,v_j)$ give rise to the post-collisional velocities $(v_i^*, v_j^*)$ by
\begin{equation}\label{11a intro}
v_i^{*} = v_i -\Bigg((v_i- v_j)\cdot\frac{ (x_i - x_j)}{\Vert x_i - x_j \Vert} \Bigg) \frac{ (x_i - x_j)}{\Vert x_i - x_j \Vert}
\end{equation}
\begin{equation}\label{11b intro}
v_j^{*} = v_j + \Bigg((v_i- v_j)\cdot\frac{ (x_i - x_j)}{\Vert x_i - x_j \Vert} \Bigg) \frac{ (x_i - x_j)}{\Vert x_i - x_j \Vert}.
\end{equation}

\item[2.] If $(y_i, w_i)$ and $(y_j,w_j)$ are such that $|y_i -y_j| = \epsilon_2$, then the pre-collisional velocities $(w_i,w_j)$ give rise to the post-collisional velocities $(w_i^*, w_j^*)$ by 
\begin{equation}\label{22a intro}
w_i^{*} = w_i -\Bigg((w_i- w_j)\cdot\frac{ (y_i - y_j)}{\Vert y_i - y_j \Vert} \Bigg) \frac{ (y_i - y_j)}{\Vert y_i - y_j \Vert}
\end{equation}
\begin{equation}\label{22b intro}
w_j^{*} = w_j +\Bigg((w_i- w_j)\cdot\frac{ (y_i - y_j)}{\Vert y_i - y_j \Vert} \Bigg) \frac{ (y_i - y_j)}{\Vert y_i - y_j \Vert}.
\end{equation}

\item[3.] If $(x_i, v_i)$ and $(y_j, w_j)$ are such that $|x_i - y_j| = (\epsilon_1 + \epsilon_2)/2$, then the pre-collisional velocities $(v_i,w_j)$ give rise to the post-collisional velocities $(v_i^*, w_j^*)$ by
\begin{equation}\label{12a intro}
v_i^{*} = v_i -\frac{2M_2}{M_1 + M_2} \Bigg( (v_i- w_j )\cdot\frac{ (x_i - y_j)}{\Vert x_i - y_j \Vert} \Bigg) \frac{ (x_i - y_j)}{\Vert x_i - y_j \Vert}
\end{equation}
\begin{equation}\label{12b intro}
w_j^{*} = w_j +\frac{2M_1}{M_1 + M_2}\Bigg((v_i- w_j)\cdot\frac{ (x_i - y_j)}{\Vert x_i - y_j \Vert} \Bigg) \frac{ (x_i - x_j)}{\Vert x_i - y_j \Vert}.
\end{equation}
\end{itemize}
 For convenience, we call the above system which satisfies \eqref{Newton}, and \eqref{11a intro} - \eqref{12b intro} an $N_1, N_2, \epsilon_1, \epsilon_2$ particle system. This particle system describes a deterministic, i.e. pointwise defined, evolution on the phase space $\mathcal{D}_{(\epsilon_1,\epsilon_2)}^{(N_1, N_2)}$. In Section \ref{Dynamics of Mixed Particles}, we prove that for almost every initial configuration in phase space, the above flow is well defined and measure preserving. Consequently, an initial density $f_{(N_1, N_2),0}$ on the phase space $\mathcal{D}_{(\epsilon_1,\epsilon_2)}^{(N_1,N_2)}$ evolves according to the following Liouville equation:
\begin{equation}\label{Liouville-intro}
\begin{cases}
&\partial_{t} f_{(N_1,N_2)} + \sum_{k=1}^{N_1} v_k \cdot \nabla_{x_k} f_{(N_1,N_2)} + \sum_{k=1}^{N_2} w_k \cdot \nabla_{y_k} f_{(N_1,N_2)} = 0  \text{ on } \mathring{\mathcal{D}}^{(N_1,N_2)}_{(\epsilon_1, \epsilon_2)} \\
& f_{(N_1, N_2)}(Z_{(N_1,N_2)}^*) = f_{(N_1,N_2)}(Z_{(N_1,N_2)}) \text{ on } \partial \mathcal{D}^{(N_1, N_2)}_{(\epsilon_1, \epsilon_2)}, \\
&f_{(N_1, N_2)}(0) = f_{(N_1, N_2),0}
\end{cases} 
\end{equation}
where $Z_{(N_1, N_2)}^*$ is the post-collisional configuration related to the pre-collisional configuration $Z_{(N_1, N_2)}$ by the collisional laws given in \eqref{11a intro} - \eqref{12b intro}.\footnote{For a precise definition, see Definition \ref{T-collision}.} We note that the above boundary condition is defined for a full surface measure subset of the boundary (see Section \ref{Dynamics of Mixed Particles}).  \\

In order to understand the statistical behavior of the $N_1, N_2, \epsilon_1, \epsilon_2$ particle system, we require each of the particles of the same mass to behave identically. This manifests as the condition that $f_{(N_1, N_2)}$ is invariant under permutations of the $(x_i,v_i)$ variables and invariant under permutations of the $(y_j,w_j)$ variables. We call this condition the \emph{identical particles} assumption.\footnote{The density $f_{(N_1,N_2)}$ obeying \eqref{Liouville-intro} and the identical particles assumption represents the \emph{statistical ensemble} for the gas mixture of the $N_1, N_2, \epsilon_1, \epsilon_2$ particle system.} Note that we do not require that $f_{(N_1,N_2)}$ is invariant under interchanging $(x_i,v_i)$ and $(y_j,w_j)$ variables for any $i, j$. This lack of symmetry forces $f_{(N_1,N_2)}$ to take into account the behavior of both types of particles, but places the $N_1, N_2, \epsilon_1, \epsilon_2$ particle system outside of the standard framework developed for hard sphere systems of a single type, such as those in \cite{GSRT13}. In order to handle this asymmetry, we introduce the following definition which will be given more precisely in Definition \ref{mixed-marginals}.
\begin{definition}
For each $s \in \{1, \dots, N_1-1\}$ and $\ell \in \{1, \dots N_2 -1\}$ we define the \emph{mixed marginal} of $f_{(N_1, N_2)}$ to be
$$
f^{(s, \ell )}_{(N_1, N_2)} : = \int \mathds{1}_{\mathcal{D}^{(N_1,N_2)}_{(\epsilon_1,\epsilon_2)}}f_{(N_1, N_2)} dx_{s+1} \dots dx_{N_1} dv_{s+1} \dots dv_{N_1} dy_{\ell+1} \dots dy_{N_2} dw_{\ell+1} \dots dw_{N_2}.
$$
\end{definition}
This concept of a mixed marginal is key to our analysis and allows us to distinguish the behavior of both types of particles.\\

Integrating by parts equation \eqref{Liouville-intro} and using the identical particles assumption, we will derive in Section \ref{formal calculations} an evolution system for $f^{(s,\ell)}_{(N_1,N_2)}$ which will be given as a two parameter hierarchy of equations called the BBGKY\footnote{Bogoliubov, Born, Green, Kirkwood, and Yvon} hierarchy:
\begin{equation}\label{BBGKY intro}
\left(\frac{D}{dt}\right)_{(s,\ell)}f_{(N_1,N_2)}^{(s,\ell)} = \mathcal{C}_{(s,\ell),(s+1,\ell)}^{(N_1,N_2)} f_{(N_1,N_2)}^{(s+1,\ell)} +\mathcal{C}_{(s,\ell),(s,\ell+1)}^{(N_1,N_2)}  f_{(N_1,N_2)}^{(s,\ell+1)}.
\end{equation}
Here, the operators $\mathcal{C}_{(s,\ell),(s+1,\ell)}^{(N_1,N_2)} , \mathcal{C}_{(s,\ell),(s,\ell+1)}^{(N_1,N_2)}$ are integral operators given explicitly in Subsection \ref{Mixed Marginals and BBGKY Hierarchies} and 
$$
\left(\frac{D}{dt}\right)_{(s,\ell)} : = \partial_t + \sum_{k=1}^{s} v_k \cdot \nabla_{x_k} + \sum_{k=1}^\ell w_k \cdot \nabla_{y_k}.
$$
Solutions $f^{(s,\ell)}_{(N_1,N_2)}$ to the BBGKY hierarchy represent the densities of subsystems to the $N_1, N_2, \epsilon_1, \epsilon_2$ particle system. Taking $N_1, N_2 \rightarrow \infty$ and $\epsilon_1, \epsilon_2 \rightarrow 0$ appropriately, we formally obtain an infinite two parameter hierarchy of equations called the Boltzmann hierarchy: 
\begin{equation}\label{Boltzmann intro}
\left(\frac{D}{dt}\right)_{(s,\ell)}f^{(s,\ell)} = \mathscr{C}_{(s,\ell),(s+1,\ell)} f^{(s+1,\ell)} +\mathscr{C}_{(s,\ell),(s,\ell+1)}  f^{(s,\ell+1)}.
\end{equation}
Here, the operators $\mathscr{C}_{(s,\ell),(s+1,\ell)},\mathscr{C}_{(s,\ell),(s,\ell+1)}$ are integral operators given explicitly in Subsection \ref{Mixed Marginals and BBGKY Hierarchies}. Solutions $f^{(s,\ell)}$ to the Boltzmann hierarchy correspond to the densities of finite subsystems to a mixture of two gases. \\

We are now ready to give an informal statement of our main theorem. The rigorous statements are given in Theorem \ref{main-thm} and Theorem \ref{thm 3}.
\begin{statement}\label{thm intro}
 Let $f^{(s,\ell)}_0$ be a sequence of initial data for the Boltzmann hierarchy \eqref{Boltzmann intro}, and for each $N_1, N_2 \in \N_+$, let $f^{(s,\ell)}_{(N_1, N_2),0}$ be an approximating sequence of data for the BBGKY hierarchy. Furthermore, let $f^{(s,\ell)}$ solve \eqref{Boltzmann intro} with initial data $f^{(s,\ell)}_0$ and let $f^{(s,\ell)}_{(N_1,N_2)}$ solve \eqref{BBGKY intro} with initial data $f_{(N_1, N_2),0}^{(s,\ell)}$. Then, for fixed $c_1, c_2,b \in \R_+$ in the mixed Boltzmann-Grad scalings
$$
N_1 \epsilon_1^{d-1} \equiv c_1, \qquad N_2 \epsilon_2^{d-1} \equiv c_2,  \qquad \epsilon_1 \equiv b \epsilon_2,
$$
we have that $f^{(s,\ell)}_{(N_1,N_2)}$ converges to $f^{(s,\ell)}$ in the sense of observables\footnote{Convergence in observables is defined explicitly in Definition \ref{convergence of observables}.} as $N_1, N_2 \rightarrow \infty$ and $\epsilon_1, \epsilon_2 \rightarrow 0$.\\

Furthermore, if $f^{(s,\ell)}_0$ is a tensor of the form $g^{\otimes s}_0 \otimes h^{\otimes \ell}_0$, then $f^{(s,\ell)}_{(N_1,N_2)}$ converges to $g^{\otimes s}\otimes h^{\otimes \ell}$ in the sense of observables as $N_1, N_2 \rightarrow \infty$ and $\epsilon_1, \epsilon_2 \rightarrow 0$, where $(g,h)$ a solution of equation \eqref{PDE-intro} with $(g_0,h_0)$ as initial data.\footnote{Here, we use the convention that $g^{\otimes s}\otimes h^{\otimes \ell}(Z_{(s,\ell)}) = \prod_{i=1}^s g(x_i,v_i) \prod_{i=1}^\ell h(y_i, w_i)$} The constants $c_{1,2}, c_{2,1}$ are given by 
\begin{equation}\label{const intro}
c_{1,2} = c_2 \left( \frac{1+b}{2}\right)^{d-1}, \qquad c_{2,1} = c_1 \left(\frac{1+b^{-1}}{2} \right)^{d-1}.
\end{equation}
\end{statement}
We note that these constants $c_{1}, c_{2}, c_{1,2}, c_{2,1}$ cannot be picked arbitrarily. In particular, we have the constants must lie in a three dimensional subset of $(0,\infty)^4$ which is parametrized by $c_1, c_{2},$ and $b$. Crucially, the constants $c_{1,2}$ and $c_{2,1}$ in \eqref{PDE-intro} which describe the strength of the interaction between the two gases can be calculated from the inverse mean free paths $c_1, c_2$ and the ratio $b$ of the diameters of the spheres. Moreover, \eqref{const intro} shows that as $b$ grows large, the constant $c_{1,2}$ grows large while the constant $c_{2,1}$ becomes small. This agrees with physical intuition that if one gas is comprised of larger particles than the other, it has a larger effect on the system as a whole. See \eqref{boltz-grad}-\eqref{mixed-boltz-grad,2} for the computation of these constants. 

\begin{rmk}
As a result of the above statement, Boltzmann's propagation of chaos assumption is rigorously verified. That is, in the infinite particle limit, the joint density $f^{(s,\ell)}_{(N_1, N_2)}$ factors as $g^{\otimes s}\otimes h^{\otimes \ell}$ as $N_1, N_2 \rightarrow \infty$. It is worthwhile to note that while the density $g^{\otimes s} \otimes h^{\otimes \ell}$ indeed factors on $(\R^d \times \R^{d})^{(s+\ell)}$, the evolution of $g$ still depends on the evolution of $h$ via \eqref{PDE-intro} and vice versa. 
\end{rmk}

The main novel contribution of this work is to develop a theoretical framework which can handle multiple types of collisions. That is achieved by employing a mixed marginal and two parameter hierarchies of equations, that to the best of our knowledge had not been studied in the kinetic context before. This two parameter hierarchy generates a quartic tree of interactions between the different types of particles. We also note that framework introduced below can be adapted to any number $k$ of different types of gases. The resulting marginals would be indexed by an element of $\N^k$. For the sake of clarity, we will present only the result for the mixture of two gases. \\

We structure the paper as follows. In Section \ref{formal calculations}, we introduce the concept of a \emph{mixed marginal}, and derive a hierarchy of equations which relate mixed marginals. In Section \ref{Dynamics of Mixed Particles}, we prove the existence almost everywhere of a global mixed particle flow. Next, Section \ref{Local Well-Posedness} covers the well-posedness theory for the hierarchies, in addition to the well-posedness theory for the Boltzmann system \eqref{PDE-intro}. Section \ref{Statement of the Main Theorem} gives the precise statements of the main results of the paper. The following sections are devoted to proving these results. First, a series of approximations are proved in Section \ref{Reduction to Term-by-Term Convergence} which allow us to handle the observables. Next, an adjunction lemma is proved in Section \ref{Good Configurations and Stability}, which enables us to add particles to our system while keeping track of our global flow. Section \ref{Elimination of Recollisions for Mixtures} uses this control to obtain a formulation of the observables in terms of specific \emph{pseudo-trajectories}. The final Section \ref{Convergence Proof} pieces all of the previous approximations together to prove the main theorem. 

\section*{Acknowledgements}
 I.A., J.K.M., and N.P. acknowledge support from NSF grants DMS-1840314 and DMS-2009549. I.A. acknowledges support from the Simons Collaboration on Wave Turbulence. J.K.M. acknowledges support from the Provost’s Graduate Excellence Fellowship at The University of Texas at Austin.  Authors are thankful to Thomas Chen, Erica de la Canal,
 and Irene M. Gamba for helpful discussions regarding physical and mathematical aspects of the problem. 
\section{Vocabulary of The Paper}\label{formal calculations}
\subsection{Definitions} \label{mixed-particle-notation-definitions} In this section, we consider two types of hard spheres evolving in $\R^d$. We call one group of particles \textit{type $(1,0)$ particles}, and the other \textit{type $(0,1)$ particles}. We assume that all particles perform rectilinear motion in $\R^d$, until they undergo a collision with another particle of either type. The collisions occurring are assumed to be perfectly elastic and instantaneous. Our goal is to keep track of both gases separately, extracting some qualitative information about the evolution of their probability distributions. To do this, we establish some new notational conventions simplifying the combinatorics involved. While the below notations are cumbersome, they allow us to carry out proofs without repetitious arguments and easily generalize to $k$ many types of particles.
\begin{itemize} 
\item We define the set
\begin{equation} \label{set of types}
\mathscr{T} := \{ \alpha\in \N^2: \, |\alpha| = 1\}
\end{equation}
to be the set of all types of particles. Greek indices such as $\alpha,\beta,$ or $\sigma$ in $\mathscr{T}$ will designate the type of particle we are considering. We introduce an ordering on the set $\mathscr{T}$ by simply declaring $(1,0)< (0,1)$. \\

\item For each $\alpha \in \mathscr{T}$ denote the number of type $\alpha$ particles by $N_\alpha$, their diameter by $\epsilon_\alpha$ and their mass by $M_\alpha$. For $i \in \{1, \dots, N_\alpha\}$, $x_i^\alpha \in \R^{d}$ will denote the position vector of $i$th type $\alpha$ particle, and $v_i^\alpha$ will denote its velocity. For a more compact notation, we will write $X_{N_\alpha}^\alpha = (x_1^{\alpha}, \dots, x_{N_\alpha}^\alpha)$, $V_{N_\alpha}^\alpha = (v_1^{\alpha}, \dots, v_{N_\alpha}^\alpha)$ and  $Z_{N_\alpha}^\alpha = (X_{N_\alpha}^\alpha, V_{N_\alpha}^\alpha)$. \\

\item For each $\alpha, \beta \in \mathscr{T}$, we define the \emph{interaction distance} between a particle of type $\alpha$ and a particle of type $\beta$ to be the quantity
$$
\epsilon_{(\alpha,\beta)} = \frac{\epsilon_\alpha + \epsilon_\beta}{2}.
$$
Additionally, we define the index set of \emph{interacting pairs} to be 
\begin{equation}\label{interacting pairs}
\mathcal{I}^{(\alpha,\beta)} = \mathcal{I}_{(N_\alpha,N_\beta)}^{(\alpha,\beta)}: =\begin{cases}
\{ 1, \dots, N_\alpha \} \times \{1, \dots, N_\beta \},& \alpha \neq \beta\\
\{(i,j) \in \{1, \dots, N_\alpha \}^2 : \, \, i < j \}, & \alpha = \beta.
\end{cases} 
\end{equation}
The set of indices $(i,j) \in \mathcal{I}^{(\alpha,\beta)}$ above is exactly the indices $i$ of type $\alpha$ particles and indices $j$ of type $\beta$ particles which are interacting each other (the ordering $i<j$ excludes double counting for particles of the same type).\\

\item For convenience, let us write the full vector of positions and velocities as 
\begin{equation}\label{Z-def}
Z=Z_{(N_{(1,0)},N_{(0,1)})} := \left(X_{N_{(1,0)}}^{(1,0)}, X_{N_{(0,1)}}^{(0,1)}, V_{N_{(1,0)}}^{(1,0)}, V_{N_{(0,1)}}^{(0,1)}\right).
\end{equation}
The spacial variables of $Z_{(N_{(1,0)}, N_{(0,1)})}$ will be written as $X_{(N_{(1,0)}, N_{(0,1)})} =(X_{N_{(1,0)}}^{(1,0)}, X_{N_{(0,1)}}^{(0,1)})$ and the velocity variables of $Z_{(N_{(1,0)}, N_{(0,1)})}$ will be written as $V_{(N_{(1,0)}, N_{(0,1)})} =(V_{N_{(1,0)}}^{(1,0)}, V_{N_{(0,1)}}^{(0,1)})$. We will sometimes abuse notation slightly and write $x_i^\alpha(\cdot)$ or $v_i^\alpha(\cdot)$ to denote the projection onto the correct component:
\begin{equation}\label{coordinate-projections}
x_i^\alpha(Z) = x_i^\alpha ,\qquad v_i^\alpha(Z) = v_i^\alpha.\\
\end{equation}

\item We will often use the convention that $N_1 = N_{(1,0)}, N_2 = N_{(0,1)}$, $\epsilon_1 = \epsilon_{(1,0)}$, $\epsilon_2=\epsilon_{(0,1)}$, $M_1 = M_{(1,0)}$, and $M_2 = M_{(0,1)}$.
\end{itemize}
With these notational conventions, our phase space is given by 
\begin{equation}\label{phase-space}
\mathcal{D} = \mathcal{D}_{(\epsilon_{1}, \epsilon_{2})}^{(N_{1}, N_{2})} : = \bigcap_{\alpha \leq \beta } \bigcap_{(i,j) \in \mathcal{I}^{(\alpha,\beta)} }\{ Z : \, \, |x_i^\alpha - x_j^\beta| \geq \epsilon_{(\alpha,\beta)} \}.
\end{equation}
where the outer intersection is taken over all $\alpha,\beta \in \mathscr{T}$ and $\alpha \leq \beta$. The condition that $\alpha \leq \beta$ excludes double counting (since $|x_i^\alpha - x_j^\beta| \geq \epsilon_{(\alpha,\beta)}$ implies $|x_j^\beta - x_i^\alpha| \geq \epsilon_{(\alpha,\beta)}$).\\

\begin{rmk}
 We note that while these notional conventions are unwieldy, their use is necessary for the derivation procedure. In particular, in Section \ref{Reduction to Term-by-Term Convergence} we expand solutions of the BBGKY and Boltzmann hierarchies in terms of their "collision histories". By using multiindices, keeping track of this collision history is simplified.
\\
 We also remark that another advantage of the above notation is that it can be extended in a natural way to $k$ types of particles by defining the set of types to be 
 $$
 \mathscr{T}_k = \{ \alpha \in \N^k : |\alpha| =1\}.
 $$ 
 For simplicity, we will only present the case $k=2$.
 \end{rmk}

\subsection{Dynamics and Liouville's equation} \label{Dynamics and Liouville's equation}
 We assume that particles perform free motion as long as there is no collision i.e. for each $\alpha\in \mathscr{T}$,
 \begin{equation}\label{free motion}
 \dot{x}_i^\alpha=v_i^\alpha,\qquad \dot{v}_i^\alpha=0,\quad\forall i\in\{1,...,N_\alpha\}.
\end{equation}  
When two particles collide, we assume that they behave like hard spheres. That is, if for some fixed $\alpha, \beta \in \mathscr{T}$ and $(i,j) \in \mathcal{I}^{(\alpha,\beta)}$ we have $|x_i^\alpha - x_j^\beta| = \epsilon_{(\alpha,\beta)}$, then the pre-collisional velocities $v_i^\alpha, v_j^\beta$ are instantaneously changed to the post-collisional velocities $(v_i^\alpha)^*, (v_j^\beta)^*$ by
\begin{equation}\label{i alpha}
(v_i^\alpha)^* = v_i^\alpha - \frac{2M_\beta}{M_\alpha + M_\beta} \left( (v_i^\alpha - v_j^\beta) \cdot \frac{x_i^\alpha - x_j^\beta}{\Vert x_i^\alpha - x_j^\beta\Vert} \right) \frac{x_i^\alpha - x_j^\beta}{\Vert x_i^\alpha - x_j^\beta\Vert}
\end{equation} 
\begin{equation}\label{j beta}
(v_j^\beta)^* = v_j^\beta + \frac{2M_\alpha}{M_\alpha + M_\beta} \left( (v_i^\alpha - v_j^\beta) \cdot \frac{x_i^\alpha - x_j^\beta}{\Vert x_i^\alpha - x_j^\beta\Vert} \right) \frac{x_i^\alpha - x_j^\beta}{\Vert x_i^\alpha - x_j^\beta\Vert}
\end{equation} 
Equations \eqref{i alpha} and \eqref{j beta} are consequences of the conservation of energy and the conservation of momentum: 
\begin{equation}\label{cons of mom}
M_\alpha v_i^\alpha  + M_\beta v_j^\beta = M_\alpha (v_i^\alpha)^* + M_\beta (v_j^\beta)^*
\end{equation}
\begin{equation}\label{cons of en}
M_\alpha |v_i^\alpha|^2 + M_\beta |v_j^\beta|^2 = M_\alpha |(v_i^\alpha)^*|^2 + M_\beta |(v_j^\beta)^*|^2
\end{equation}
In Section \ref{Dynamics of Mixed Particles}, we show that these conditions give dynamics which is globally in time defined for a.e. initial configuration. \\

We now consider probability densities which are constant along the above dynamics. Formally assuming sufficient regularity for our calculations to make sense, the generated flow yields that the probability density $f_{(N_1,N_2)}$ of the full particle system satisfies the Liouville equation:
\begin{equation}\label{liouville}
\partial_t f_{(N_1, N_2)} +\sum_{\alpha \in \mathscr{T}} \sum_{i =1}^{N_\alpha} v_i^\alpha \cdot \nabla_{x_i^\alpha} f_{(N_1,N_2)} = 0 \quad \text{ on $[0,T] \times \mathring{\mathcal{D}}$ },
\end{equation}
where $\mathring{\mathcal{D}}$ is the interior of the phase space $\mathcal{D}$ defined in \eqref{phase-space}. This is accompanied by the boundary condition 
\begin{equation}\label{boundary-condition}
f_{(N_1,N_2)}(t, Z) = f_{(N_1, N_2)} (t, Z^*), \quad \text{ for all $(t,Z) \in [0,T]\times \partial{ \mathcal{D}}$}.
\end{equation}
We define $Z^*$ in the following way. If $Z \in \partial{\mathcal{D}}$ is such that there exists exactly one pair $\alpha, \beta \in \mathscr{T}$ and exactly one pair $(i,j) \in \mathcal{I}^{(\alpha,\beta)}$ such that $|x_i^\alpha - x_j^\beta| = \epsilon_{(\alpha,\beta)}$, then $Z^*$ is the vector $Z$ with the $v_i^\alpha, v_j^\beta$ components replaced by $(v_i^\alpha)^*, (v_j^\beta)^*$ as defined in \eqref{i alpha},\eqref{j beta}. It can be shown (see Section \ref{Dynamics of Mixed Particles}) that the set of all such $Z\in \partial \mathcal{D}$ fills a full surface measure subset of $\partial \mathcal{D}$, and so the boundary equation \eqref{boundary-condition} is defined almost everywhere. 

 \subsection{Symmetry with respect to same type particles}
In order to observe the statistical behavior of a mixture of gases, we require all particles of the same type to behave identically. For this reason we assume that the joint probability density $f_{(N_1,N_2)}$ is invariant under permutations among the same type of particles. Mathematically, we assume that for any $\sigma \in S(N_1)$ and $\sigma' \in S(N_2)$\footnote{Here, we let $S(N)$ denote the symmetric group on $\{1, \dots, N\}$}
\begin{equation}\label{identical_particles}
f_{(N_1,N_2)}(t, Z) = f_{(N_1,N_2)}(t, \sigma\oplus \sigma'(Z) ),
\end{equation}
 where we are defining 
 $$
 \sigma \oplus \sigma'(Z) = (x_{\sigma(1)}^{(1,0)}, \dots, x_{\sigma(N_1)}^{(1,0)}, x_{\sigma'(1)}^{(0,1)}, \dots, x_{\sigma'(N_2)}^{(0,1)}, v_{\sigma(1)}^{(1,0)}, \dots, v_{\sigma(N_1)}^{(1,0)}, v_{\sigma'(1)}^{(0,1)}, \dots, v_{\sigma'(N_2)}^{(0,1)}).
 $$
That is, we require that $f_{(N_1, N_2)}$ is invariant under the above action by the group $S(N_1) \oplus S(N_2)$.

\subsection{Mixed Marginals and BBGKY Hierarchies} \label{Mixed Marginals and BBGKY Hierarchies}
 Since Liouville's equation \eqref{liouville} is a linear transport equation, it yields a complete description of the mixed particle system. However, since the number of particles is extremely large, efficiently solving it is almost impossible. As mentioned in the introduction, we wish to extract  a statistical description  with the hope that qualitative properties of the gases mixture will be revealed as $N_1,N_2\to\infty$ and $\epsilon_1,\epsilon_2\to 0$. In the case of a single gas of hard spheres, this description is the Boltzmann equation. Rigorous derivation of the Boltzmann equation was made by Lanford \cite{Lan75} and later revisited by Gallagher \cite{GSRT13}. Their key idea is to derive a hierarchy of equations for the marginal densities of the density function. In the mixture case though,  it is impossible to keep track of the two gases separately.  The reason is that by using ordinary marginals, we would view the mixture as a uniform  gas of  $N_1+N_2$ particles. To overcome this problem, we introduce the notion of \textit{mixed marginals}.
\\
\begin{definition}\label{mixed-marginals}
For each $\alpha \in \mathscr{T}$, let $s_\alpha \in \{1, \dots, N_\alpha -1\}$ and let $s_1 := s_{(1,0)}$, $s_2 := s_{(0,1)}$. We use the notation
$$
Z_{(s_1, s_2)} : = (x_{1}^{(1,0)}, \dots, x_{s_1}^{(1,0)}, x_{1}^{(0,1)}, \dots, x_{s_2}^{(0,1)},v_{1}^{(1,0)}, \dots, v_{s_1}^{(1,0)}, v_{1}^{(0,1)}, \dots, v_{s_2}^{(0,1)}) 
$$
$$
Z^{(s_1, s_2)}_{(N_1,N_2)} : = (x_{s_1 + 1}^{(1,0)}, \dots, x_{N_1}^{(1,0)}, x_{s_2 + 1}^{(0,1)}, \dots, x_{N_2}^{(0,1)},v_{s_1 + 1}^{(1,0)}, \dots, v_{N_1}^{(1,0)}, v_{s_2 + 1}^{(0,1)}, \dots, v_{N_2}^{(0,1)}).
$$
We also define
$$
Z_{(N_1,N_2)}^{(s_1,N_2)} : = (x_{s_1+1}^{(1,0)}, \dots, x_{N_1}^{(1,0)}, v_{s_1+1}^{(1,0)}, \dots, v_{N_1}^{(1,0)}) , \quad \quad  Z_{(N_1,N_2)}^{(N_1,s_2)} : = (x_{s_2+1}^{(0,1)}, \dots, x_{N_2}^{(0,1)}, v_{s_2+1}^{(0,1)}, \dots, v_{N_2}^{(0,1)}).
$$
For $(s_1,s_2) \neq (N_1,N_2)$, we define the $(s_1, s_2)$ mixed marginal of $f_{(N_1, N_2)}$ to be the function
\begin{equation}\label{marginal def}
f^{(s_1, s_2)}_{(N_1,N_2)}(t, Z_{(s_1, s_2)}) : = \int_{\mathcal{D}( Z_{(s_1, s_2)})} f_{(N_1, N_2)}(t,Z_{(N_1,N_2)})d Z^{(s_1, s_2)}_{(N_1,N_2)}
\end{equation}
where the integral is taken over the set
\begin{equation}
\mathcal{D}( Z_{(s_1, s_2)}) =  \Big\{ Z^{(s_1, s_2)}_{(N_1,N_2)}  :  Z_{(N_1, N_2)} \in \mathcal{D}_{(\epsilon_1,\epsilon_2)}^{(N_1,N_2)} \Big\}.
\end{equation}
Additionally, we will define $f^{(N_1,N_2)}_{(N_1,N_2)} :  = f_{(N_1,N_2)}$.
\end{definition}
Using Liouville's equation \eqref{liouville}, the boundary condition \eqref{boundary-condition}, and the symmetry condition \eqref{identical_particles} we may formally derive a relation between these mixed marginals. \\
\\
Let us assume that $f_{(N_1, N_2)} \in \mathcal{C}^\infty_c([0,T]\times\mathcal{D})$ and calculate by definition \eqref{marginal def} that
$$
\partial_t f^{(s_1, s_2)}_{(N_1,N_2)}(t,Z_{(s_1,s_2)}) =  \int_{\mathcal{D}( Z_{(s_1, s_2)})} \partial_t f_{(N_1, N_2)}(t,Z_{(N_1,N_2)}) d Z^{(s_1,s_2)}_{(N_1,N_2)}.
$$
Now fix $\alpha \in \mathscr{T}$ and $i \in \{1, \dots, s_\alpha\}$ and compute, using \eqref{marginal def} again, that
\begin{align*}
v_i^\alpha \cdot \nabla_{x_i^\alpha} f^{(s_1, s_2)}_{(N_1,N_2)}(t,Z_{(s_1, s_2)})  &= \int_{\mathcal{D}( Z_{(s_1, s_2)})} v_i^\alpha \cdot \nabla_{x_i^\alpha}f_{(N_1, N_2)}(t,Z_{(N_1,N_2)})d Z^{(s_1,s_2)}_{(N_1,N_2)} \\
&+  \int_{\partial\mathcal{D}( Z_{(s_1, s_2)};i,\alpha)} v_i^\alpha \cdot n f_{(N_1, N_2)}(t,Z_{(N_1,N_2)}) dS
\end{align*}
where $n$ is the outward normal vector to the surface
\begin{equation}\label{partial D}
\partial\mathcal{D}( Z_{(s_1, s_2)};i,\alpha): = \bigcup_{\beta} \bigcup_{j=s_\beta+1}^{N_\beta}\left\{ Z^{(s_1,s_2)}_{(N_1,N_2)}\in \mathcal{D}(Z_{(s_1, s_2)}) : 
 | x_i^\alpha - x_j^\beta| = \epsilon_{(\alpha,\beta)}\right\}.
\end{equation}
Summing over all $i, \alpha$ and using \eqref{liouville} we obtain
\begin{equation}
\begin{aligned}
 \left(\partial_t + \sum_{\alpha} \sum_{i=1}^{s_\alpha} v_i^{\alpha} \cdot \nabla_{x_i^\alpha}\right) f^{(s_1, s_2)}_{(N_1,N_2)}(t,Z_{(s_1,s_2)}) &= - \sum_{\alpha} \sum_{i=s_\alpha+1}^{N_\alpha}\int_{\mathcal{D}( Z_{(s_1, s_2)})} v_i^\alpha \cdot \nabla_{x_i^\alpha} f_{(N_1, N_2)}(t,Z_{(N_1,N_2)})d Z^{(s_1,s_2)}_{(N_1,N_2)} \label{int by parts}\\
& +\sum_{\alpha} \sum_{i=1}^{s_\alpha} \int_{\partial\mathcal{D}( Z_{(s_1, s_2)};i,\alpha)} v_i^\alpha \cdot n  f_{(N_1, N_2)}(t,Z_{(N_1,N_2)}) dS \nonumber
\end{aligned}
\end{equation}
Now, integrating by parts the first term on the right results in only the boundary terms
$$
\int_{\mathcal{D}( Z_{(s_1, s_2)})} v_i^\alpha \cdot \nabla_{x_i^\alpha}f_{(N_1, N_2)}(t,Z_{(N_1,N_2)})d Z^{(s_1,s_2)}_{(N_1,N_2)} =  \int_{\partial^* \mathcal{D}( Z_{(s_1, s_2)};i,\alpha )} v_i^\alpha \cdot n f_{(N_1, N_2)}(t,Z_{(N_1,N_2)})d S(x_i^\alpha) dv_{i}^\alpha dZ^{(s_1,s_2),i,\alpha}_{(N_1,N_2)}.
$$
where $ Z^{(s_1,s_2),i,\alpha}_{(N_1,N_2)}$ is the $ Z^{(s_1,s_2)}_{(N_1,N_2)}$ variable without the $x_i^\alpha,v_i^\alpha$ components and 
\begin{equation}\label{partial* D}
\partial^* \mathcal{D}( Z_{(s_1, s_2)};i,\alpha ) : = \bigcup_{\beta} \bigcup_{j=1}^{s_\beta} \left\{ Z^{(s_1,s_2)}_{(N_1,N_2)} \in \mathcal{D}( Z_{(s_1, s_2)}) : | x_j^\beta - x_i^\alpha| = \epsilon_{(\alpha,\beta)}\right\}.
\end{equation}
Combining the above expressions together\footnote{Note that the boundary sets \eqref{partial D} and \eqref{partial* D} each consists of a union of sets which are pairwise disjoint except for a set of surface measure zero. See Remark \ref{boundary-measure-rmk}.}, we obtain 
\begin{align}
& \left(\partial_t + \sum_{\alpha} \sum_{i=1}^{s_\alpha} v_i^{\alpha} \cdot \nabla_{x_i^\alpha}\right) f^{(s_1, s_2)}_{(N_1,N_2)}(t,Z_{(s_1,s_2)})= \nonumber \\
&- \sum_{\alpha} \sum_{i=s_\alpha+1}^{N_\alpha}\sum_{\beta} \sum_{j= 1}^{s_\beta} \int_{\mathcal{D}( Z_{(s_1, s_2)+\alpha})}\int_{\R^d} \int_{\partial B_{\epsilon_{(\alpha,\beta)}}(x_j^\beta)} v_i^\alpha \cdot n f_{(N_1, N_2)}(t,Z_{(N_1,N_2)})dS(x_i^\alpha) dv_i^\alpha d Z^{(s_1,s_2),i,\alpha}_{(N_1,N_2)}\label{int term 1}\\
& +\sum_{\alpha} \sum_{i=1}^{s_\alpha} \sum_{\beta} \sum_{j=s_\beta+1}^{N_\beta}\int_{\mathcal{D}( Z_{(s_1, s_2)+\beta})} \int_{\R^d}\int_{\partial B_{\epsilon_{(\alpha,\beta)}}(x_i^\alpha)} v_i^\alpha \cdot n  f_{(N_1, N_2)}(t,Z_{(N_1,N_2)}) dS(x_j^\beta) dv_j^\beta d Z^{(s_1,s_2),j,\beta}_{(N_1,N_2)} \label{int term 2}
\end{align}
where here we are taking $n$ to be the inward pointing normal. Relabeling \eqref{int term 1} by reversing the roles of $(j,\beta)$ and $(i,\alpha)$, we can combine \eqref{int term 1} and \eqref{int term 2} to produce
\begin{align*}
& \left(\partial_t + \sum_{\alpha} \sum_{i=1}^{s_\alpha} v_i^{\alpha} \cdot \nabla_{x_i^\alpha}\right) f^{(s_1, s_2)}_{(N_1,N_2)}(t,Z_{(s_1,s_2)})= \nonumber \\
& \sum_{\alpha} \sum_{i=1}^{s_\alpha} \sum_{\beta} \sum_{j=s_\beta+1}^{N_\beta}\int_{\mathcal{D}( Z_{(s_1, s_2)+\beta})} \int_{\R^d}\int_{\partial B_{\epsilon_{(\alpha,\beta)}}(x_i^\alpha)} (-v_j^\beta + v_i^\alpha) \cdot n  f_{(N_1, N_2)}(t,Z_{(N_1,N_2)}) dS(x_j^\beta) dv_j^\beta d Z^{(s_1,s_2),j,\beta}_{(N_1,N_2)}
\end{align*}
By our symmetries assumption \eqref{identical_particles} and recalling definition \eqref{marginal def} we may simplify the each of the above integrals into
\begin{align*}
 &\int_{\mathcal{D}( Z_{(s_1, s_2)+\beta})}\int_{\R^d}\int_{\partial B_{\epsilon_{(\alpha,\beta)}}(x_i^\alpha)} (-v_j^\beta + v_i^\alpha) \cdot n  f_{(N_1, N_2)}(t,Z_{(N_1,N_2)}) dS(x_j^\beta) dv_j^\beta d Z^{(s_1,s_2),j,\beta}_{(N_1,N_2)}= \\
 &  \int_{\mathcal{D}( Z_{(s_1, s_2)+\beta})}\int_{\R^d}\int_{\partial B_{\epsilon_{(\alpha,\beta)}}(x_i^\alpha)} (-v_{s_\beta+1}^\beta+ v_i^\alpha) \cdot n  f_{(N_1, N_2)}(t,Z_{(N_1,N_2)}) dS(x_{s_\beta+1}^\beta) dv_{s_\beta+1}^\beta d Z^{(s_1,s_2)+\beta}_{(N_1,N_2)}=\\
 & \int_{\R^d}\int_{\partial B_{\epsilon_{(\alpha,\beta)}}(x_i^\alpha)} (-v_{s_\beta+1}^\beta + v_i^\alpha) \cdot n  f^{(s_1 , s_2) +\beta }_{(N_1,N_2)}(t,Z_{(s_1 , s_2) +\beta}) dS(x_{s_\beta+1}^\beta) dv_{s_\beta+1}^\beta
\end{align*}
For each of these integrals, change variables so that we integrate over $\mathbb{S}^{d-1}$ instead of $\partial B_{\epsilon_{(\alpha,\beta)}}(x_i^\alpha)$ to obtain
\begin{align}
\epsilon_{(\alpha,\beta)}^{d-1}\int_{\R^d}\int_{\mathbb{S}^{d-1}} (v_{s_\beta+1}^\beta - v_i^\alpha) \cdot \theta  f^{(s_1 , s_2) +\beta}_{(N_1,N_2)}(t,Z_{(s_1, s_2)+\beta,\epsilon}^{i,\alpha,+}) d\theta dv_{s_\beta+1}^\beta. \label{post change of variables}
\end{align}
Here, we are using the notation that $Z_{(s_1, s_2)+\beta,\epsilon}^{i,\alpha,+} \in \mathcal{D}^{(s_1, s_2)+\beta}_{(\epsilon_1, \epsilon_2)}$ is the vector 
\begin{equation}\label{Z i a +}
Z_{(s_1, s_2)+\beta,\epsilon}^{i,\alpha,+} : =
\begin{cases}
\left(X_{s_1}^{(1,0)}, x_i^\alpha + \epsilon_{(\alpha,\beta)}\theta, X_{s_2}^{(0,1)}, V_{s_1}^{(1,0)}, v_{s_{1} + 1}^{(1,0)}, V_{s_2}^{(0,1)} \right), & \beta = (1,0) \\
\left(X_{s_1}^{(1,0)}, X_{s_2}^{(0,1)},x_i^\alpha + \epsilon_{(\alpha,\beta)}\theta, V_{s_1}^{(1,0)}, V_{s_2}^{(0,1)},v_{s_{2} + 1}^{(0,1)} \right), & \beta = (0,1)
\end{cases}.
\end{equation}
Next, break \eqref{post change of variables} into parts $S_+:=\{\theta \in \mathbb{S}^{d-1}: (v_{s_\beta+1}^\beta -v_i^\alpha)\cdot \theta> 0 \}$ and $S_-:=\{\theta \in \mathbb{S}^{d-1}: (v_{s_\beta+1}^\beta -v_i^\alpha)\cdot \theta< 0 \}$. On $S_-$, we use the change of variables $\theta \mapsto -\theta$ and on $S_+$ we use the boundary condition \eqref{boundary-condition} to obtain 
\begin{align}
& \left(\partial_t + \sum_{\alpha} \sum_{i=1}^{s_\alpha} v_i^{\alpha} \cdot \nabla_{x_i^\alpha}\right) f^{(s_1, s_2)}_{(N_1,N_2)}(t,Z_{(s_1,s_2)})= \sum_{\beta,\alpha} (N_\beta-s_\beta) \epsilon_{(\alpha,\beta)}^{d-1} \sum_{i=1}^{s_\alpha} \\
&\int_{\R^d}\int_{\mathbb{S}^{d-1}} [(v_{s_\beta+1}^\beta - v_i^\alpha) \cdot \theta]_+\left( f^{(s_1 , s_2) +\beta}_{(N_1,N_2)}(t,Z_{(s_1, s_2)+\beta,\epsilon}^{i,\alpha,*} )- f^{(s_1 , s_2) +\beta}_{(N_1,N_2)}(t,Z_{(s_1, s_2)+\beta,\epsilon}^{i,\alpha})\right) d\theta dv_{s_\beta+1}^\beta
\end{align}
where we are using the notations 
\begin{equation}\label{Z i a *}
Z_{(s_1, s_2)+\beta,\epsilon}^{i,\alpha,*} := (Z_{(s_1, s_2)+\beta,\epsilon}^{i,\alpha,+})^*,
\end{equation}
\begin{equation}\label{Z i a -}
Z_{(s_1, s_2)+\beta,\epsilon}^{i,\alpha} : =
\begin{cases}
\left(X_{s_1}^{(1,0)}, x_i^\alpha - \epsilon_{(\alpha,\beta)}\theta, X_{s_2}^{(0,1)}, V_{s_1}^{(1,0)}, v_{s_{1} + 1}^{(1,0)}, V_{s_2}^{(0,1)} \right), & \beta = (1,0) \\
\left(X_{s_1}^{(1,0)}, X_{s_2}^{(0,1)},x_i^\alpha - \epsilon_{(\alpha,\beta)}\theta, V_{s_1}^{(1,0)}, V_{s_2}^{(0,1)},v_{s_{2} + 1}^{(0,1)} \right), & \beta = (0,1)
\end{cases}.
\end{equation}
From these calculations, we are lead to the following definition. 
\begin{definition} \label{bbgky-collision-operator}
For each $\alpha,\beta \in \mathscr{T}$ and $s_1, s_2 \in \N$ we define the collision operator 
$$
\mathcal{C}_{(s_1, s_2),(s_1, s_2)+ \beta}^\alpha := \mathcal{C}_{(s_1, s_2),(s_1, s_2)+ \beta}^{(N_1, N_2), \alpha }: \mathcal{C}^\infty_c\left(\mathcal{D}^{(s_1, s_2) +\beta}_{(\epsilon_1,\epsilon_2)}\right) \rightarrow \mathcal{C}^\infty_c \left( \mathcal{D}^{(s_1, s_2) }_{(\epsilon_1,\epsilon_2)} \right)
$$
 by the expression 
 \begin{align*}
& \mathcal{C}_{(s_1, s_2),(s_1, s_2)+ \beta}^\alpha f^{(s_1 ,s_2) + \beta}_{(N_1,N_2)}(Z_{(s_1, s_2)}) : =(N_\beta-s_\beta) \epsilon_{(\alpha,\beta)}^{d-1} \sum_{i=1}^{s_\alpha} \int_{\R^d}\int_{\mathbb{S}^{d-1}} [(v_{s_\beta+1}^\beta - v_i^\alpha) \cdot \theta]_+ \cdot\\
& \qquad \qquad \left( f^{(s_1 , s_2) +\beta}_{(N_1,N_2)}(Z_{(s_1, s_2)+\beta,\epsilon}^{i,\alpha,*} )- f^{(s_1 , s_2) +\beta}_{(N_1,N_2)}(Z_{(s_1, s_2)+\beta,\epsilon}^{i,\alpha})\right) d\theta dv_{s_\beta+1}^\beta
\end{align*}
\end{definition}
The collisional operator $\mathcal{C}_{(s_1,s_2),(s_1,s_2)+\beta}^{\alpha}$ can be viewed as counting the effects of colliding a new type $\beta$ particle with existing particles of type $\alpha$. As with the Boltzmann equation, the above equations represent the effects of a collision between two particles and naturally split into \emph{gain} and \emph{loss} terms. \\

 Summarizing the discussion above, if $f_{(N_1,N_2)} \in \mathcal{C}_c^\infty ( [0,T] \times \mathcal{D})$ satisfies the Liouville's equation \eqref{liouville}, the boundary condition \eqref{boundary-condition}, and the identical particles assumption \eqref{identical_particles}, then we have for each $\alpha\in \mathscr{T}$ and $s_\alpha \in \{1, \dots, N_\alpha\}$ that the mixed marginals satisfy
\begin{align}\label{BBGKY-strong}
& \left(\partial_t + \sum_{\alpha} \sum_{i=1}^{s_\alpha} v_i^{\alpha} \cdot \nabla_{x_i^\alpha}\right) f^{(s_1, s_2)}_{(N_1,N_2)}(t,Z_{(s_1,s_2)})= \sum_{\beta,\alpha\in \mathscr{T}} \mathcal{C}_{(s_1, s_2),(s_1, s_2)+ \beta}^\alpha f^{(s_1 ,s_2) + \beta}_{(N_1,N_2)}(t,Z_{(s_1, s_2)})
\end{align}
We call the set of $N_1N_2$ coupled equations in \eqref{BBGKY-strong} the BBGKY hierarchy. Solutions to this hierarchy will be our main object of study. We will give the function spaces and definition of solutions to the BBGKY hierarchy in Section \ref{Local Well-Posedness}.

\subsection{Scalings and the Boltzmann hierarchy}
Recall that our goal is the asymptotic behavior of the system i.e.  $N_1, N_2 \rightarrow \infty$ and $\epsilon_1, \epsilon_2 \rightarrow 0$. The only possible scaling to make this feasible is dictated by Definition \ref{bbgky-collision-operator}, and so we assume that $N_\beta  \epsilon_{(\alpha,\beta)}^{d-1} \approx 1$, where 
\begin{equation} \label{boltz-grad}
N_{\beta}  \epsilon_{(\alpha,\beta)}^{d-1} = \begin{cases}
N_1 \epsilon_1^{d-1} , & \alpha = \beta = (1,0)\\
N_2 \epsilon_2^{d-1},& \alpha = \beta = (0,1) \\
N_1 \left(\frac{\epsilon_1 + \epsilon_2}{2}\right)^{d-1}, & \alpha = (0,1), \beta = (1,0) \\
 N_2\left(\frac{\epsilon_1 + \epsilon_2}{2}\right)^{d-1}, & \alpha= (1,0), \beta = (0,1)
\end{cases}
\end{equation}
Note that the terms where $\alpha \neq \beta$ in \eqref{boltz-grad} imply $N_1 \approx N_2$. This fact, combined with the terms in \eqref{boltz-grad} where $\alpha = \beta$, implies $\epsilon_1 \approx  \epsilon_2$. Hence, we will explicitly require the following scalings 
\begin{equation}\label{mixed-boltz-grad,1}
N_1 \epsilon_1^{d-1} \equiv c_1, \qquad N_2 \epsilon_2^{d-1} \equiv c_2,  \qquad \epsilon_1 \equiv b \epsilon_2.
\end{equation}
A simple calculation then yields
\begin{equation} \label{mixed-boltz-grad,2}
N_1 \Big( \frac{\epsilon_1 +\epsilon_2}{2}\Big)^{d-1} = c_1 \Big( \frac{1 + b^{-1}}{2} \Big)^{d-1}, \qquad N_2 \Big( \frac{\epsilon_1 +\epsilon_2}{2}\Big)^{d-1} = c_2 \Big( \frac{1 + b}{2} \Big)^{d-1}.
\end{equation}
\begin{equation}\label{N1 N2 scaling}
N_1 = \left(\frac{c_1}{c_2} b^{1-d} \right) N_2
\end{equation}
Here, the constants $c_1, c_2>0$ describe the mean free path density of the different types of gases, and $b$ is the ratio of the type $(1,0)$ particle diameter to the type $(0,1)$ particle diameter. For the sake of simplicity, we assume that the constant $c_1 b^{1-d}/c_2$ is a rational number. While this condition is not strictly needed, it simplifies the relation $N_1$ has with $N_2$ in \eqref{N1 N2 scaling}. Assuming that our scalings \eqref{mixed-boltz-grad,1} hold, we can now obtain a formal limit of the above collisional operators by taking $N_1, N_2 \rightarrow \infty$ and $\epsilon_1, \epsilon_2 \rightarrow 0$.

\begin{definition}\label{infinite-collision-operators} For $s_1,s_2 \in \N_+$ define for each $\alpha,\beta \in \mathscr{T}$ the operators 
\begin{equation}
\mathscr{C}_{(s_1, s_2),(s_1, s_2) + \beta}^\alpha : \mathcal{C}_c^\infty\left(\R^{d(s_1 + s_2+1)} \right) \rightarrow \mathcal{C}_c^\infty\left(\R^{d(s_1 + s_2)} \right)
\end{equation}
given by 
\begin{equation}
\begin{aligned}
&\mathscr{C}_{(s_1, s_2),(s_1, s_2) + \beta}^\alpha f^{(s_1, s_2) + \beta}(Z_{(s_1, s_2)}) : = \\
&\quad A_{\beta}^\alpha \sum_{i=1}^{s_\alpha} \int_{\R^d}\int_{\mathbb{S}^{d-1}} [(v_{s_\beta+1}^\beta - v_i^\alpha) \cdot \theta]_+\left( f^{(s_1 , s_2) +\beta}(Z_{(s_1, s_2)+\beta}^{i,\alpha,*} )- f^{(s_1 , s_2) +\beta}(Z_{(s_1, s_2)+\beta}^{i,\alpha})\right) d\theta dv_{s_\beta+1}^\beta. \label{boltz alpha beta kernel}
\end{aligned}
\end{equation}
Here, we are using the notation that 
\begin{equation}
Z_{(s_1,s_2)+\beta}^{i,\alpha}: =
\begin{cases}
\left(X_{s_1}^{(1,0)}, x_i^\alpha , X_{s_2}^{(0,1)}, V_{s_1}^{(1,0)}, v_{s_{1} + 1}^{(1,0)}, V_{s_2}^{(0,1)} \right), & \beta = (1,0) \\
\left(X_{s_1}^{(1,0)}, X_{s_2}^{(0,1)},x_i^\alpha, V_{s_1}^{(1,0)}, V_{s_2}^{(0,1)},v_{s_{2} + 1}^{(0,1)} \right), & \beta = (0,1)
\end{cases}
\end{equation}
and $Z_{(s_1,s_2)+ \beta}^{i,\alpha, *}$ is the vector $Z_{(s_1,s_2)+\beta}^{i,\alpha}$ whose components $v_i^\alpha, v_{s_\beta+1}^\beta$ are replaced with $(v_i^\alpha)^*, (v_{s_\beta+1}^\beta)^*$ given by the collisional laws
\begin{equation}\label{alpha beta boltz 1}
(v_i^\alpha)^* = v_i^\alpha - \frac{2M_\beta}{M_\alpha + M_\beta} \left( (v_i^\alpha - v_{s_\beta+1}^\beta) \cdot \theta\right)\theta,\quad 
(v_{s_\beta+1}^\beta)^* = v_{s_\beta+1}^\beta + \frac{2M_\alpha}{M_\alpha + M_\beta} \left( (v_i^\alpha - v_{s_\beta+1}^\beta) \cdot \theta \right) \theta.
\end{equation}
The constants $A_{\beta}^\alpha$ in \eqref{boltz alpha beta kernel} are given by 
\begin{equation}\label{A a b def}
A_{\beta}^\alpha = \begin{cases}
c_1, & \alpha = \beta = (1,0)\\
c_2, &\alpha = \beta = (0,1)\\
c_1 \Big( \frac{1 + b^{-1}}{2} \Big)^{d-1}, & \alpha = (0,1), \beta = (1,0) \\
c_2 \Big( \frac{1 + b}{2} \Big)^{d-1}, & \alpha = (1,0), \beta= (0,1)
\end{cases}.
\end{equation}
\end{definition}
The associated limiting differential equation of \eqref{BBGKY-strong} is given by
\begin{equation}\label{boltz-strong}
\Big(\partial_t  + \sum_{\alpha \in \mathscr{T}} \sum_{k=1}^{s_\alpha} v_k^\alpha \cdot \nabla_{x_k^\alpha}\Big)f^{(s_1, s_2)} =\sum_{\alpha,\beta \in \mathscr{T}} \mathscr{C}_{(s_1, s_2),(s_1, s_2) + \beta}^\alpha f^{(s_1, s_2) + \beta}
\end{equation}
We call the infinite set of coupled equations \eqref{boltz-strong} the Boltzmann hierarchy. The exact definition the functional spaces on which we consider solutions is given in Section \ref{section LWP-Boltzmann hierarchy}. Formal solutions to this hierarchy which are tensors of the form
\begin{equation}\label{tensor ansatz}
f^{(s_1, s_2)}(t, Z_{(s_1, s_2)}) = \prod_{\gamma\in \mathscr{T}} \prod_{k=1}^{s_\gamma} f_{\gamma} (t,x_k^\gamma,v_k^\gamma)
\end{equation}
for some functions $f_{(1,0)}, f_{(0,1)}$ are intimately related to the Boltzmann equation for mixtures. Namely, \eqref{tensor ansatz} solves the Boltzmann hierarchy if for each $\alpha \in \mathscr{T}$
\begin{equation}\label{alpha beta boltz pde}
 \left[ \partial_t + v_{k}^\alpha \cdot \nabla_{x_k^\alpha} \right]f_\alpha( t, x_k^\alpha, v_k^\alpha)=\sum_{\beta \in \mathscr{T}} A_\beta^\alpha Q_{\beta}^\alpha (f_\alpha,f_\beta)(t, x_k^\alpha,v_k^\alpha).
\end{equation}
The collision operators $Q_{\beta}^\alpha$ above are given by the following definition.
\begin{definition} \label{boltz-kernels} For $G,H: [0,\infty) \times \R^{d} \times \R^d\rightarrow \R$ and $\alpha,\beta\in \mathscr{T}$, we define the bilinear forms
\begin{equation*}
\begin{aligned}
    &Q_\beta^\alpha(G,H)(t,x^\alpha,v^\alpha) : =\\
    &\quad \int_{\R^d} \int_{\mathbb{S}^{d-1}} ((v^\beta - v^\alpha)\cdot \theta)_+ \left[G(t,x^\alpha,(v^\alpha)^*)H(t,x^\alpha, (v^\beta)^*) -G(t,x^\alpha,v^\alpha)H(t,x^\alpha, v^\beta)  \right] d\theta dv^\beta 
    \end{aligned}
\end{equation*}
where $(v^\alpha)^*,(v^\beta)^*$ are given by the collisional laws
\begin{equation*}
(v^\alpha)^* = v^\alpha - \frac{2M_\beta}{M_\alpha + M_\beta} \left( (v_i^\alpha - v^\beta) \cdot \theta\right)\theta, \qquad
(v_{}^\beta)^* = v_{}^\beta + \frac{2M_\alpha}{M_\alpha + M_\beta} \left( (v_i^\alpha - v_{}^\beta) \cdot \theta \right) \theta.
\end{equation*}
\end{definition}
These operators $Q_{\beta}^\alpha$ are related to the operators $Q_{i,j}$ in the Boltzmann system for mixtures \eqref{PDE-intro} by
\begin{equation}
    Q^\alpha_\beta = \begin{cases}
    Q_{1,1}, & \alpha = \beta = (1,0)\\
    Q_{2,2}, & \alpha = \beta = (0,1)\\
    Q_{1,2}, & \alpha = (1,0), \beta = (0,1) \\
    Q_{2,1}, & \alpha = (0,1), \beta = (1,0).
    \end{cases}
\end{equation}

\section{Dynamics of Mixed Particles} \label{Dynamics of Mixed Particles}
In this section we rigorously show that a measure preserving global in time flow for a.e. initial configuration can be defined for the mixture of two hard sphere gases such that the first gas consists of $N_1$ identical particles of mass $M_1$ and diameter $\epsilon_1$ and the second gas consists of $N_2$ identical particles of mass $M_2$ and diameter $\epsilon_2$. We assume that both types of particles perform rectilinear motion until they run into a binary collision with a particle of either type. Depending on the type of particles colliding, velocities instantaneously transform according to \eqref{i alpha}, \eqref{j beta}. However, since the exchange of velocities is not smooth in time, it is not obvious that a global dynamics can be defined. In particular,  the system might run into pathological trajectories (multiple collisions of particles, grazing collisions, infinitely many collisions in a finite time).

In the case of a single gas of identical hard spheres pathologies might arise as well and existence of a global flow was established by Alexander \cite{Ale75}. Inspired by the ideas of  \cite{Ale75}, \cite{Ale76}, Ampatzoglou and Pavlovi\'c \cite{AP19b} constructed a global flow for a system of particles performing ternary interactions and later \cite{AP19a} for particles performing both binary and ternary interactions.

Our case does not directly follow from \cite{Ale75} because our mixture consists of gases of different masses, and therefore we have to construct the dynamics from the very beginning. To do this, we adapt ideas from \cite{AP19a} in the case of mixture of gases. The crucial lemma required to pass from the local flow to the global flow is to note that once a collision occurs, subsequent collisions cannot involve the same particles. This observation allows us to remove a set of measure zero leading to pathological configurations while having well defined trajectories on the complement.
\subsection{Mixed Particle Notation}\label{mixed-particle-notation}

Recall our phase space $\mathcal{D} = \mathcal{D}^{(N_1,N_2)}_{(\epsilon_1,\epsilon_2)}$ is given by \eqref{phase-space} and our set of interacting pairs $\mathcal{I}^{(\alpha,\beta)}$ is given by \eqref{interacting pairs}. We denote the interior of $\mathcal{D}$ by $\mathring{\mathcal{D}}$, and we write its boundary as 
\begin{equation}\label{phase-space-boundary} 
\partial \mathcal{D}  =  \bigcup_{\alpha \leq \beta } \bigcup_{(i,j) \in \mathcal{I}^{(\alpha,\beta)} }\Sigma_{(i,j)}^{(\alpha, \beta)}, \qquad 
 \Sigma_{(i,j)}^{(\alpha,\beta)} : = \{ Z : \, \, |x_i^\alpha - x_j^\beta| = \epsilon_{(\alpha,\beta)} \}.
\end{equation}
These parts $\Sigma_{(i,j)}^{(\alpha,\beta)}$ of the boundary are not disjoint, and the intersection of two non-identical parts form sets where three or more particles are colliding. The subset of the boundary where the only particles in collision are exactly the $i$th $\alpha$-particle and $j$th $\beta$-particle will be denoted 
\begin{equation}\label{sigma-sc}
 \Sigma_{(i,j),sc}^{(\alpha,\beta)} : = \{ Z \in \partial \mathcal{D} : \, \, Z \in  \Sigma_{(i,j)}^{(\alpha,\beta)} \text{ for a unique $4$-tuple $(\alpha,\beta,i,j)$ with $\alpha \leq \beta$ and $(i,j) \in \mathcal{I}^{(\alpha,\beta)}$}\} .
\end{equation}
The set where exactly two particles are colliding is called the \textit{simple collisional subset of the boundary}, and will be given by the disjoint union
\begin{equation}\label{phase-space-sc}
\partial_{sc} \mathcal{D}: = \bigcup_{\alpha \leq \beta} \bigcup_{(i,j) \in \mathcal{I}^{(\alpha,\beta)}}\Sigma_{(i,j),sc}^{(\alpha,\beta)}.
\end{equation} 
The subset of the boundary not in simple collision is said to be in \textit{multiple collisions}, and is denoted
\begin{equation}\label{phase-space-mc}
\partial_{mc} \mathcal{D} : = \partial \mathcal{D} \setminus \partial_{sc} \mathcal{D}.
\end{equation}
In the following definition, we further classify $\Sigma_{(i,j),sc}^{(\alpha,\beta)}$ into parts where we have \textit{grazing} or \textit{non-grazing} collisions. 
\begin{definition}\label{grazing-cond}
Let $\alpha \leq \beta \in \mathscr{T}$. For $(i,j) \in \mathcal{I}^{(\alpha,\beta)}$ and $Z \in \Sigma_{(i,j),sc}^{(\alpha,\beta)}$ we define the following collision types:
\begin{equation*}
\begin{cases}
\text{Pre-collisional,} & (x_i^\alpha - x_j^\beta)\cdot( v_i^\alpha - v_j^\beta) < 0 \\
\text{Post-collisional, } & (x_i^\alpha - x_j^\beta)\cdot( v_i^\alpha - v_j^\beta) > 0 \\
\text{Grazing, } & (x_i^\alpha - x_j^\beta)\cdot( v_i^\alpha - v_j^\beta) = 0. \\
\end{cases}
\end{equation*}
This leads to the definition of the simple collisional grazing and non-grazing sets
\begin{equation*}
\Sigma_{(i,j),sc, g}^{(\alpha,\beta)} : = \{ Z \in \Sigma_{(i,j),sc}^{(\alpha,\beta)} : \, \, Z \text{ is grazing } \}, \qquad \qquad 
\Sigma_{(i,j),sc, ng}^{(\alpha,\beta)} : = \Sigma_{(i,j),sc}^{(\alpha,\beta)}\setminus \Sigma_{(i,j),sc,g}^{(\alpha,\beta)}
\end{equation*}
Moreover, we have the decomposition 
\begin{equation*}
    \partial_{sc} \mathcal{D}  = \partial_{sc,ng} \mathcal{D}\cup \partial_{sc,g} \mathcal{D}
\end{equation*}
where we define the simple collisional grazing and non-grazing sets
\begin{equation*}
\partial_{sc,ng} \mathcal{D} : = \{ Z \in \partial_{sc} \mathcal{D} : \, \,  Z \text{ is not grazing }\},\qquad \qquad
\partial_{sc,g} \mathcal{D} : = \{ Z \in \partial_{sc} \mathcal{D} : \, \,  Z \text{ is grazing }\}.
\end{equation*}
\end{definition}
On these non-grazing sets $\Sigma_{(i,j),sc, ng}^{(\alpha,\beta)}$ we have a natural impact operator.
\begin{definition} \label{T-collision}
Let $\alpha\leq \beta \in \mathscr{T}$. Then for $(i,j) \in \mathcal{I}^{(\alpha,\beta)}$, define $T_{(i,j)}^{(\alpha,\beta)} : \Sigma_{(i,j),sc, ng}^{(\alpha,\beta)}  \rightarrow \Sigma_{(i,j),sc, ng}^{(\alpha,\beta)} $ by 
\begin{multline}
T_{(i,j)}^{(\alpha,\beta)}(Z) := \\
\begin{cases}
 \left( X_{N_{1}}^{(1,0)}, X_{N_{2}}^{(0,1)}, V_{N_{1}}^{(1,0)}, v_1^{(0,1)}, \dots (v_{i}^{(0,1)})^*, \dots, (v_j^{(0,1)})^*, \dots, v_{N_{2}}\right), & \alpha = \beta=(0,1)\\
 \left ( X_{N_{1}}^{(1,0)}, X_{N_{2}}^{(0,1)}, v_{1}^{(1,0)}, \dots, (v_i^{(1,0)})^*,\dots, v_{N_{1}}^{(1,0)}, v_1^{(0,1)}, \dots (v_{j}^{(0,1)})^*, \dots, v_{N_{2}}\right),& \alpha = (1,0), \beta = (0,1) \\
 \left( X_{N_{1}}^{(1,0)}, X_{N_{2}}^{(0,1)}, v_1^{(1,0)}, \dots, (v_{i}^{(1,0)})^*, \dots, (v_j^{(1,0)})^*, \dots, v_{N_{1}},V_{N_{2}}^{(0,1)}\right),& \alpha =\beta = (1,0)
\end{cases} 
\nonumber
\end{multline}
where the $(v_i^\alpha)^*,(v_j^\beta)^*$ are given by \eqref{i alpha}, \eqref{j beta}. Since the sets $\Sigma_{(i,j),sc,ng}^{(\alpha,\beta)}$ are disjoint for $\alpha \leq \beta$ and $(i,j) \in \mathcal{I}^{(\alpha,\beta)}$ the above operators define an operator $T : \partial_{sc,ng} \mathcal{D} \rightarrow \partial_{sc,ng} \mathcal{D}$. For notational convienience, we will often write 
\begin{equation}\label{z star}
    Z^* : = T(Z)
\end{equation} 
\end{definition}
\begin{definition}
We define the \emph{energy} of the configuration $Z \in \mathcal{D}$ to be 
\begin{equation}\label{energy}
E(Z) : = \sum_{i=1}^{N_1} M_{1}|v_i^{(1,0)}|^2 + \sum_{i=1}^{N_2} M_{2}|v_i^{(0,1)}|^2.
\end{equation}
\end{definition}
We conclude this section with the following remarks. 
\begin{rmk}\label{post-pre-rmk}
It is clear that $T$ is an involution. Moreover, it satisfies that $Z \in \partial_{sc,ng}\mathcal{D}$ is pre-collisional (post-collisional) if and only if $T(Z)$ is post-collisional (pre-collisional). 
\end{rmk}
\begin{rmk}\label{conservation-of-energy}
One can check by a simple computation that the conservation of energy holds under the action of the operator $T$. That is, for every $Z \in \partial_{sc,ng} \mathcal{D}$, we have that $E(T(Z)) = E(Z)$. 
\end{rmk}
\begin{rmk}\label{positions-T-invariant}
The operator $T: \partial_{sc,ng}\mathcal{D} \rightarrow \partial_{sc,ng} \mathcal{D}$ leaves positions invariant. That is $
x_i^\alpha \circ T = x_i^\alpha
$, where here $x_i^\alpha$ is the projection operator given by \eqref{coordinate-projections}.
\end{rmk}

\subsection{Mixed Particle Local Flow}\label{Mixed Particle Local Flow}
In this section, we construct our mixed particle flow up to the time of the first collision. First, let us define the \textit{refined phase space}
\begin{equation}\label{refined-phase-space-def}
\mathcal{D}^* : = \partial_{sc,ng} \mathcal{D} \cup \mathring{ \mathcal{D}}.
\end{equation}
\begin{rmk}\label{boundary-measure-rmk}
The sets $\partial_{mc} \mathcal{D}$ and $\partial_{sc,g} \mathcal{D}$ have zero Hausdorff $2d(N_1 + N_2) -1$ measure. Hence, the set $\mathcal{D}^*$ differs from the full phase space $\mathcal{D}$ only up to a set of surface measure zero.
\end{rmk}

To discuss the propagation of the particles in the direction of their velocities, we define the velocity propagator 
\begin{equation}
P : \mathcal{D} \rightarrow \R^{2d(N_{1} + N_{2})}, \qquad \qquad
P(Z) : = \left( V_{N_{1}}^{(1,0)}, V_{N_{2}}^{(0,1)}, 0, 0\right) 
\end{equation}
This vector $P(Z)$ allows us to keep track of the velocities of our particles and formulate the following lemma compactly.
\begin{lem}\label{local-flow-lem}
Let $N_{1}, N_{2} \geq 2$ and $Z \in \mathcal{D}^*$, where $\mathcal{D}$ is defined in \eqref{refined-phase-space-def} Then, there is a time $\tau^1 = \tau^1_{Z} \in (0, \infty]$ such that the function 
\begin{equation}
Z : [0,\tau^1] \rightarrow \mathcal{D}, \quad 
\text{given by}\quad
Z(t) : = 
\begin{cases}
Z + t P(Z), & \text{ if $Z$ is non or post-collisional } \\
T(Z) +t P\circ T(Z) & \text{ if $Z$ is pre-collisional }
\end{cases}
\end{equation}
with $T$ given by Definition \ref{T-collision} satisfies the following conditions: 
\begin{itemize}
\item $Z(t) \in \mathring{ \mathcal{D}}$ for $t \in (0,\tau^1)$. 
\item If $\tau^1 < \infty$, then $Z(\tau^1) \in\partial\mathcal{D}$. 
\item If $Z \in \Sigma_{(i,j),ng}^{(\alpha,\beta)}$ for some $\alpha \leq \beta$ and $(i,j) \in \mathcal{I}^{(\alpha,\beta)}$ and $\tau^1 < \infty$, then $Z(\tau^1) \notin \Sigma_{(i,j),ng}^{(\alpha,\beta)}$.
\end{itemize}
\end{lem}
\begin{proof}
This is a standard \textit{stopping time} construction. Note that $Z(\cdot)$ given above is a well defined function from $[0, \infty) \rightarrow \R^{2d(N_1 +N_2)}$. In order to keep its range within the domain $\mathcal{D}$, we define 
$$
\tau^1 = \tau^1_{ Z } : = \inf \{ t > 0 / \,  Z(t) \in \partial \mathcal{D}\}.
$$
First note that if $\tau^1_{Z} = \infty$, then $Z(t) \notin \partial \mathcal{D}$ for all times $t > 0$. So we can assume that $\tau^1_Z < \infty$. Since $Z \in \mathcal{D}^*$, it is either an interior point of $\mathcal{D}$ or it is in the simple, non-grazing subset of the boundary. If it is an interior point, by continuity $\tau^1 > 0$, and by the definition of the stopping time we have that the conclusions of the lemma are satisfied.\\

Next, assume that $ Z \in \Sigma_{(i,j),sc,ng}^{(\alpha,\beta)}$ for some $\alpha \leq \beta $ and index $(i,j)\in \mathcal{I}^{(\alpha,\beta)}$. We consider the two cases: where $Z$ is post-collisional and where $Z$ is pre-collisional. For the post-collisional case, we calculate for $t> 0$ 
\begin{align*}
|x_i^{\alpha} + t v_i^{\alpha}- (x_j^{\beta}+ t v_j^\beta)|^2 &= | x_ i^\alpha -x^\beta_j +t(v_i^\alpha - v_j^\beta) |^2 \\
&\geq  |x_i^\alpha -x_j^\beta|^2 +  2 t \langle x_i^\alpha - x_j^\beta , v_i^\alpha - v_j^{\beta}\rangle\\
& > |x_i^\alpha- x_j^\beta|^2 = \epsilon_{(\alpha,\beta)}^2.
\end{align*}
This strict inequality shows that $x_i^\alpha$ does not collide with $x_j^\beta$ for small, positive times $t$. Moreover, since we are assuming a simple collision, no other particles are colliding at time $t=0$ apart from $x_i^\alpha, x_j^\beta$. By using continuity, this fact combined with the strict inequality above shows that $\tau^1 > 0$. By the definition of the stopping time, the first two conditions of the lemma are also immediate. Moreover, the inequality directly above also shows that the last condition of the lemma is true. \\

Now, for the pre-collisional case, apply the above argument to the post-collisional configuration $T(Z)$. This completes the proof in all cases.

\end{proof}

\subsection{Mixed Particle Global Flow} \label{Mixed Particle Global Flow}
To define global flow, we inductively apply the above construction. We will use the notation $\tau^1 = \tau^1_{Z}$ as in the Lemma \ref{local-flow-lem} above, and if $Z(\tau^1) \in \partial_{sc, ng} \mathcal{D}$, then we define 
\begin{equation}
\tau^2 _{Z}: = \tau^1_{ Z(\tau^1)}.
\end{equation}
That is, $\tau^2_Z$ is the stopping time of the process started at $Z(\tau^1)$ which is guaranteed to exist by Lemma \ref{local-flow-lem}. We then can define a process $Z(\cdot) : [0,\tau^2] \rightarrow \mathcal{D}$ given by 
\begin{equation}
\begin{cases}
Z(\cdot) : [0,\tau^1] \rightarrow \mathcal{D} & \text{ on $[0,\tau^1]$ } \\
Z(\cdot) : (\tau^1,\tau^2] \rightarrow \mathcal{D} & \text{ on $(\tau^1, \tau^2]$ }
\end{cases}
\end{equation}
We wish to continue this local construction of the flow $Z(\cdot)$ for arbitrarily large times for initial configurations $Z$ outside a set of measure zero in $\mathcal{D}$. \\

Now, to analyze the measure of the sets in which this inductive construction is well defined, we introduce a time truncation parameter $\delta > 0$ and a velocity truncation parameter $R> 0$ that satisfy 
\begin{equation} 
0< \delta R \ll 1 \ll R.
\end{equation}
Given $\rho >0$, we define
\begin{equation}
B^x_\rho : =  \{ ( X_{N_{1}}^{(1,0)}, X_{N_{2}}^{(0,1)})\in \R^{d(N_{1} + N_{2})} / \, \vert (X_{N_{1}}^{(1,0)},  X_{N_{2}}^{(0,1)}) \vert \leq \rho \}.
\end{equation}
We additionally define 
\begin{equation}
B^v_R : = \{ ( V_{N_{1}}^{(1,0)}, V_{N_{2}}^{(0,1)})\in \R^{d(N_{1} + N_{2})} / \, \vert (V_{N_{1}}^{(1,0)},  V_{N_{2}}^{(0,1)}) \vert \leq R \}.
\end{equation}
Now, define the \textit{truncated, refined phase space} as 
\begin{equation}
\mathcal{D}(\rho,R): = \mathcal{D}^* \cap (B^x_\rho \times B^v_R)
\end{equation}
where $\mathcal{D}^*$ is given by \eqref{refined-phase-space-def}. We now decompose the truncated refined phase space into five parts:
\begin{equation}
I_{free} : = \{ Z \in \mathcal{D}(\rho,R) / \, \tau^1_{Z} > \delta \}
\end{equation}
\begin{equation}
I_{sc, ng}^1 : = \{ Z \in \mathcal{D}(\rho,R) / \, \tau^1_{Z} \leq \delta , \, Z(\tau^1)  \in \partial_{sc, ng} \mathcal{D}, \, \text{ and } \tau^2_{Z} > \delta \} 
\end{equation}
\begin{equation}
I_{sc, ng}^2 : = \{ Z \in \mathcal{D}(\rho,R) / \, \tau^1_{Z} \leq \delta , \, Z(\tau^1) \in  \partial_{sc, ng} \mathcal{D}, \, \text{ and } \tau^2_{Z} \leq  \delta \} 
\end{equation}
\begin{equation}
I_{sc, g}^1 : = \{ Z \in \mathcal{D}(\rho,R) / \, \tau^1_{Z} \leq \delta , \, Z(\tau^1) \in \partial_{sc,g} \mathcal{D}\} 
\end{equation}
\begin{equation}
I_{mc}^1 : = \{ Z \in \mathcal{D}(\rho,R) / \, \tau^1_{Z} \leq \delta , \, Z(\tau^1) \in \partial_{mc} \mathcal{D} \} 
\end{equation}
Clearly, we have that
$$
\mathcal{D}(\rho,R) = I_{free} \cup I_{sc,ng}^1 \cup I_{sc, g}^1 \cup I_{mc}^1 \cup I_{sc,ng}^2.
$$
Here, the \textit{good} sets are $I_{free}$ and $ I_{sc,ng}^1$, for they allow us to continue our inductive construction without issue. The \textit{bad} sets for which we seek measure estimates for are $I_{sc, g}^1, I_{mc}^1$, and $ I_{sc,ng}^2$. To handle the sets $I_{sc, ng}^2, I_{mc}^1$, we first show they have a particular covering by intersections of orthogonal annuli. In particular for $\alpha \leq \beta$ and $(i,j) \in \mathcal{I}^{(\alpha,\beta)}$ we define the annuli to be
\begin{equation}
U_{(i,j)}^{(\alpha,\beta)} : = \{ Z \in B^x_\rho \times B^v_{R}  / \, \epsilon_{(\alpha,\beta)}  \leq |x_i^\alpha- x_j^\beta| \leq  \epsilon_{(\alpha,\beta)}  + 2 C \delta R \},
\end{equation}
where $C= C(M_1, M_2)>1$ is a constant depending only on the masses of the particles. We have the following lemma. 

\begin{lem}\label{covering-lem} We have
\begin{equation}
I_{sc, ng}^2\cup I_{mc}^1 \subset  \bigcup_\mathscr{I} U^{(\alpha,\beta)}_{(i,j) } \cap U^{(\alpha',\beta')}_{(i',j')}
\end{equation}
where the index set $\mathscr{I}$ is defined to be the set of pairs of $4$-tuples $(\alpha,\beta,i,j), (\alpha',\beta',i',j')$ such that $\alpha \leq \beta, \alpha'\leq \beta'$, $(i,j) \in \mathcal{I}^{(\alpha,\beta)}$, $(i',j') \in \mathcal{I}^{(\alpha',\beta')}$, and \begin{equation}
(\alpha, \beta,i,j) \neq (\alpha',\beta',i',j').
\end{equation}
\end{lem}
\begin{proof} We prove that the set $I_{sc,ng}^2$ is contained in the union. The proof for the set $I_{mc}^1$ is carried out similarly. \\

Let $Z(\cdot) : [0,\tau^2] \rightarrow \mathcal{D}$ be our constructed flow and assume that $Z \in I_{sc,ng}^2$. Then by definition we have $\tau^1, \tau^2 \leq \delta$, and 
$$
Z(\tau^1) \in \Sigma_{(i,j),sc,ng }^{(\alpha,\beta) } ,\qquad \text{and} \qquad Z ( \tau^2) \in \Sigma_{(i',j')}^{{(\alpha ',\beta')}}
$$
for some $\alpha\leq \beta$, $\alpha' \leq \beta'$, and indices $(i,j) \in \mathcal{I}^{(\alpha,\beta)}$, $(i',j')\in \mathcal{I}^{(\alpha',\beta')}$. Note that by the last condition of Lemma \ref{local-flow-lem}, we \textit{cannot} have that $(\alpha,\beta) = (\alpha',\beta')$ and $(i,j) = (i',j')$. \\
\\
Recall from equation \eqref{coordinate-projections} the projection operators. We consider two cases. \\
$\bullet$ Assume that $Z$ is non-collisional or post-collisional. Then we can write
\begin{align*}
\epsilon_{(\alpha,\beta)} &\leq |x_i^{\alpha}- x_j^{\beta} |\\
& \leq |x_i^{\alpha}+ \tau^1 v_i^{\alpha}- (x_j^{\beta} + \tau^1 v_j^{\beta})|+ \tau^1 | v_i^{\alpha}- v_j^{\beta}|\\
& = |x_{i}^{\alpha}(Z(\tau^1)) - x_j^{\beta} (Z(\tau^1))| + \tau^1 | v_i^{\alpha} - v_j^{\beta}| \\
& = \epsilon_{(\alpha,\beta)} + \tau^1 | v_i^{\alpha} - v_j^{\beta}|.
\end{align*}
Noting that $|v_i^{\alpha}|, |v_j^{\beta}| \leq R$ and $\tau^1 \leq \delta$, we obtain that $
\epsilon_{\alpha} \leq |x_i^{\alpha} - x_j^{\beta} |\leq \epsilon_{(\alpha,\beta)} + 2 R \delta. $\\
$\bullet$ When $Z$ is pre-collisional, apply the above argument to the pre-collisional $T(Z)$. Note that since $Z \in B_\rho^x \times B_{R}^v$, we have $T( Z) \in B_\rho^x \times B_{CR}^v$ by Remarks \ref{conservation-of-energy}, \ref{positions-T-invariant} and a comparison of $E(Z)$ with the Euclidean norm on $\R^{d(N_1 + N_2)}$. The constant $C>1$ above can explicitly be computed in terms of $M_1, M_2$. We thus obtain $\epsilon_{(\alpha,\beta)} \leq |x_i^{\alpha} (T(Z))- x_j^{\beta} (T(Z)) |\leq \epsilon_{(\alpha,\beta)} + 2CR \delta.$ Since $T$ leaves the positions unaffected, by Remark \ref{positions-T-invariant} we have that $\epsilon_{(\alpha,\beta)} \leq |x_i^{\alpha} - x_j^{\beta}|\leq \epsilon_{(\alpha,\beta)} + 2CR \delta$ in this case also. So in all cases, $Z \in U_{(i,j)}^{(\alpha,\beta)}$. 

Next, let us show that $Z \in U_{(i',j')}^{(\alpha',\beta')}$. We can calculate, by Remark \ref{positions-T-invariant}
\begin{align*}
\epsilon_{(\alpha',\beta')} &\leq | x_{i'}^{\alpha'} - x_{j'}^{\beta'}|  \\
&=  | x_{i'}^{\alpha'}(T(Z )) - x_{j'}^{\beta'}(T(Z ))| \\
& \leq  | x_{i'}^{\alpha'}(T(Z )) + \tau^2 v_{i'}^{\alpha'} (T(Z )) - (x_{j'}^{\beta'}(T(Z ))+ \tau^2 v_{j'}^{\beta'} (T(Z )))| + \tau^2 | v_{i'}^{\alpha'} (T(Z )) - v_{j'}^{\beta'} (T(Z ))|.
\end{align*}
Now, since there are no collisions in the interval $(\tau^1, \tau^2)$, we have
$$
x_{i'}^{\alpha'}(T(Z )) + \tau^2 v_{i'}^{\alpha'} (T(Z )) - (x_{j'}^{\beta'}(T(Z ))+ \tau^2 v_{j'}^{\beta'} (T(Z ))) =  x_{i'}^{\alpha'}(T(Z(\tau^2) ))- x_{j'}^{\beta'}(T(Z(\tau^2)) ).
$$
Hence we can conclude that 
\begin{align*}
\epsilon_{(\alpha',\beta')} &\leq | x_{i'}^{\alpha'}- x_{j'}^{\beta'}| \\
& \leq | x_{i'}^{\alpha'}(T(Z(\tau^2) ))- x_{j'}^{\beta'}(T(Z(\tau^2)) )| +  \tau^2 | v_{i'}^{\alpha'} (T(Z)) - v_{j'}^{\beta'} (T(Z ))| \\
& = \epsilon_{(\alpha',\beta')} +  \tau^2 | v_{i'}^{\alpha'} (T(Z)) - v_{j'}^{\beta'} (T(Z ))| 
\end{align*}
Again, by Remark \ref{conservation-of-energy}, and a comparison of $E(Z)$ with the Euclidean norm on $\R^{d(N_1 + N_2)}$  we obtain that $
|v_{i'}^{\alpha'} (T(Z ))|, |v_{j'}^{\beta'} (T(Z ))| \leq CR.$ Hence, since $\tau^2 \leq \delta$, we obtain $\epsilon_{(\alpha',\beta')} \leq | x_{i'}^{\alpha'} - x_{j'}^{\beta'}| \leq \epsilon_{(\alpha',\beta')} + 2 CR \delta.$ This shows that $Z \in U_{(i,j)}^{(\alpha,\beta)} \cap U_{(i',j')}^{(\alpha',\beta')}$ and $I_{sc,ng}^2$ is contained in the claimed union.
\end{proof} 
\begin{rmk}\label{grazing-flow}
One can show that the Hausdorff measure  $\mathcal{H}^{2d(N_1 + N_2)}(I_{sc,g}^1) = 0$. 
\end{rmk}
Next, we estimate the measure of our "bad" sets using Lemma \ref{covering-lem} and Remark \ref{grazing-flow}.

\begin{lem}\label{flow-measure} Assume that $0<\epsilon_1, \epsilon_2< 1$. We have the measure estimate
$$
\mathcal{H}^{2d(N_{1} + N_{2})} ( I_{sc,g}^1 \cup I_{mc}^1 \cup I_{sc, ng}^2) \leq C(N_1,N_2, d, R) \rho^{d(N_1+N_2- 2)} \delta^2.
$$
\end{lem}

\begin{proof} By Lemma \ref{covering-lem} and Remark \ref{grazing-flow}, it suffices to uniformly estimate $
U^{(\alpha,\beta)}_{(i,j) } \cap U^{(\alpha',\beta')}_{(i',j')}$ for the indices $\alpha \leq \beta$, $\alpha'\leq \beta'$, $(i,j) \in \mathcal{I}^{(\alpha,\beta)}$, and $(i',j') \in \mathcal{I}^{(\alpha',\beta')}$ such that either $(\alpha,\beta) \neq  (\alpha',\beta')$ or $(i,j) \neq (i',j')$. 
Note that this implies we have only three possibilities for the coordinates $x_i^{\alpha}, x_j^{\beta},x_{i'}^{\alpha'}, x_{j'}^{\beta'}$: 
\begin{enumerate} 
\item $x_i^{\alpha} = x_{i'}^{\alpha'}$ and the rest are distinct.
\item $x_j^{\beta} = x_{j'}^{\beta'}$ and the rest are distinct.
\item All coordinates are distinct.
\end{enumerate}
By symmetry, cases 1 and 2 are identical, so it suffices to prove the bounds for cases 1 and 3 only.
\begin{enumerate}
\item[1.] Assume  $x_i^{\alpha} = x_{i'}^{\alpha'}$ and the rest are distinct. Then, we have that 
\begin{align*}
\int_{\R^{2d(N_{1}+N_{2})} } \1_{ U^{(\alpha,\beta)}_{(i,j) } \cap U^{(\alpha',\beta')}_{(i',j')} }(Z ) dZ =
|B(0,\rho)|_d^{N_{1}+N_{2} -3} |B(0,R)|_d^{N_{1} +N_{2}} \int_{B(0,\rho)^3} \1_{S_{(i,j,j')}^{(\alpha, \beta,\beta') }} dx_i^{\alpha} dx_j^{\beta}  dx_{j'}^{\beta'} 
\end{align*}
where 
\begin{align*}
S_{(i,j,j')}^{(\alpha, \beta,\beta') }  : = \Big\{(x_i^{\alpha}, x_j^{\beta} , x_{j'}^{\beta'} ) \big/ \epsilon_{(\alpha,\beta)} \leq |x_i^{\alpha} -  x_j^{\beta}| \leq \epsilon_{(\alpha,\beta)}+ 2 C\delta R , \, \, \epsilon_{(\alpha',\beta')} \leq | x_{i}^{\alpha}- x_{j'}^{\beta'}| \leq  \epsilon_{(\alpha',\beta')}  + 2 C\delta R\Big\}.
\end{align*}
We have
\begin{align*}
\int_{B(0,\rho)^3} \1_{S_{(i,j,j')}^{(\alpha, \beta,\beta') }} dx_i^{\alpha} dx_j^{\beta}  dx_{j'}^{\beta'} & \leq  \int_{\R^d} \int_{\R^d} \int_{B(0,\rho)}  \1_{S_{(i,j,j')}^{(\alpha, \beta,\beta') }}dx_i^{\alpha} dx_j^{\beta}  dx_{j'}^{\beta'} \\
& \leq C\int_{B(0,\rho)} \Big((2C\delta R + \epsilon_{(\alpha,\beta)} )^d  - \epsilon_{(\alpha,\beta)}^d\Big)  \Big((2C\delta R + \epsilon_{(\alpha',\beta')})^d  - \epsilon_{(\alpha',\beta')}^d\Big) dx
\end{align*}
Since we picked $\delta$, $R$ such that $\delta R \ll 1$, this is bounded above by $C\rho^{d} R^2 \delta^2$. Hence we get 
$$
\int_{\R^{2d(N_{1}+N_{2})} } \1_{ U^{\alpha}_{(i,j) } \cap U^{\alpha'}_{(i',j')} }(Z ) dZ \leq C \rho^{d(N_{1} + N_{2} -2)} R^{d(N_{1} + N_{2}) +2} \delta^2.
$$
This completes this case. \\

\item[3.] Assume that all coordinates $x_i^{\alpha}, x_j^{\beta},x_{i'}^{\alpha'}, x_{j'}^{\beta'}$ are distinct. In this case, we have  
\begin{align*}
\int_{\R^{2d(N_{1}+N_{2})} } \1_{ U^{(\alpha,\beta)}_{(i,j) } \cap U^{(\alpha',\beta')}_{(i',j')} }(Z ) dZ =|B(0,\rho)|_d^{N_{1}+N_{2} -4} |B(0,R)|_d^{N_{1} +N_{2}} \int_{B(0,\rho)^4} \1_{S_{(i,j)}^{(\alpha,\beta)} \cap S_{(i',j')}^{(\alpha',\beta')}} dx_i^{\alpha} dx_j^{\beta}  dx_{i'}^{\alpha'} dx_{j'}^{\beta'} 
\end{align*}
where here we are defining 
\begin{align*}
S_{(i,j)}^{(\alpha,\beta)} & : = \Big\{(x_i^{\alpha}, x_j^{\beta} , x_{i'}^{\alpha'}, x_{j'}^{\beta'} ) \big/ \, \epsilon_{(\alpha,\beta)} \leq |x_i^{\alpha} -  x_j^{\beta}| \leq \epsilon_{(\alpha,\beta)} + 2 C\delta R\Big\} \end{align*}
\begin{align*}
S_{(i',j')}^{(\alpha',\beta')} & : = \Big\{(x_i^{\alpha}, x_j^{\beta} , x_{i'}^{\alpha'},x_{j'}^{\beta'} ) \big/ \, \epsilon_{(\alpha',\beta')} \leq |x_{i'}^{\alpha'} -  x_{j'}^{\beta'}| \leq \epsilon_{(\alpha',\beta')} + 2 C\delta R\Big\} \end{align*}
Again, we can estimate the above integral by Fubini
\begin{align*}
 \int_{B(0,\rho)^4} \1_{S_{(i,j)}^{(\alpha,\beta)} \cap S_{(i',j')}^{(\alpha',\beta')}} dx_i^{\alpha} dx_j^{\beta}  dx_{i'}^{\alpha'} dx_{j'}^{\beta'} & \leq \int_{\R^d} \int_{B(0,\rho)} \int_{\R^d} \int_{B(0,\rho)}\1_{S_{(i,j)}^{(\alpha,\beta)} \cap S_{(i',j')}^{(\alpha',\beta')}} dx_i^{\alpha} dx_j^{\beta}  dx_{i'}^{\alpha'} dx_{j'}^{\beta'} \\
 & =  C\rho^d \Big( (2 C\delta R + \epsilon_{(\alpha,\beta)})^d - \epsilon_{(\alpha,\beta)}^d \Big)\rho^d  \Big( (2C \delta R + \epsilon_{(\alpha',\beta')}) ^d - \epsilon_{(\alpha',\beta')}^d  \Big) \\
 & \leq C\rho^{2d}R^2 \delta^2.
\end{align*}
Hence we obtain
$$
\int_{\R^{2d(N_{1}+N_{2})} } \1_{ U^{\alpha}_{(i,j) } \cap U^{\alpha'}_{(i',j')} }(Z ) dZ \leq C \rho^{d(N_{1} + N_{2} -2)} R^{d(N_{1} + N_{2}) +2} \delta^2.
$$ 
\end{enumerate}
\end{proof}
Using the measure estimates in Lemma \ref{flow-measure}, we obtain via standard covering arguments \cite{Ale75},\cite{thesis} the global flow existence of our hard sphere system.
\begin{thm}\label{global-flow-thm} 
Let $N_{1}, N_{2} \in \N_+$ and $\epsilon_1, \epsilon_2 > 0$ and recall the definition of the energy \eqref{energy}. Then there exists a family of measure preserving maps $( \Psi_{(N_{1}, N_{2})}^t)_{t\in \R} : \mathcal{D} \rightarrow  \mathcal{D}$ such that for all $t \in \R^+$
\begin{enumerate}
\item[1.] $\Psi_{(N_{1},N_{2})}^{t + s} (Z) =  \Psi_{(N_{1}, N_{2})}^t(Z) \circ  \Psi_{(N_{1}, N_{2})}^s(Z) =  \Psi_{(N_{1}, N_{2})}^s(Z) \circ  \Psi_{(N_{1}, N_{2})}^t(Z)$ for a.e. $Z \in \mathcal{D}$. 

\item[2.] $E( \Psi_{(N_{1}, N_{2})}^t(Z)) = E(Z)$ for a.e. $Z \in \mathcal{D}$. 

\item[3.] We have $ \Psi_{(N_{1}, N_{2})}^t(T(Z)) = \Psi_{(N_{1}, N_{2})}^t(Z)$ for $\mathcal{H}^{2d(N_{1} + N_{2}) - 1}$ a.e. $Z \in \partial_{sc, ng} \mathcal{D}$.
\end{enumerate}
\end{thm}
\section{Local Well-Posedness}\label{Local Well-Posedness}
In this section, we establish local in time well-posedness of the BBGKY hierarchy, the Boltzmann hierarchy, and the Boltzmann system for mixtures in their mild forms. The proofs for these three results are carried out in a similar spirit, and we will present it only in the BBGKY setting. 
\subsection{Well-Posedness of BBGKY Hierarchy}\label{section LWP-BBGKY}
 Recall that $\mathscr{T}$ is the set of types of particles, as given in \eqref{set of types}. Let $\bm{\epsilon} = (\epsilon_{(1,0)} , \epsilon_{(0,1)})$ record the diameters of the particles of each type, and let $\bm{N}=(N_{(1,0)}, N_{(0,1)})$ record the total number of particles of each type. Throughout this section, we will assume the Boltzmann-Grad scaling \eqref{mixed-boltz-grad,1} holds. We first define the Maxwellian weighted spaces that we will be working in. 

\begin{definition} Recall the definition of the phase space \eqref{phase-space} and the energy \eqref{energy}. Fix $\gamma > 0$, and $\mu\in \R$. Define the index set 
\begin{equation}\label{n-set-def}
    [\bm{N}] : = \{ 1, \dots, N_{(1,0)}\} \times \{1, \dots, N_{(0,1)}\}.
\end{equation}
\begin{itemize}
\item For $\bm{s} \in[\bm{N}]$ and $f_{\bm{N}}^{(\bm{s})} \in L^\infty (\mathcal{D}_{\bm{\epsilon}}^{\bm{s}}; \R)$, define the Banach space
\begin{equation}\label{norm, s,l,epsilon,beta}
X_{\bm{s}, \bm{\epsilon}, \gamma} : = \left\{ f_{\bm{N}}^{(\bm{s})} \in L^\infty (\mathcal{D}_{\bm{\epsilon}}^{\bm{s}}; \R) \,  \Big| \,  |f_{\bm{N}}^{(\bm{s})}|_{\bm{s}, \bm{\epsilon}, \gamma} < \infty \right\}, \quad \text{with} \quad |f_{\bm{N}}^{(\bm{s})}|_ {\bm{s}, \bm{\epsilon}, \gamma} : = \left \Vert e^{\gamma E(\cdot)} f_{\bm{N}}^{(\bm{s})} \right \Vert_{L^\infty(\mathcal{D}_{\bm{\epsilon}}^{\bm{s}})}.
\end{equation}

\item Define for $F_{\bm{N}}= \left(f_{\bm{N}}^{(\bm{s})}\right)_{ \bm{s}\in [\bm{N}]}$ with $f_{\bm{N}}^{(\bm{s})} \in L^\infty (\mathcal{D}_{\bm{\epsilon}}^{\bm{s}}; \R) $ the norm and the Banach space 
\begin{equation*}
\bm{X}_{\bm{\epsilon}, \gamma, \mu}^{\bm{N}} : = \left \{ F_{\bm{N}} =  (f_{\bm{N}}^{(\bm{s})})_{ \bm{s} \in [\bm{N}]} \, \Big| \, f_{\bm{N}}^{(\bm{s})} \in L^\infty (\mathcal{D}_{\bm{\epsilon}}^{\bm{s}}; \R), \quad \Vert F_{\bm{N}} \Vert_{\bm{\epsilon}, \gamma, \mu}<\infty\right\}.
\end{equation*}
\begin{equation*}
\Vert F_{\bm{N}} \Vert_{\bm{\epsilon}, \gamma, \mu} : = \sup_{\bm{s}\in [\bm{N}]} \left( e^{ \mu |\bm{s}|} |f_{\bm{N}}^{(\bm{s})}|_{\bm{s}, \bm{\epsilon}, \gamma} \right)
\end{equation*}

\item Let $T > 0$, $\bm{\gamma} : [0,T] \rightarrow (0,\infty)$, and $\bm{\mu} : [0,T] \rightarrow \R$ be non-increasing functions. Define $\bm{X}_{\bm{\epsilon}, \bm{\gamma}, \bm{\mu} }^{\bm{N}}([0,T])$ be the set of all continuous mappings $F_{\bm{N}}: [0,T] \rightarrow \bm{X}_{\bm{\epsilon},\gamma(t), \mu(t)}^{\bm{N}}$ such that the following the norm is finite
$$
\vertiii{F_{\bm{N}}}_{\bm{\epsilon}, \bm{\gamma},\bm{\mu} } : = \sup_{0\leq t \leq T} \Vert F_{\bm{N}}(t) \Vert_{\bm{\epsilon},\gamma(t), \mu(t)}.
$$
\end{itemize}
\end{definition}
 To state the mild formulation of the BBGKY hierarchy, we define for each $\bm{s} \in \N_+^2$ and $t > 0$  
\begin{equation}\label{T-def}
 T_{\bm{s},\bm{\epsilon}}^t: X_{\bm{s}, \bm{\epsilon}, \gamma} \rightarrow X_{\bm{s}, \bm{\epsilon}, \gamma}, \quad \text{given by} \quad
 T_{\bm{s}, \bm{\epsilon}}^t f^{(\bm{s})}_{\bm{N}}(Z_{\bm{s}} ): = f^{(\bm{s})}_{\bm{N}} \big(\Psi_{\bm{s},\bm{\epsilon}}^{-t}( Z_{\bm{s}}))
\end{equation}
where $\Psi_{\bm{s},\bm{\epsilon}}$ is the $\bm{s}$ particle flow given by Theorem \ref{global-flow-thm}. Conservation of energy and invariance of the flow under particle collisions as proven in Theorem \ref{global-flow-thm} imply that $T_{\bm{s}, \bm{\epsilon}}^t$ is an isometry of $X_{\bm{s},\bm{\epsilon}, \gamma}$. Also $T_{\bm{s},\bm{\epsilon}}$ is the semigroup which generates the left hand side of \eqref{BBGKY-strong}. With this in mind, we obtain the following mild formulation of the BBGKY hierarchy.

\begin{definition}\label{BBGKY weak defn}
Given $\bm{\gamma}: [0,T] \rightarrow (0,\infty)$, $\bm{\mu}: [0,T] \rightarrow \R$, and $F_{\bm{N},0} \in  \bm{X}^{\bm{N}}_{\bm{\epsilon},\bm{\gamma},\bm{\mu}}$ we say that 
$$
F_{\bm{N}}= \left(f_{\bm{N}}^{(\bm{s})}\right)_{\bm{s} \in [\bm{N}]}  \in \bm{X}^{\bm{N}}_{\bm{\epsilon},\bm{\gamma},\bm{\mu}}([0,T])
$$ is a mild solution to the BBGKY Hierarchy with initial data $F_{\bm{N},0}$ if for each $\bm{s} \in [\bm{N}]$ and $t\in[0,T]$, we have
\begin{equation}\label{bbgky-weak-soln}
    F_{\bm{N}}(t) = \bm{T}_{\bm{s}, \bm{\epsilon}}(t) F_{\bm{N},0} + \sum_{\alpha,\beta \in \mathscr{T}} \int_0^t \bm{T}_{\bm{s},\bm{\epsilon}}^{t -\tau} \bm{\mathcal{C}}_{\beta}^\alpha F_{\bm{N}}(\tau) d\tau, 
\end{equation}
where for each $\alpha,\beta \in \mathscr{T}$, we define the operators $\bm{\mathcal{C}}^{\alpha}_{\beta},\bm{T}_{\bm{\epsilon}} : \bm{X}_{ \bm{\epsilon}, \gamma, \mu}^{\bm{N}} \rightarrow \bm{X}_{ \bm{\epsilon}, \gamma,\mu}^{\bm{N}}$ by their action on each component:
\begin{equation}
\bm{\mathcal{C}}^{\alpha}_{\beta} F_{\bm{N}} := \left (\mathcal{C}^{\alpha}_{\bm{s},\bm{s}+\beta} f^{(\bm{s}+\beta)}_{\bm{N}}\right)_{\bm{s} \in [\bm{N}]},\qquad 
\bm{T}_{\bm{\epsilon}}^t F_{\bm{N}} : = \left (T_{\bm{s},\bm{\epsilon}}^t f^{(\bm{s})}_{\bm{N}}\right)_{\bm{s} \in [\bm{N}]}.
\end{equation}
Here, the operators $\mathcal{C}^\alpha_{\bm{s}, \bm{s} + \beta}$ are given in Definition \ref{bbgky-collision-operator} and $T_{\bm{s},\bm{\epsilon}}$ is given by \eqref{T-def}.
\end{definition}

\begin{rmk}\label{ill def C}
We note that the above collision operators $\mathcal{C}_{\bm{s},\bm{s}+\beta}^\alpha$ are ill-defined on $L^\infty$ since they involve integration over a set of measure zero (the sphere $\mathbb{S}^{d-1}$). However, by filtering our BBGKY hierarchy by the flow $\bm{T}_{\bm{\epsilon}}^{-t}$, we may obtain a well defined operator on $\bm{X}_{\bm{\epsilon},\gamma,\mu}^{\bm{N}}$. This is done in detail in the erratum of Chapter 5 of \cite{GSRT13} and does not affect the energy estimates or local well-posedness of the hierarchy. This filtering process can be adapted to our context. Hence, we will abuse the notation and continue to work with the operators $\mathcal{C}_{\bm{s},\bm{s}+\beta}^\alpha$. See \cite{Sim14} for a different approach which avoids this issue by working with measures on the phase space. 
\end{rmk}

\begin{lem}\label{pointwise-estimates}
Assume that we have the Boltzmann-Grad scalings \eqref{mixed-boltz-grad,1}. For all $\bm{s} = (s_{(1,0)}, s_{(0,1)}) \in \N_+^2$, all $\alpha, \beta \in \mathscr{T}$, and all $Z_{\bm{s}} \in \mathcal{D}_{\bm{\epsilon}}^{\bm{s}}$ we have
\begin{equation}
    \left|\mathcal{C}_{\bm{s}, \bm{s}+\beta}^\alpha f^{(\bm{s}+\beta)}(Z_{\bm{s}}) \right| \leq C \gamma^{-d/2} \left[s_{\alpha} \gamma^{-1/2} + \sum_{k=1}^{s_\alpha} | v_k^{\alpha}| \right]e^{-\gamma E(Z_{\bm{s}})} |f^{(\bm{s}+\beta)}|_{\bm{s} +\beta, \bm{\epsilon}, \gamma}
\end{equation}
The constant $C$ above depends only on $d, c_1, c_2, a,$ and $b$ as in the Boltzmann-Grad scalings \eqref{mixed-boltz-grad,1} and the masses $M_1, M_2$ of the particles. 
\end{lem}
\begin{proof} Fix $\alpha,\beta,$ and $Z_{\bm{s}}$. Using the definition of the operator, the triangle inequality, the definition of the norm \eqref{norm, s,l,epsilon,beta}, and the conservation of energy we obtain
\begin{align*}
|\mathcal{C}^{\alpha}_{\bm{s},\bm{s}+\beta} f^{(\bm{s}+\beta)}(Z_{\bm{s}})| &\leq C\sum_{k=1}^{s_\alpha} \int_{\R^d} (|v_{s_\beta+1}^{\beta}| + |v_k^{\alpha}|)e^{-\gamma E(Z_{\bm{s}})}e^{-\gamma M_\beta |v_{s_\beta+1}^\beta|^2} | f^{(\bm{s}+\beta)}|_{\bm{s}+\beta,\bm{\epsilon},\gamma}  dv_{s_\beta+1}^{\beta} \\
& = C e^{-\gamma E(Z_{\bm{s}})}| f^{(\bm{s}+\beta)}|_{\bm{s}+\beta,\bm{\epsilon},\gamma} \\
& \quad \times \left[s_{\alpha} \int_{\R^d}  |v_{s_\beta+1}^\beta| e^{-\gamma M_\beta |v_{s_\beta+1}^\beta|^2}dv_{s_\beta+1}^\beta + \sum_{k=1}^{s_\alpha} |v_k^\alpha| \int_{\R^d} e^{-\gamma M_\beta |v_{s_\beta+1}^\beta|^2}dv_{s_\beta+1}^\beta \right].
\end{align*}
Computing these Gaussian integrals results in the desired bounds.
\end{proof}
\begin{rmk}
Using these estimates, it can be shown that the operators $\bm{\mathcal{C}}_\beta^\alpha$ are continuous operators on their respective spaces. See \cite{GSRT13} for proof.
\end{rmk}
\begin{lem}\label{BBGKY-operator-estimates}
Fix $\bm{N} \in \N_+^2$ and $\bm{\epsilon}\in (0,\infty)^2$ to agree with the mixed Boltzmann-Grad scaling \eqref{mixed-boltz-grad,1}. For any $\gamma_0 > 0$ and $\mu_0 \in \R$, if $T,\lambda>0$ such that $T\lambda < \gamma$, then we have the following bounds on $\bm{X}^{\bm{N}}_{\bm{\epsilon},\bm{\gamma},\bm{\mu}}([0,T])$ 
\begin{equation}
   \vertiii{ \int_0^t \bm{T}_{\bm{\epsilon}}^{t - \tau} \bm{\mathcal{C}}^{\alpha}_{\beta} F_{\bm{N}}(\tau) d\tau }_{ \bm{\epsilon}, \bm{\gamma},\bm{\mu} } \leq c\vertiii{F_{\bm{N}}}_{\bm{\epsilon}, \bm{\gamma}, \bm{\mu}} \qquad \text{ for all $\alpha,\beta \in\mathscr{T}$.}
\end{equation}
where $\bm{\gamma}(t) = \gamma_0 - \lambda t$, $\bm{\mu}(t) = \mu_0 -\lambda t$ and $c= c(d,c_1, c_2, a,b, M_1, M_2, \gamma_0, \mu_0, \lambda,T)$. We can pick $T >0$ and a $\lambda \in (0,\gamma_0/T)$ independent of $\bm{N},\bm{\epsilon}$ such that $c<1/8$. 
\end{lem}
\begin{proof} This is a standard argument which follows from carefully lifting the estimate in Lemma \ref{pointwise-estimates} to the space $\bm{X}^{\bm{N}}_{\bm{\epsilon}, \bm{\gamma}(t), \bm{\mu}(t)}$ for each $t > 0$, and then estimating a time integral. Details can be found in \cite{GSRT13}.
\end{proof}
\begin{thm} \label{LWP-BBGKY}
 Fix $\bm{N} \in \N_+^2$ and $\bm{\epsilon} \in (0,\infty)^2$ to agree with the Boltzmann-Grad scaling \eqref{mixed-boltz-grad,1}. Let $\mu_0 \in \R$ and $\gamma_0 >0$ be given. Then there exists $T,\lambda > 0$ independent of $\bm{N},\bm{\epsilon}$ such that for any $F_{\bm{N},0} \in X^{\bm{N}}_{\bm{\epsilon}, \gamma_0,\mu_0}$ we have a unique mild solution $F_{\bm{N}} \in \bm{X}^{\bm{N}}_{\bm{\epsilon},\bm{\gamma},\bm{\mu}}([0,T])$ to the BBGKY hierarchy \eqref{bbgky-weak-soln} with $\bm{\gamma}(t) = \gamma_0 - \lambda t$ and $\bm{\mu}(t) = \mu_0 - \lambda t$. Moreover, this unique solution satisfies 
 \begin{equation}
     \vertiii{ F_{\bm{N}}}_{\bm{\epsilon},\bm{\gamma},\bm{\mu}} \leq 2 \Vert F_{\bm{N},0} \Vert_{\bm{\epsilon}, \gamma_0, \mu_0}
 \end{equation}
 Additionally, for any $G_{\bm{N}} \in \bm{X}_{\bm{\epsilon},\bm{\gamma},\bm{\mu}}^{{\bm{N}}}([0,T])$, we have
   \begin{equation}
   \vertiii{ \int_0^t \bm{T}_{\bm{\epsilon}}^{t - \tau} \bm{\mathcal{C}}^{\alpha}_{\beta} G_{\bm{N}}(\tau) d\tau }_{ \bm{\epsilon}, \bm{\gamma},\bm{\mu} } \leq 1/8\vertiii{G_{\bm{N}}}_{\bm{\epsilon}, \bm{\gamma}, \bm{\mu}}, \qquad \text{for all $\alpha,\beta \in \mathscr{T}.$}
\end{equation}
 \end{thm}
\begin{proof} The proof follows by applying a fixed point argument which can be done thanks to Lemma \ref{BBGKY-operator-estimates}.
\end{proof}
\subsection{Well-Posedness of the Boltzmann Hierarchy}\label{section LWP-Boltzmann hierarchy}
 This subsection is devoted to proving the local well-posedness of the Boltzmann Hierarchy. The estimates and proofs essentially mirror those of the previous Subsection \ref{section LWP-BBGKY}, with appropriate adjustment of the functional spaces. We begin by introducing the relevant Maxwellian weighted spaces.

\begin{definition} Recall the definition \eqref{energy}of the energy $E(\cdot)$. Fix $\gamma > 0$, and $\mu\in \R$. 
\begin{itemize}
\item For $\bm{s} \in \N_+^2$ and $f^{(\bm{s})} \in L^\infty (\R^{d|\bm{s}|}; \R)$, define the Banach space  
\begin{equation}\label{norm, s,l,0,beta}
|f^{(\bm{s})}|_ {\bm{s}, 0, \gamma} : = \left \Vert e^{\gamma E(\cdot)} f^{(\bm{s})} \right \Vert_{L^\infty(\R^{d|\bm{s}|})}, \qquad X_{\bm{s}, 0, \gamma} : = \left\{ f^{(\bm{s})} \in L^\infty (\R^{d|\bm{s}|}; \R) \,  \Big| \,  |f^{(\bm{s})}|_{\bm{s}, 0, \gamma} < \infty \right\}.
\end{equation}

\item Define for $F= \left(f^{(\bm{s})}\right)_{ \bm{s}\in \N_+^2}$ with $f^{(\bm{s})} \in L^\infty (\R^{d|\bm{s}|}; \R) $ the Banach space 
\begin{equation}
\Vert F \Vert_{0, \gamma, \mu} : = \sup_{\bm{s}\in \N_+^2} \left( e^{ \mu |\bm{s}|} |f^{(\bm{s})}|_{\bm{s}, 0, \gamma} \right)
\end{equation}
\begin{equation}
\bm{X}_{0, \gamma, \mu}^{\infty} : = \left \{ F =  (f^{(\bm{s})})_{ \bm{s} \in \N_+^2} \, \Big| \, f^{(\bm{s})} \in L^\infty (\R^{d|\bm{s}|}; \R), \quad \Vert F \Vert_{0, \gamma, \mu}<\infty\right\}.
\end{equation}

\item Let $T > 0$, $\bm{\gamma} : [0,T] \rightarrow (0,\infty)$, and $\bm{\mu} : [0,T] \rightarrow \R$ be non-increasing functions. Define $\bm{X}_{0, \bm{\gamma}, \bm{\mu} }^{\infty}([0,T])$ be the set of all continuous mappings $F: [0,T] \rightarrow \bm{X}_{0,\gamma(t), \mu(t)}^{\infty}$ such that the following norm is finite
$$
\vertiii{F}_{0, \bm{\gamma},\bm{\mu} } : = \sup_{0\leq t \leq T} \Vert F(t) \Vert_{0,\gamma(t), \mu(t)}.
$$
\end{itemize}
\end{definition}
 In order to state the mild formulation of the Boltzmann hierarchy, we define for each $\bm{s} \in \N_+^2$ and $t > 0$ the free flow propagator  
\begin{equation}\label{S-def}
 S_{\bm{s}}^t: X_{\bm{s}, 0, \gamma} \rightarrow X_{\bm{s}, 0, \gamma}, \qquad S_{\bm{s}}^t f^{(\bm{s})}(Z_{\bm{s}} ): = f^{(\bm{s})} \big(\Phi_{\bm{s}}^{-t}( Z_{\bm{s}}))
 \end{equation}
where $\Phi_{\bm{s}}$ is the $\bm{s} $ particle free flow given by 
\begin{equation} \label{phi-s-def}
    \Phi_{\bm{s}}^{t}( Z_{\bm{s}}) = Z_{\bm{s}} + t( V_{\bm{s}},0). 
\end{equation}
It follows directly from the definitions that $S_{\bm{s}}$ is an isometry of $X_{\bm{s},0, \gamma}$. Moreover, one can check that $S_{\bm{s}}$ is the semigroup whose generator is the left hand side of  \eqref{boltz-strong}. With these definitions in hand, we define a solution to the Boltzmann Hierarchy as follows.   

\begin{definition}
Given $\bm{\gamma}: [0,T] \rightarrow (0,\infty)$, $\bm{\mu}: [0,T] \rightarrow \R$, and $F_{0} \in  \bm{X}^{\infty}_{0,\bm{\gamma},\bm{\mu}}$ we say that 
$$
F= \left(f^{(\bm{s})}\right)_{\bm{s} \in \N_+^2}  \in \bm{X}^{\infty}_{0,\bm{\gamma},\bm{\mu}}([0,T])
$$ is a mild solution to the Boltzmann Hierarchy with initial data $F_{0} = ( f_0^{(\bm{s})})_{\bm{s} \in \N_+^2}$ if for each $t\in [0,T]$ we have 
\begin{equation}\label{boltz-weak-soln}
    F(t) = \bm{S}^t F_{0} + \sum_{\alpha,\beta \in \mathscr{T}} \int_0^t \bm{S}^{t -\tau} \bm{\mathfrak{C}}_{\beta}^\alpha F(\tau) d\tau, 
\end{equation}
where for each $\alpha,\beta \in \mathscr{T}$, we define the operators $\bm{\mathfrak{C}}^{\alpha}_{\beta},\bm{S} : \bm{X}_{ 0, \gamma, \mu}^{\infty} \rightarrow \bm{X}_{ 0, \gamma,\mu}^{\infty}$ by their action on each component:
\begin{equation}
\bm{\mathfrak{C}}^{\alpha}_{\beta} F := \left (\mathscr{C}^{\alpha}_{\bm{s},\bm{s}+\beta} f^{(\bm{s}+\beta)}\right)_{\bm{s} \in \N_+^2},\qquad
\bm{S}^t F : = \left (S_{\bm{s}}^t f^{(\bm{s})}\right)_{\bm{s} \in \N_+^2} .
\end{equation}
Here, the operators $\mathscr{C}^\alpha_{\bm{s}, \bm{s} + \beta}$ are given in by \eqref{boltz alpha beta kernel} and $S_{\bm{s}}$ is given by \eqref{S-def}.
\end{definition}
\begin{rmk}\label{ill def C boltz}
As noted in Remark \ref{ill def C}, the operators $\mathscr{C}_{\bm{s},\bm{s}+\beta}^\alpha$ are ill defined on $L^\infty$ due to the integration over the measure zero set $\mathbb{S}^{d-1}$. As in the BBGKY case, one can filter the infinite hierarchy by $\bm{S}^{-t}$ to obtain a well defined operator. Hence, we will abuse notation and continue to use the operators $\mathscr{C}_{\bm{s},\bm{s}+\beta}^\alpha$. 
\end{rmk}
\begin{rmk}
Using estimates similar to those found in Section \ref{LWP-BBGKY}, it can be shown that the operators $ \bm{\mathfrak{C}}_{\beta}^\alpha$ are continuous operators. A priori estimates analogous to those found in Section \ref{LWP-BBGKY} lead to the following well posedness theorem. See \cite{GSRT13} for proof.
\end{rmk}
\begin{thm} \label{LWP-boltz-hierarchy}
Let $\mu_0 \in \R$ and $\gamma_0 >0$ be given. Then there exists $T,\lambda > 0$ such that for any $F_{0} \in X^{\infty}_{0, \gamma_0,\mu_0}$ we have a unique mild solution $F \in \bm{X}^{\infty}_{0,\bm{\gamma},\bm{\mu}}([0,T])$ to the Boltzmann hierarchy \eqref{boltz-weak-soln} with $\bm{\gamma}(t) = \gamma_0 - \lambda t$ and $\bm{\mu}(t) = \mu_0 - \lambda t$. Moreover, this unique solution satisfies 
 \begin{equation}
     \vertiii{ F}_{0,\bm{\gamma},\bm{\mu}} \leq 2 \Vert F_{0} \Vert_{0, \gamma_0, \mu_0}
 \end{equation}
 Additionally, for any $G \in \bm{X}_{0, \bm{\gamma},\bm{\mu}}^\infty([0,T])$ we have
   \begin{equation}
   \vertiii{ \int_0^t \bm{S}^{ t - \tau} \bm{\mathfrak{C}}^{\alpha}_{\beta} G(\tau) d\tau }_{ 0, \bm{\gamma},\bm{\mu} } \leq 1/8\vertiii{G}_{0, \bm{\gamma}, \bm{\mu}}, \qquad \text{for all $\alpha,\beta \in \mathscr{T}.$}
\end{equation}
 \end{thm}

\subsection{Well-Posedness of the Boltzmann Equation for Mixtures}
Recall from Section \ref{Mixed Marginals and BBGKY Hierarchies} the Boltzmann system for mixtures given by \eqref{alpha beta boltz pde}. Recall also that the set of types $\mathscr{T}$ is given by \eqref{set of types}. We begin our analysis by defining the appropriate function spaces. 
\begin{definition}
Let $\gamma> 0$ and $\mu\in \R$. Define for each $\alpha \in \mathscr{T}$ the one particle space
$$
X_{\alpha,\gamma,\mu} : = \{ g_\alpha | \, \,g_\alpha \in L^\infty(\R^{2d}), \,\, |g_\alpha|_{\alpha,\gamma,\mu},  <\infty \}, \qquad 
|g_\alpha|_{\alpha,\gamma,\mu} = \left \Vert e^{\mu+\gamma M_\alpha|v|^2} g_\alpha(x,v) \right \Vert_{L^\infty(\R^{2d})},
$$
where $M_\alpha$ is the mass of the type $\alpha$ particle. We also define the two particle space given by $
X_{\gamma,\mu} : = \prod_{\alpha\in \mathscr{T}} X_{\alpha, \gamma,\mu}$ with the induced $\ell^1$ product norm $
\vert G \vert_{\gamma,\mu} : = \sum_{\alpha \in \mathscr{T}}\vert g_\alpha \vert_{\alpha, \gamma,\mu}$ where $G = (g_{\alpha})_{\alpha \in \mathscr{T}}.$ \\
Given $T > 0$ and non-increasing functions $\bm{\gamma} : [0,T] \rightarrow \R_+$, $\bm{\mu} :[0,T] \rightarrow \R$, we define 
$$
X_{\bm{\gamma}, \bm{\mu}}([0,T]) : = \left\{ G =(g_\alpha)_{\alpha \in \mathscr{T}} \big| \, \,  g_\alpha \in \mathcal{C}^0([0,T]; L^\infty(\R^{2d})) \text{ and } \Vert G \Vert_{\bm{\gamma}, \bm{\mu}} < \infty\right\},
$$
where the norm is given by 
$$
\Vert G \Vert_{\bm{\gamma}, \bm{\mu}} : = \sup_{t \in [0,T]} \vert G(t,\cdot) \vert_{\bm{\gamma}(t), \bm{\mu}(t)}.
$$
\end{definition}
It is clear that the space $(X_{\bm{\gamma}, \bm{\mu}}([0,T]),  \Vert \cdot \Vert_{\bm{\gamma}, \bm{\mu}})$ is a complete metric space. We will study the local well posedness of the Boltzmann equation for mixtures on the above spaces. We begin with defining our notion of solution. First, fix $T>0$ and $\bm{\gamma} : [0,T] \rightarrow \R_+$, $\bm{\mu} : [0,T] \rightarrow \R$ be non-increasing functions. Recalling the constants \eqref{A a b def} and collision operators given in Definition \ref{boltz-kernels}, introduce the non-linear operator $\mathcal{N}$ on $ X_{\bm{\gamma}, \bm{\mu}}([0,T]) $ for each component $\alpha \in \mathscr{T}$ by
$$
\left[\mathcal{N} (G)\right]_\alpha: =
\sum_{\beta \in \mathscr{T}} A^\alpha_\beta Q^\alpha_\beta(g_{\alpha},g_{\beta}), \quad \text{where} \quad G = (g_\alpha)_{\alpha \in \mathscr{T}}.
$$
Recall the free flow operator $S_{(1,1)}^t$ from \eqref{S-def}. For this section, we will set $S^t = S_{(1,1)}^t$ to reduce the notation. As in the case for the BBGKY and Boltzmann hierarchies, we consider mild formulations. 
\begin{definition}
We say $G = (g_\alpha)_{\alpha\in\mathscr{T}} \in X_{\bm{\gamma}, \bm{\mu}}([0,T]) $ solves the Boltzmann equation with initial data $G_0 =(g_{\alpha,0})_{\alpha\in\mathscr{T}}   \in X_{\gamma_0,\mu_0}$ if 
\begin{equation}\label{pde-soln-weak}
G(t) = S^t G_0 + \int_0^t S^{t - \tau} \mathcal{N}G (\tau) d\tau, \qquad \text{for every $t \in [0,T]$.}
\end{equation}
\end{definition}
\begin{rmk}
As in Remarks \ref{ill def C}, \ref{ill def C boltz}, the operators $Q_{\beta}^\alpha$ can be filtered by the free flow $S^{-t}$ in order to define the operator $\mathcal{N}$ on $L^\infty$. Hence, we will abuse notation and continue to work with the operators $Q_{\beta}^\alpha$.
\end{rmk}
As in the previous subsections, we first prove some estimates on the nonlinearity $\mathcal{N}$ and use this to set up a contraction mapping. 

\begin{lem}\label{nonlinear-estimates}
Fix $\gamma> 0$, $\mu \in \R$, and $G = (g_{\alpha})_{\alpha \in \mathscr{T}}, G' = (g_\alpha')_{\alpha \in \mathscr{T}} \in X_{\gamma,\mu}$. Then we obtain the pointwise estimates for all $x,v \in \R^d$:
\begin{align*}
   &| Q^\alpha_{\beta}(g_{\alpha},g_{\beta}) -Q^\alpha_{\beta}(g'_{\alpha},g_{\beta}')|(x,v) \leq C e^{-2\mu - \gamma (M_{\alpha} +M_{\beta})|v|^2/2} \gamma^{-d/2} (\gamma^{-1/2} + |v|)\\
   & \times \begin{cases}
    (|g_\alpha|_{\alpha,\gamma,\mu} + |g_\alpha'|_{\alpha,\gamma,\mu}) |g_\alpha - g'_\alpha|_{\alpha,\gamma,\mu}, & \alpha = \beta \\
    |g_\alpha|_{\alpha,\gamma,\mu} |g_\beta - g_\beta'|_{\beta,\gamma,\mu} + |g_{\beta}'|_{\beta,\gamma,\mu}  |g_\alpha - g_\alpha'|_{\alpha,\gamma,\mu}, & \alpha \neq \beta
   \end{cases}
\end{align*}

Here, the constant $C$ depends only on the dimension $d$, the masses $M_\alpha,M_\beta$, and the constants $c_1, c_2, a,b$ in the Boltzmann-Grad scaling \eqref{mixed-boltz-grad,1}. Also note that these inequalities are invariant under the action of $S^t$. 
\end{lem}
\begin{proof}
This follows from applications of the triangle inequality and the conservation of energy and momentum for the collision laws stated in Definition \ref{boltz alpha beta kernel}.
\end{proof}
\begin{thm}\label{pde-lwp}
Let $\gamma_0 >0$ and $\mu_0\in \R$ be given. Then there exists $T, \lambda > 0$ depending only on $\gamma_0,\mu_0$ and the Boltzmann-Grad scalings as in \eqref{mixed-boltz-grad,1} such that given any $G_0=(g_{\alpha,0})_{\alpha \in \mathscr{T}} \in X_{\gamma_0, \mu_0}$ with $|G_0|_{\gamma_0, \mu_0} \leq 1$ there exists a unique solution $G \in X_{\bm{\gamma}, \bm{\mu}}([0,T])$ of equation \eqref{pde-soln-weak} that satisfies 
$$
\Vert G \Vert_{\bm{\gamma}, \bm{\mu}} \leq 2 |G_0|_{\gamma_0,\mu_0},
$$
with $\bm{\gamma}(t) = \gamma_0 - \lambda t$ and $\bm{\mu}(t) = \mu_0 - \lambda t$.
 \end{thm}
 \begin{proof} This proof follows again from setting up a contraction mapping using Lemma \ref{nonlinear-estimates}. 
 \end{proof}
\begin{rmk}
The local time of existence $T$ and the constant $\lambda>0$ in Theorem \ref{LWP-BBGKY}, Theorem \ref{LWP-boltz-hierarchy}, and Theorem \ref{pde-lwp} can all be taken to be the same, which we assume throughout the rest of the paper. Crucially, they only depend on the parameters $\gamma_0,\mu_0$, the masses $M_1,M_2$, the Boltzmann-Grad scaling \eqref{mixed-boltz-grad,1}, the dimension $d$, and universal constants. 
\end{rmk}
\section{Statement of the Main Theorem} \label{Statement of the Main Theorem}
 In this section, we define the appropriate notion of convergence and state the main Theorem of the paper. Throughout, we will be using the convention that $\bm{N}=(N_1,N_2) $ and $ \bm{\epsilon} =(\epsilon_1, \epsilon_2)$. Moreover, $\bm{N}$ and $\bm{\epsilon}$ are related by the Boltzmann-Grad scalings \eqref{mixed-boltz-grad,1}. We now give some notation and define approximate Boltzmann initial data. 
\subsection{Approximation of Boltzmann Initial Data}
\begin{definition}[Joint Limit in $N_1,N_2$]
We say a doubly indexed sequence of real numbers $(A_{\bm{N}})_{\bm{N} \in \N_+^2}$ converges to a real number $A$ with respect to the mixed Boltzmann-Grad scalings if the following condition holds: For every $\zeta > 0$, there exists $N_1^*, N_2^* \in \N$ such that for all $N_i \geq N_i^*$, $i=1,2$, which satisfy the scalings \eqref{N1 N2 scaling}, we have 
$$
|A_{\bm{N}} - A | < \zeta.
$$
We denote this type of convergence as $\lim_{\bm{N} \rightarrow \infty} A_{\bm{N}} = A$ or $A_{\bm{N}} \rightarrow A$.
\end{definition}
In analogy with \cite{AP19b}, we introduce an approximating sequence for the BBGKY initial data.

\begin{definition} \label{approx-bbgky} Let $\mu_0 \in \R$ and $\gamma_0 > 0$. Given $F_0 = (f^{(\bm{s})}_0 ) \in \bm{X}_{0, \gamma_0, \mu_0}^\infty$, we say a sequence $F_{\bm{N},0} = (f^{(\bm{s})}_{\bm{N},0})_{\bm{s} \in [\bm{N}]}$ is a BBGKY hierarchy  sequence approximating $F_0$ if the following conditions hold:
\begin{itemize}
    \item We have $\sup_{\bm{N}} \Vert F_{\bm{N},0}\Vert_{\bm{X}_{\bm{\epsilon},\gamma_0,\mu_0}} <\infty$, where the supremum is taken over  $(\bm{N},\bm{\epsilon})$ obeying the mixed Boltzmann-Grad scallings \eqref{mixed-boltz-grad,1}.
    \item We have for each $\bm{s} \in \N_+^2$ and $\sigma > 0$ that $\lim_{\bm{N} \rightarrow \infty} \Vert f^{(\bm{s})}_0 - f^{(\bm{s})}_{\bm{N},0} \Vert_{L^\infty(\Delta_{\bm{s}}(\sigma))} = 0$,
where the set $\Delta_{\bm{s}} ( \sigma)$ is the space of $\sigma$ well separated configurations given by
\begin{equation}
\Delta_{\bm{s}}(\sigma) : = \{ Z_{\bm{s}} : \forall \alpha, \beta, \in \mathscr{T}, \, \, \forall (i,j) \in \mathcal{I}_{(s_1,s_2)}^{(\alpha,\beta)}, \, \, |x^{\alpha}_i - x_j^\beta| > \sigma \} .
\end{equation}
where we recall that $\mathscr{T}$ is the set of types \eqref{set of types} and $\mathcal{I}_{(s_1,s_2)}^{(\alpha,\beta)}$ is the set of interacting pairs \eqref{interacting pairs}.
\end{itemize}
\end{definition}

\begin{rmk}\label{existence of approx data}
For any initial data $F_0 \in \bm{X}_{0,\gamma_0,\mu_0}^\infty$, there always exists at least one approximating BBGKY hierarchy sequence. A simple example of such a sequence is given by $f^{(\bm{s})}_{\bm{N},0} = \1_{\Delta_{\bm{s}}(\max(\epsilon_1,\epsilon_2))}f^{(\bm{s})}_0.$ The uniform upper bound for $F_{\bm{N},0} = (f_{\bm{N},0}^{(\bm{s})})_{\bm{s} \in [\bm{N}]}$ follows by the definition of the function spaces, and the convergence follows from the fact that $\max(\epsilon_1,\epsilon_2)\rightarrow 0$ as $\bm{N} \rightarrow \infty$ in the mixed Boltzmann-Grad scaling \eqref{mixed-boltz-grad,1}. 
\end{rmk}
\begin{rmk}\label{truncation on energy surfaces}
Consider initial data $u_0 = (f_{0,\alpha})_{\alpha\in\mathscr{T}}\in  X_{\gamma_0,\mu_0+1}$ for the Boltzmann equation for mixtures with $\Vert u_0 \Vert_{\gamma_0,\mu_0+1} \leq 1/2$. Assume that $\int_{\R^{2d}} f_{0,\alpha}(x,v)dx dv = 1$ and $f_{0,\alpha} > 0$ a.e. for each $\alpha \in \mathscr{T}$. Define 
$$
F_0 : = (f^{(\bm{s})}_0)_{\bm{s} \in \N^2}, \qquad f^{(\bm{s})}_0 : = f_{0,(1,0)}^{\otimes s_1} \otimes f_{0,(0,1)}^{\otimes s_2}.
$$
This tensor corresponds to an initially chaotic configuration of the Boltzmann hierarchy. From this data, we may form the \emph{conditioned BBGKY initial data} defined by 
\begin{equation} \label{approx bbgky conditioned}
    f_{\bm{N},0}^{(\bm{s})}(Z_{\bm{s}}) : = \mathcal{Z}^{-1}_{\bm{N}} \int_{ \R^{2d|\bm{N} - \bm{s}|}} \1_{\mathcal{D}_{\bm{\epsilon}}^{\bm{N}}} f_{0,(1,0)}^{\otimes N_1} \otimes f_{0,(0,1)}^{\otimes N_2}(Z_{\bm{N}})dx^{(1,0)}_{s_1+1} dv^{(1,0)}_{s_1+1}  \dots dx^{(1,0)}_{N_1}dv^{(1,0)}_{N_1} dx^{(0,1)}_{s_2+1} dv^{(0,1)}_{s_2+1}  \dots dx^{(0,1)}_{N_2}dv^{(0,1)}_{N_2},
\end{equation}
where $\mathcal{Z}_{\bm{N}}$ is a normalization factor given by 
$$
\mathcal{Z}_{\bm{N}} : = \int_{\mathcal{D}_{\bm{\epsilon}}^{\bm{N}}} f_{0,(1,0)}^{\otimes N_1} \otimes f_{0,(0,1)}^{\otimes N_2}(Z_{\bm{N}}) dZ_{\bm{N}}.
$$
This sequence $F_{\bm{N},0} : = ( f_{\bm{N},0}^{(\bm{s})} )_{\bm{s} \in [\bm{N}]}$ is a BBGKY hierarchy sequence approximating $F_0$. In fact, it can be shown (see Chapter 6 of \cite{GSRT13}) that we get explicit rates for any $\bm{s} \in \N^2$ and $\sigma>0$:
\begin{equation}\label{convergence of approx conditioned} 
    \Vert f_{\bm{N},0}^{(\bm{s})} - f_0^{(\bm{s})}\Vert_{L^\infty(\Delta_{\bm{s}}(\sigma))} \leq C|\bm{s}| \max(\epsilon_1,\epsilon_2)^{d}|\bm{N}| \Vert F_0 \Vert_{0,\gamma_0,\mu_0} \leq C_{\bm{s}} \max(\epsilon_1,\epsilon_2)\Vert F_0 \Vert_{0,\gamma_0,\mu_0},
\end{equation}
where in the last inequality we have crucially used the scalings \eqref{mixed-boltz-grad,1}.
\end{rmk}

\subsection{Convergence in Observables}
In this subsection, we define the notion of convergence of observables. As usual, we let $\bm{s} = (s_1, s_2) \in \N_+^2$ and denote the space of test functions $\mathcal{C}_c(\R^{d|\bm{s}|})$ to be the space of continuous compactly supported functions.
\begin{definition}
Let $T > 0$, $\bm{s} \in \N_+^2$, and $f^{(\bm{s})} \in L^\infty([0,T]; L^\infty( \R^{2d|\bm{s}|}))$. Given a test function $\phi_{\bm{s}} \in \mathcal{C}_c(\R^{d|\bm{s}|})$, we define the $\bm{s}$-observable function as 
$$
I_{\phi_{\bm{s}}}(f^{(\bm{s})})(t, X_{\bm{s}}) : =\int_{\R^{d|\bm{s}|}} \phi_{\bm{s}}(V_{\bm{s}})f^{(\bm{s})}(t, X_{\bm{s}},V_{\bm{s}}) dV_{\bm{s}}.
$$
\end{definition}
With this definition in hand, we can now define the notion of convergence in observables. 
\begin{definition}[Convergence in Observables] \label{convergence of observables} Let $T> 0$. Given a sequence $(F_{\bm{N}})_{\bm{N} \in \N_+^2 }$ such that 
$$
F_{\bm{N}} = (f_{\bm{N}}^{(\bm{s})})_{\bm{s} \in [\bm{N}]} \in \prod_{\bm{s} \in [\bm{N}]}L^\infty([0,T]; L^\infty( \R^{2d|\bm{s}|})),\quad \text{and} \quad F = (f^{(\bm{s})})_{\bm{s} \in \N_+^2 }\in \prod_{\bm{s} \in \N_+^2 } L^\infty([0,T]; L^\infty( \R^{2d|\bm{s}|})),
$$
we say that $F_{\bm{N}}$ converges to $F$ if for any $\bm{s} \in \N_+^2$, any $\sigma > 0$, and any $\phi_{\bm{s} } \in \mathcal{C}_c(\R^{d|\bm{s}|})$, we have 
$$
\lim_{\bm{N} \rightarrow \infty} \Vert I_{\phi_{\bm{s}}}(f^{(\bm{s})}_{\bm{N}})(t)- I_{\phi_{\bm{s}}}(f^{(\bm{s})})(t)\Vert_{L^\infty( \Delta^X_{\bm{s}}(\sigma))} = 0 \quad \text{uniformly in $[0,T]$}.
$$
We denote this type of convergence by $F_{\bm{N}} \widetilde{\rightarrow} F$.
\end{definition}
\subsection{Statement of the Main Theorem}
We are now ready to state the main results of this paper. We start with the most general theorem.
\begin{thm}\label{main-thm}
Let $\gamma_0 > 0$, $\mu_0 \in \R$, and $T>0$ be the existence time for the BBGKY and Boltzmann Hierarchies as found in Theorem \ref{LWP-BBGKY} and Theorem \ref{LWP-boltz-hierarchy}. Let $F_0 = (f^{(\bm{s})}_0)_{\bm{s}\in \N_+^2} \in \bm{X}_{0, \gamma_0,\mu_0}^\infty$ be a Boltzmann initial data and $(F_{\bm{N},0})_{\bm{N} \in \N_+^2}$ be an approximating BBGKY hierarchy sequence as given in Definition \ref{approx-bbgky}. Assume the following:
\begin{itemize}
\item For each $\bm{N} \in \N_+^2$, $F_{\bm{N}} \in \bm{X}_{\bm{\epsilon}, \bm{\gamma}, \bm{\mu}}^{\bm{N}}([0,T])$ is a solution to the BBGKY hierarchy with initial data $F_{\bm{N},0}$ (as in Theorem \ref{LWP-BBGKY}). Here, $\bm{\epsilon}$ is related to $\bm{N}$ by \eqref{mixed-boltz-grad,1}, and the functions $\bm{\gamma}, \bm{\mu}: [0,T] \rightarrow \R$ are given as in Theorem \ref{LWP-BBGKY}.
\item $F \in \bm{X}_{0, \bm{\gamma}, \bm{\mu}}^\infty([0,T])$ solves the Boltzmann hierarchy (as in Theorem \ref{LWP-boltz-hierarchy}), where $\bm{\gamma}, \bm{\mu}: [0,T] \rightarrow \R$ are given as in Theorem \ref{LWP-boltz-hierarchy}. 
\item The initial data $F_0$ satisfies the following uniform continuity condition: There exists a constant $C$ such that for all $\zeta >0$, there exists a $q = q(\zeta)$ such that for all $\bm{s} \in \N_+^2$ and $Z_{\bm{s}}, Z_{\bm{s}}' \in \R^{2d|\bm{s}|}$ with $|Z_{\bm{s}} - Z_{\bm{s}}'| < q(\zeta)$, we have 
\begin{equation}\label{unif-cont-cond}
|f^{(\bm{s})}_0(Z_{\bm{s}}) - f^{(\bm{s})}_0(Z_{\bm{s}}')| < C^{|\bm{s}|- 1} \zeta.
\end{equation}
\end{itemize}
Then $F_{\bm{N}}$ converges to $F$ in the sense of observables. 
\end{thm}
\begin{rmk}
Using Definition \ref{convergence of observables}, the convergence in observables of Theorem \ref{main-thm} is equivalent to showing that for every $\sigma > 0$, $\bm{s} \in \N_+^2$, and every $\phi_{\bm{s}} \in \mathcal{C}_c( \R^{d|\bm{s}|}$) we have 
\begin{equation*}
    \lim_{\bm{N}\rightarrow \infty} \Vert  I_{\bm{s}}^\infty - I_{\bm{s}}^{\bm{N}} \Vert_{L^\infty(\Delta^X(\sigma))} = 0, \qquad \text{ uniformly in $[0,T]$,}
\end{equation*} 
where we are defining 
\begin{equation}\label{I infty s}
    I_{\bm{s}}^{\bm{N}}(X_{\bm{s}})  : = \int_{\R^{d|\bm{s}|}} \phi_{\bm{s}}(V_{\bm{s}}) f_{\bm{N}}^{(\bm{s})} (X_{\bm{s}}, V_{\bm{s}})dV_{\bm{s}}, \qquad  I_{\bm{s}}^{\infty}(X_{\bm{s}})  : = \int_{\R^{d|\bm{s}|}} \phi_{\bm{s}}(V_{\bm{s}}) f^{(\bm{s})}(X_{\bm{s}}, V_{\bm{s}}) dV_{\bm{s}}.
\end{equation}
\end{rmk}
\begin{rmk}
The convergence in observables above can be upgraded to another type of weak convergence by a density argument. In particular, we can show using Theorem \ref{main-thm} and the continuity of the BBGKY and Boltzmann hierarchy solution mappings that for any $\phi$ in the exponentially weighted space $L^1_\omega(\R^{2d|\bm{s}|})$ with $\omega(V_{\bm{s}}) = e^{-\bm{\gamma}(T)E(Z_{\bm{s}})}$ that 
\begin{equation}\label{convergence-density-integral}
\int_{\R^{2d|\bm{s}|}} \phi f_{\bm{N}}^{(\bm{s})} d X_{\bm{s}} dV_{\bm{s}} \rightarrow \int_{\R^{2d|\bm{s}|}} \phi f^{(\bm{s})} d X_{\bm{s}} dV_{\bm{s}}
\end{equation}
uniformly for all $t \in [0,T]$. Here we are extending $f_{\bm{N}}^{(\bm{s})}$ to be zero outside of $\mathcal{D}^{\bm{N}}_{\bm{\epsilon}}$. \\
\\
 In statistical mechanics \cite{LL58}, the integrals \eqref{convergence-density-integral} above correspond to the expected value of the function $\phi$ of the system. For example, if $A \subset \R^{2d}$ is measurable and $|A|< \infty$, define the function $\phi(X_{(1,1)}, V_{(1,1)}) = \1_A(X_{(1,1)}) v_1^{(1,0)}$. The integral of $\phi$ against $f_{\bm{N}}^{(1,1)}$ corresponds to the mean velocity of the type $(1,0)$ particle in the region $A$ of space.
\end{rmk} 

We end this section with the theorem showing the \emph{propagation of chaos} and the relation between the finite BBGKY hierarchy and solutions of the Boltzmann equation for mixtures. 
\begin{thm}\label{thm 3}
 Let $\gamma_0 > 0$, $\mu_0 \in \R$ and $T >0$ be the existence time for the BBGKY Hierarchy \eqref{bbgky-weak-soln}, the Boltzmann Hierarchy \eqref{boltz-weak-soln}, and the Boltzmann equation for mixtures \eqref{pde-soln-weak} obtained in Theorems \ref{LWP-BBGKY}-\ref{pde-lwp}. Additionally, define $\bm{\gamma}, \bm{\mu}$ to be be as in those theorems. Also let $u_0=(g_0,h_0) \in X_{\gamma_0,\mu_0+1}$ be H\"older $\mathcal{C}^{0,\lambda}$ initial data with $0<\lambda \leq 1$ and $|u_0|_{\gamma_0,\mu_0+1}\leq 1/2$. Let $(g,h) \in \bm{X}_{\bm{\gamma}, \bm{\mu}}([0,T])$ be the unique solution to the Boltzmann equation \eqref{pde-soln-weak} as given by Theorem \ref{pde-lwp}. Define
$$
F_0 := \left( g_0^{\otimes s_1} \otimes h_0^{\otimes s_2} \right)_{s_1, s_2 \in \N_+}, \quad \text{and} \quad
F : =  \left( g^{\otimes s_1} \otimes h^{\otimes s_2} \right)_{s_1, s_2 \in \N_+}.
$$
Then $F_0 \in \bm{X}_{0,\gamma_0,\mu_0}^\infty$, $F \in \bm{X}_{\bm{\gamma}, \bm{\mu}}^\infty([0,T])$, and $F$ solves the Boltzmann Hierarchy with initial data $F_0$. Moreover, if $(F_{\bm{N},0})_{\bm{N} \in \N_+^2}$ is the conditioned BBGKY hierarchy initial data as given by \eqref{approx bbgky conditioned} and $F_{\bm{N}} = (f_{\bm{N}}^{(\bm{s})})_{\bm{s} \in [\bm{N}]} \in \bm{X}_{\bm{\epsilon}, \bm{\gamma}, \bm{\mu}}^{\bm{N}}([0,T])$ is the unique solution to the BBGKY hierarchy with initial data $F_{\bm{N},0}$, then we have explicit rates of convergence for each $\phi \in \mathcal{C}_c(\R^{d|\bm{s}|})$ given by 
\begin{equation}
    \Vert I_{\phi}(f^{(\bm{s})}_{\bm{N}})(t) - I_{\phi}( g^{\otimes s_1} \otimes h^{\otimes s_2} )(t)  \Vert_{L^\infty(\Delta_{\bm{s}}^X(\sigma))} = O( \epsilon^r ), \qquad \epsilon : = \max_{\alpha \in \mathscr{T}}\epsilon_\alpha 
\end{equation}
for any $ r < \lambda$ and uniformly in $t \in [0,T]$.
\end{thm}

\section{Reduction to Term-by-Term Convergence}\label{Reduction to Term-by-Term Convergence}
\subsection{Initial Expansion and Notation} \label{Initial Expansion and Notation}
 In this section, we show that solutions to the BBGKY and Boltzmann Hierarchies can be expanded into a series depending only on initial data. In the BBGKY hierarchy, this series expansion terminates. In the Boltzmann hierarchy, the series expansion converges in the sense of observables. For notational convenience, we will continue to use our convention that $\bm{N} = (N_{(1,0)}, N_{(0,1)}) = (N_1, N_2)$ in addition to $\bm{s} = (s_{(1,0)}, s_{(0,1)})$. We note that the combinatorial complexity of this series expansion motivated us to introduce the vector notation for particle numbers and particle types, which in turn simplifies the record-keeping and proofs. In particular, the notation allows the quartic tree generated by above series expansions to be dealt with. \\

Recall that $F_{\bm{N}} = (f^{(\bm{s})}_{\bm{N}})_{\bm{s} \in [\bm{N}]}$ with initial data $F_{\bm{N},0} =(f^{(\bm{s})}_{\bm{N},0})_{\bm{s} \in [\bm{N}]}$ (where $[\bm{N}]$ is given by \eqref{n-set-def}) solves the BBGKY hierarchy \eqref{bbgky-weak-soln} if it satisfies, for every $\bm{s}\in [\bm{N}]$,
\begin{align*}
f_{\bm{N}}^{(\bm{s})}(t) = T_{\bm{s},\bm{\epsilon}}^t f^{(\bm{s})}_{\bm{N},0} + \sum_{\alpha_1,\beta_1 \in \mathscr{T}} \int_0^t T_{\bm{s},\bm{\epsilon}}^{t- \tau_1} \mathcal{C}^{\alpha_1}_{\bm{s},\bm{s} + \beta_1}f^{(\bm{s} + \beta_1)}_{\bm{N}}(\tau_1) d\tau_1.
\end{align*}
Here, we recall that $T_{\bm{s},\bm{\epsilon}}$ is the particle flow operator given in \eqref{T-def}, $\mathscr{T}$ is the set of types given in \eqref{set of types}, and $\mathcal{C}^\alpha_{\bm{s},\bm{s}+ \beta}$ is the collisional operator given in \eqref{bbgky-collision-operator}. Using this representation for $f^{(\bm{s}+\beta)}_{\bm{N}}$ for each $\beta \in \mathscr{T}$, we may expand the above expression in terms of the initial data $f_{\bm{N},0}^{(\bm{s}+\beta)}$ and a time dependent remainder. For clarity, we do this for the first term below:
\begin{align*}
f_{\bm{N}}^{(\bm{s})}(t) &= T_{\bm{s},\bm{\epsilon}}^t f^{(\bm{s})}_{\bm{N},0} + \sum_{\alpha,\beta \in \mathscr{T}} \int_0^t T_{\bm{s},\bm{\epsilon}}^{t- \tau} \mathcal{C}^{\alpha}_{\bm{s},\bm{s} + \beta} T_{\bm{s}+\beta,\bm{\epsilon}}^{\tau}f^{(\bm{s} + \beta)}_{\bm{N},0} d\tau\\
& + \sum_{\alpha_1, \beta_1 \in \mathscr{T}}\sum_{\alpha_2, \beta_2 \in \mathscr{T}} \int_0^t \int_0^{\tau_1}T_{\bm{s},\bm{\epsilon}}^{t- \tau_1} \mathcal{C}^{\alpha_1}_{\bm{s},\bm{s} + \beta_1} T_{\bm{s}+\beta_1,\bm{\epsilon}}^{\tau_1-\tau_2}\mathcal{C}^{\alpha_2}_{\bm{s}+\beta_1,\bm{s} + \beta_1 + \beta_2}f^{(\bm{s} + \beta_1 +\beta_2)}_{\bm{N}}(\tau_2)  d\tau_2 d\tau_1
\end{align*}
where we have used the continuity and linearity of the operators to commute the integrals and operators. This can be simplified notationally. Let us define 
\begin{equation}\label{collisional history}
S_k : = \{ \bm{\beta} = ( \beta_1, \dots, \beta_k) | \, \forall i, \, \beta_i \in \mathscr{T}\}.
\end{equation}
Then, for each $1 \leq l \leq k$ and $\bm{\beta} \in S_k$, define the quantities
\begin{equation}\label{beta tilde}
\tilde{\bm{\beta}}_l := \sum_{i = 1}^l \beta_i \in \N^2, \quad \text{and} \quad \tilde{\bm{\beta}}_l  = ( \tilde{\bm{\beta}}_l^{(1,0)}, \tilde{\bm{\beta}}_l ^{(0,1)}).
\end{equation}
Also, define the sequence of times 
\begin{equation}\label{time seq}
\mathcal{T}_k(t) : = \{ (t_1, t_2, \dots, t_k) \in \R_+^k : \, 0 \leq t_k \leq \dots \leq t_1 \leq t \}.
\end{equation}
Now for $\bm{\alpha},\bm{\beta} \in S_k$ introduce the function 
\begin{equation}\label{fN,k}
\begin{aligned}
f_{\bm{N}, (\bm{\alpha},\bm{\beta})  }^{(\bm{s})}(t) := \int_{\mathcal{T}_k(t)} T_{\bm{s},\bm{\epsilon}}^{t-\tau_1}\mathcal{C}_{\bm{s},\bm{s}+\tilde{\bm{\beta}}_1 }^{\alpha_1}T_{\bm{s}+\tilde{\bm{\beta}}_1,\bm{\epsilon}}^{\tau_1-\tau_2}\mathcal{C}_{\bm{s}+\tilde{\bm{\beta}}_1, \bm{s}+\tilde{\bm{\beta}}_2}^{\alpha_2} \cdots T_{\bm{s}+\tilde{\bm{\beta}}_{k-1},\bm{\epsilon}}^{\tau_{k-1} - \tau_k}  \mathcal{C}_{\bm{s} + \tilde{\bm{\beta}}_{k-1}, \bm{s} + \tilde{\bm{\beta}}_{k}}^{\alpha_k} T^{\tau_k}_{\bm{s}+ \tilde{\bm{\beta}}_k,\bm{\epsilon}} f^{(\bm{s} + \tilde{\bm{\beta}}_{k})}_{\bm{N},0} d\tau_k...d\tau_1.
\end{aligned}
\end{equation}
For $k = 0$, define $S_0 : = \{ (0,0) \}$, and write for $\bm{\alpha},\bm{\beta} \in S_0$
\begin{equation}\label{fN,0}
f_{\bm{N}, (\bm{\alpha},\bm{\beta})  }^{(\bm{s})}(t) : = T_{\bm{s},\bm{\epsilon}}^t f^{(\bm{s})}_{\bm{N},0}. 
\end{equation}
These are the terms in the expansion of $f^{(\bm{s})}_{\bm{N}}$ which only depend on initial data. The rest of the terms are considered as a remainder, which we define as
\begin{align*}
 R_{\bm{N},(\bm{\alpha},\bm{\beta}) }^{(\bm{s})}(t) :=   \int_{\mathcal{T}_k(t)} T_{\bm{s}}^{t-\tau_1}\mathcal{C}_{\bm{s},\bm{s}+\tilde{\bm{\beta}}_1 }^{\alpha_1}T_{\bm{s}+\tilde{\bm{\beta}}_1}^{\tau_1-\tau_2}\mathcal{C}_{\bm{s}+\tilde{\bm{\beta}}_1, \bm{s}+\tilde{\bm{\beta}}_2}^{\alpha_2} \cdots T_{\bm{s}+\tilde{\bm{\beta}}_{k-1}}^{\tau_{k-1} - \tau_k}  \mathcal{C}_{\bm{s} + \tilde{\bm{\beta}}_{k-1}, \bm{s} + \tilde{\bm{\beta}}_{k}}^{\alpha_k} f^{(\bm{s} + \tilde{\bm{\beta}}_{k})}_{\bm{N}}(\tau_k)d\tau_k...d\tau_1.
\end{align*}
By induction, the solution $F_{\bm{N}} = (f^{(\bm{s})}_{\bm{N}})_{\bm{s} \in [\bm{N}]}$ of the BBGKY hierarchy \eqref{bbgky-weak-soln} can be written as:
\begin{equation}
\boxed{ f^{(\bm{s})}_{\bm{N}}(t) = \sum_{k=0}^n \sum_{\bm{\alpha},\bm{\beta}  \in S_{k} } f^{(\bm{s})}_{\bm{N}, (\bm{\alpha},\bm{\beta})  }(t) + \sum_{\bm{\alpha},\bm{\beta} \in S_{n+1}} R_{\bm{N},(\bm{\alpha},\bm{\beta})}^{(\bm{s})}(t),\qquad \text{ where $n < \min_{\alpha \in \mathscr{T}} (N_\alpha - s_\alpha$.}}
\end{equation}

We can proceed in a similar manner for the Boltzmann hierarchy \eqref{boltz-weak-soln}, where $F = (f^{(\bm{s})})_{\bm{s}\in \N_+^2}$ a solution to the Boltzmann hierarchy. Define the following expression for $ \bm{\alpha},\bm{\beta} \in S_k$ 
\begin{equation}\label{f,k}
\begin{aligned}
f_{ (\bm{\alpha},\bm{\beta})  }^{(\bm{s})}(t) :=  \int_{\mathcal{T}_k(t)} S_{\bm{s}}^{t-\tau_1}\mathscr{C}_{\bm{s},\bm{s}+\tilde{\bm{\beta}}_1 }^{\alpha_1}S_{\bm{s}+\tilde{\bm{\beta}}_1}^{\tau_1-\tau_2}\mathscr{C}_{\bm{s}+\tilde{\bm{\beta}}_1, \bm{s}+\tilde{\bm{\beta}}_2}^{\alpha_2} \cdots S_{\bm{s}+\tilde{\bm{\beta}}_{k-1}}^{\tau_{k-1} - \tau_k}  \mathscr{C}_{\bm{s} + \tilde{\bm{\beta}}_{k-1}, \bm{s} + \tilde{\bm{\beta}}_{k}}^{\alpha_k} S^{\tau_k}_{\bm{s}+ \tilde{\bm{\beta}}_k} f^{(\bm{s} + \tilde{\bm{\beta}}_{k})}_{0} d\tau_k...d\tau_1.
\end{aligned}
\end{equation}
Here, we recall that $S^t_{\bm{s}}$ is free transport as given in \eqref{S-def}, and $\mathscr{C}_{\bm{s}, \bm{s}+\beta}^\alpha$ is the collision operator as given in \eqref{boltz alpha beta kernel}. As above, we define for $\bm{\alpha},\bm{\beta} \in S_0$
\begin{equation}\label{f,0}
f_{(\bm{\alpha},\bm{\beta})  }^{(\bm{s})}(t) : = S_{\bm{s}}^t f^{(\bm{s})}_0
\end{equation}
Next, define the remainder term: 
\begin{align*}
R_{(\bm{\alpha},\bm{\beta}) }^{(\bm{s})}(t) & := \int_{\mathcal{T}_k(t)} S_{\bm{s}}^{t-\tau_1}\mathscr{C}_{\bm{s},\bm{s}+\tilde{\bm{\beta}}_1 }^{\alpha_1}S_{\bm{s}+\tilde{\bm{\beta}}_1}^{\tau_1-\tau_2}\mathscr{C}_{\bm{s}+\tilde{\bm{\beta}}_1, \bm{s}+\tilde{\bm{\beta}}_2}^{\alpha_2} \cdots S_{\bm{s}+\tilde{\bm{\beta}}_{k-1}}^{\tau_{k-1} - \tau_k}  \mathscr{C}_{\bm{s} + \tilde{\bm{\beta}}_{k-1}, \bm{s} + \tilde{\bm{\beta}}_{k}}^{\alpha_k} f^{(\bm{s} + \tilde{\bm{\beta}}_{k})}(\tau_k)d\tau_k ...d\tau_1.
\end{align*}
Then, as above, we have for any $n \in \N_+$
\begin{equation}
\boxed{f^{(\bm{s})}(t) = \sum_{k=0}^n \sum_{\bm{\alpha},\bm{\beta} \in S_{k} } f_{(\bm{\alpha},\bm{\beta})  }^{(\bm{s})}(t) + \sum_{\bm{\alpha},\bm{\beta} \in S_{n+1}}R_{(\bm{\alpha},\bm{\beta}) }^{(\bm{s})}(t)}.
\end{equation}
\subsection{Reduction to Finitely Many Terms}\label{Initial Truncation}
Here we reduce the convergence proof to term by term
convergence of terms with bounded energy and separated collision times. Recalling \eqref{energy}, given $R>0$, $\bm{\ell} \in\mathbb{N}_+^2$, and $\alpha,\beta \in \mathscr{T}$ types as given in \eqref{set of types}, we define the energy truncated operators
\begin{equation}\label{velocity truncation of operators}
\mathcal{C}_{\bm{\ell},\bm{\ell}+\beta,R}^{\alpha}g_{\bm{N}}^{(\bm{\ell}+\beta)}:=\mathcal{C}_{\bm{\ell},\bm{\ell}+\beta}^\alpha \left(g_{\bm{N}}^{(\bm{\ell}+\beta)}\mathds{1}_{[B_R^{d(|\bm{\ell}+\beta|)}]}\right),\quad
\mathscr{C}_{\bm{\ell},\bm{\ell}+\beta,R}^{\alpha}g^{(\bm{\ell}+\beta)}:=\mathscr{C}_{\bm{\ell},\bm{\ell}+\beta}^\alpha \left(g^{(\bm{\ell}+\beta)}\mathds{1}_{[B_R^{d(|\bm{\ell}+\beta|)}]}\right).
\end{equation}
Consider  $\delta>0$. Given $t\geq 0$ and $k\in\mathbb{N}$, we define   the separated collision times
\begin{equation}\label{separated collision times}
\mathcal{T}_{k,\delta}(t):=\left\{(t_1,...,t_k)\in\mathcal{T}_k(t):\quad 0\leq t_{i+1}\leq t_i-\delta,\quad\forall i\in [0,k]\right\},\quad t_{k+1}:=0,\text{ }t_0:=t.
\end{equation}
For the BBGKY hierarchy, we define for $k\in\mathbb{N}_+$ and $\bm{\alpha},\bm{\beta} \in S_k$,
\begin{equation}\label{fNRdk}
\begin{aligned}
f_{\bm{N}, (\bm{\alpha},\bm{\beta})  ,R,\delta}^{(\bm{s})}(t) :=  \int_{\mathcal{T}_{k,\delta}(t)} T_{\bm{s},\bm{\epsilon}}^{t-\tau_1}\mathcal{C}_{\bm{s},\bm{s}+\tilde{\bm{\beta}}_1 ,R}^{\alpha_1}T_{\bm{s}+\tilde{\bm{\beta}}_1,\bm{\epsilon}}^{\tau_1-\tau_2}\mathcal{C}_{\bm{s}+\tilde{\bm{\beta}}_1, \bm{s}+\tilde{\bm{\beta}}_2,R}^{\alpha_2} \cdots T_{\bm{s}+\tilde{\bm{\beta}}_{k-1},\bm{\epsilon}}^{\tau_{k-1} - \tau_k}  \mathcal{C}_{\bm{s} + \tilde{\bm{\beta}}_{k-1}, \bm{s} + \tilde{\bm{\beta}}_{k},R}^{\alpha_k} T^{\tau_k}_{\bm{s}+ \tilde{\bm{\beta}}_k,\bm{\epsilon}} f^{(\bm{s} + \tilde{\bm{\beta}}_{k})}_{\bm{N},0} d\tau_k...d\tau_1,
\end{aligned}
\end{equation}
and for $k=0$, we define
$
f_{\bm{N},(0,0),R,\delta}^{\bm{s}}(t,Z_{\bm{s}}):=T_{\bm{s}}^t\left(f_{\bm{N},0}^{(\bm{s})}\mathds{1}_{[B_R^{d(|\bm{s}|)}]}\right)(Z_{\bm{s}}).$ 

For the Boltzmann hierarchy, we define for $k\in\mathbb{N}_+$ and $\bm{\alpha},\bm{\beta} \in S_k$,
\begin{equation}\label{fRdk}
\begin{aligned}
f_{ (\bm{\alpha},\bm{\beta})  ,R,\delta}^{(\bm{s})}(t) :=  \int_{\mathcal{T}_{k,\delta}(t)} S_{\bm{s}}^{t-\tau_1}\mathscr{C}_{\bm{s},\bm{s}+\tilde{\bm{\beta}}_1,R}^{\alpha_1}S_{\bm{s}+\tilde{\bm{\beta}}_1}^{\tau_1-\tau_2}\mathscr{C}_{\bm{s}+\tilde{\bm{\beta}}_1, \bm{s}+\tilde{\bm{\beta}}_2,R}^{\alpha_2} \cdots S_{\bm{s}+\tilde{\bm{\beta}}_{k-1}}^{\tau_{k-1} - \tau_k}  \mathscr{C}_{\bm{s} + \tilde{\bm{\beta}}_{k-1}, \bm{s} + \tilde{\bm{\beta}}_{k},R}^{\alpha_k} S^{\tau_k}_{\bm{s}+ \tilde{\bm{\beta}}_k} f^{(\bm{s} + \tilde{\bm{\beta}}_{k})}_{0} d\tau_k...d\tau_1,
\end{aligned}
\end{equation}
and for $k=0$, we define
$
f_{(0,0),R,\delta}^{\bm{s}}(t,Z_{\bm{s}}):=S_{\bm{s}}^t\left(f_{0}\mathds{1}_{[B_R^{d(|\bm{s}|)}]}\right)(Z_{\bm{s}}).$

Given $\phi_{\bm{s}}\in \mathcal{C}_c(\mathbb{R}^{d|\bm{s}|})$, $k\in\mathbb{N}$, and $\bm{\alpha},\bm{\beta} \in S_k$, let us write
\begin{align}
I_{\bm{s},k,R,\delta}^{\bm{N}}(t)(X_{\bm{s}}):=\sum_{\bm{\alpha},\bm{\beta} \in S_k} \int_{B_R^{d|\bm{s}|}}\phi_{\bm{s}}(V_{\bm{s}})f_{\bm{N},(\bm{\alpha}, \bm{\beta}),R,\delta}^{(\bm{s})}(t,X_{\bm{s}},V_{\bm{s}})\,dV_{\bm{s}},\label{bbgky truncated time}
\end{align}
\begin{align}
I_{\bm{s},k,R,\delta}^\infty(t)(X_{\bm{s}}):=\sum_{\bm{\alpha},\bm{\beta} \in S_k} \int_{B_R^{d|\bm{s}|}}\phi_{\bm{s}}(V_{\bm{s}})f_{(\bm{\alpha},\bm{\beta}),R,\delta}^{(\bm{s})}(t,X_{\bm{s}},V_{\bm{s}})\,dV_{\bm{s}}.\label{boltzmann truncated time}
\end{align}

The following estimates show that the observables $I_{\bm{s}}^{\bm{N}}$, $I_{\bm{s}}^\infty$ defined in \eqref{I infty s} can be approximated by the truncated observables \eqref{bbgky truncated time},\eqref{boltzmann truncated time} for small $\delta$ and large $n$ and $R$. 
\begin{prop}\label{reduction}
For any $\bm{s}\in \N_{+}^2$, $n\in\mathbb{N}$, $R>1$, $\delta>0$ and $t\in[0,T]$, the following estimates hold:
\begin{equation*}
\|I_{\bm{s}}^N(t)-\sum_{k=0}^n I_{\bm{s},k,R,\delta}^{\bm{N}}(t)\|_{L^\infty_{X_{\bm{s}}}}\leq C_{\bm{s},\gamma_0,\mu_0,T}\|\phi_{\bm{s}}\|_{L^\infty_{V_{\bm{s}}}}\left(2^{-n}+e^{-\frac{\gamma_0}{3}R^2}+\delta C_{d,\bm{s},\gamma_0,\mu_0,T}^n\right)\|F_{\bm{N},0}\|_{\bm{N},\gamma_0,\mu_0},
\end{equation*}
\begin{equation*}
\|I_{\bm{s}}^\infty(t)-\sum_{k=0}^n I_{\bm{s},k,R,\delta}^\infty(t)\|_{L^\infty_{X_{\bm{s}}}}\leq C_{{\bm{s}},\gamma_0,\mu_0,T}\|\phi_{\bm{s}}\|_{L^\infty_{V_{\bm{s}}}}\left(2^{-n}+e^{-\frac{\gamma_0}{3}R^2}+\delta C_{d,{\bm{s}},\gamma_0,\mu_0,T}^n\right)\|F_{0}\|_{\infty,\gamma_0,\mu_0}.
\end{equation*}
\end{prop}
\begin{proof}
For the proof, one needs to use the a-priori bounds of Section \ref{Local Well-Posedness} to perform successive reductions to finitely many terms, bounded energies, and separated collision times respectively, and connect these estimates via the triangle inequality. The proof is similar to the corresponding reductions in the binary case \cite{GSRT13}. For more details on the strategy of the proof, see \cite{thesis} where related reductions were made for the case of ternary interactions.
\end{proof}

\section{Good Configurations and Stability}\label{Good Configurations and Stability}
\subsection{Construction of Good Sets and Notation}
 Let $\bm{m}=(m_1,m_2) \in \N_+^2$ be the number of particles of each type, $\bm{\epsilon} = (\epsilon_1, \epsilon_2) \in (0,\infty)^2$ be their diameters, and recall from Section \ref{mixed-particle-notation-definitions} that we denote the vectors of all positions and velocities by
$$
X_{\bm{m}} = \left( X_{m_1}^{(1,0)}, X_{m_2}^{(0,1)} \right), \quad \text{and} \quad V_{\bm{m}} = \left( V_{m_1}^{(1,0)}, V_{m_2}^{(0,1)} \right).
$$
The full phase space vector is similarly given by $Z_{\bm{m}} : = (X_{\bm{m}}, V_{\bm{m}})$. 
Recall the definition of the impact operator $T$ given in Definition \ref{T-collision}, and the set $\mathcal{I}_{(i,j)}^{(\alpha,\beta)}$ given by \eqref{interacting pairs}. For each $\theta > 0$ and $\bm{m} \in \N_+^2$, define a $\theta$-well separated configuration to be 
\begin{equation}\label{separated space data}
\Delta_{\bm{m}}(\theta) : = \{ Z_{\bm{m}} : \forall \alpha, \beta \in \mathscr{T}, \, \, \forall (i,j) \in \mathcal{I}_{\bm{m}}^{(\alpha,\beta)}, \, \, |x^{\alpha}_i - x_j^\beta| > \theta \} .
\end{equation}
We also define spacial components of this set as
\begin{equation}\label{separated space data 1,2}
\Delta_{\bm{m}}^X(\theta) : = \{ X_{\bm{m}} : \forall \alpha, \beta \in \mathscr{T}, \, \, \forall (i,j) \in \mathcal{I}_{\bm{m}}^{(\alpha,\beta)}, \, \, |x^{\alpha}_i - x_j^\beta| > \theta \} .
\end{equation}
Now let $Z_{\bm{m}}(t)$ be the backwards particle flow given by 
\begin{equation}\label{backward flow}
    Z_{\bm{m}}(t) = \Psi_{\bm{m}, \bm{\epsilon}}^{-t}(Z_{\bm{m}}),
\end{equation}
where $\Psi_{\bm{m},\bm{\epsilon}}$ is given in Theorem \ref{global-flow-thm}. Define a $(\theta, t_0)$-good configuration to be 
\begin{equation}\label{good configurations}
G_{\bm{m}}(\theta, t_0) : = \{ Z_{\bm{m}} : Z_{\bm{m}}(t) \in \Delta_{\bm{m}}(\theta) ,\, \, \forall t\geq t_0 \}.
\end{equation}
Define additionally 
\begin{equation} \label{S x Br plus}
(\mathbb{S}_1^{d-1} \times B_{R}^d)^+( v) : = \{ (\omega_1 , v_1)\in\mathbb{S}_1^{d-1} \times B_{R}^d : \, \, \omega_1 \cdot( v_1 - v)> 0 \}.
\end{equation}
The ``plus" is meant to indicate post-collisional configurations. Now, as in \cite{Lan75},\cite{GSRT13}, we wish to exclude trajectories on which recollisions occur in the backwards flow. The strategy is to construct a ``bad set"  whose complement is exactly the set of initial configurations of an adjoined particle which do not run into collisions under backwards flow. \\
\\
We will now fix parameters $\gamma, \epsilon_0, R, \eta, \delta$ (to be chosen later) which are related by 
\begin{equation}\label{parameter-relations}
\max(\epsilon_{1},\epsilon_2) \ll \gamma  \ll \epsilon_0 \ll \eta \delta, \qquad R\gamma \ll \eta \epsilon_0.
\end{equation}
\begin{prop}\label{bad-sets} Fix $\bm{m} = (m_{(1,0)}, m_{(0,1)}) \in \N_+^2$ and recall that $\mathscr{T}$ is the set given by \eqref{set of types}. Let $\overline{Z}_{\bm{m}}=(\overline{X}_{\bm{m}},\overline{V}_{\bm{m}}) \in G_{\bm{m}}(\epsilon_0, 0)$ and assume that $E(\overline{Z}_{\bm{m}}) \leq R^2$, where $E(\cdot)$ is given by \eqref{energy}. For each $\alpha \in \mathscr{T}$ and $\ell \in \{1, \dots, m_{\alpha}\}$, there exists a \emph{bad set} $\mathcal{B}_{\ell, \alpha}( \overline{Z}_{\bm{m}}) \subset (\mathbb{S}_1^{d-1} \times B_{R}^d)^+( \overline{v}^{\alpha}_{\ell})$ such that for any $X_{\bm{m}} \in B_{\gamma/2}^{d|\bm{m}| }(\overline{X}_{\bm{m}})$ we have the following: 
\begin{enumerate}
\item For all $\beta \in \mathscr{T}$ and $(\omega_1, v^{\beta}_{m_\beta + 1}) \in (\mathbb{S}_1^{d-1} \times B_{R}^d)^+( \overline{v}^{\alpha}_{\ell}) \setminus \mathcal{B}_{\ell, \alpha}( \overline{Z}_{\bm{m}})$, we have 
	\begin{enumerate}
	\item $Z_{\bm{m}+\beta} ( t) \in \mathring{\mathcal{D}}_{\bm{\epsilon}}^{\bm{m} + \beta}$ for all $t\geq 0$.  
	\item $Z_{\bm{m}+\beta} \in G_{\bm{m} + \beta} (\epsilon_0/2, \delta)$
	\item $\overline{Z}_{\bm{m}+\beta} \in G_{\bm{m} + \beta} (\epsilon_0, \delta)$
	\end{enumerate}
where we have defined $Z_{\bm{m}+\beta}(t)$ via \eqref{backward flow} and
$$
Z_{\bm{m}+\beta} : = \begin{cases}
(X_{m_{(1,0)}}^{(1,0)}, x_{m_{(1,0)} + 1}^{(1,0)}, X_{m_{(0,1)}}^{(0,1)}, \overline{V}_{m_{(1,0)}}^{(1,0)}, v_{m_{(1,0)} + 1}^{(1,0)}, \overline{V}_{m_{(0,1)}}^{(0,1)}), &  \beta = (1,0)\\
(X_{m_{(1,0)}}^{(1,0)}, X_{m_{(0,1)}}^{(0,1)}, x_{m_{(0,1)} + 1}^{(0,1)}, \overline{V}_{m_{(1,0)}}^{(1,0)}, \overline{V}_{m_{(0,1)}}^{(0,1)},v_{m_{(0,1)} + 1}^{(0,1)}), & \beta = (0,1)
\end{cases}
$$ 
$$
\overline{Z}_{\bm{m}+\beta} : = \begin{cases}
(\overline{X}_{m_{(1,0)}}^{(1,0)}, x_{m_{(1,0)} + 1}^{(1,0)}, \overline{X}_{m_{(0,1)}}^{(0,1)}, \overline{V}_{m_{(1,0)}}^{(1,0)}, v_{m_{(1,0)} + 1}^{(1,0)}, \overline{V}_{m_{(0,1)}}^{(0,1)}), &  \beta = (1,0)\\
(\overline{X}_{m_{(1,0)}}^{(1,0)}, \overline{X}_{m_{(0,1)}}^{(0,1)}, x_{m_{(0,1)} + 1}^{(0,1)}, \overline{V}_{m_{(1,0)}}^{(1,0)}, \overline{V}_{m_{(0,1)}}^{(0,1)},v_{m_{(0,1)} + 1}^{(0,1)}), & \beta = (0,1)
\end{cases}
$$ 
$$
x_{m_\beta + 1}^\beta = x_\ell^\alpha - \epsilon_{(\alpha,\beta)} \omega_1.\\
$$
\\
\item For all $\beta \in \mathscr{T}$ and $(\omega_1, v^{\beta}_{m_\beta + 1}) \in (\mathbb{S}_1^{d-1} \times B_{R}^d)^+( \overline{v}^{\alpha}_{\ell}) \setminus \mathcal{B}_{\ell, \alpha}( \overline{Z}_{\bm{m}})$, we have 
	\begin{enumerate}
	\item $ Z_{\bm{m}+\beta}^* ( t) \in \mathring{\mathcal{D}}_{\bm{\epsilon}}^{\bm{m} + \beta}$ for all $t\geq 0$,
	\item $ Z_{\bm{m}+\beta}^* \in G_{\bm{m} + \beta} (\epsilon_0/2, \delta)$,
	\item $\overline{ Z}_{\bm{m}+\beta}^* \in G_{\bm{m} + \beta} (\epsilon_0, \delta)$,
	\end{enumerate}
	where we have defined the backwards flow via \eqref{backward flow}, $Z_{\bm{m}+\beta}^* : = T(Z_{\bm{m}+\beta}^+)$, $\overline{Z}_{\bm{m}+\beta}^* : = T(\overline{Z}_{\bm{m}+\beta}^+)$, where $T$ is given by Definition \ref{T-collision}, and
$$
Z_{\bm{m}+\beta}^+: = \begin{cases}
(X_{m_{(1,0)}}^{(1,0)}, (x_{m_{(1,0)} + 1}^{(1,0)})^+, X_{m_{(0,1)}}^{(0,1)}, \overline{V}_{m_{(1,0)}}^{(1,0)}, v_{m_{(1,0)}+ 1}^{(1,0)}, \overline{V}_{m_{(0,1)}}^{(0,1)}), &  \beta = (1,0)\\
(X_{m_{(1,0)}}^{(1,0)}, X_{m_{(0,1)}}^{(0,1)}, (x_{m_{(0,1)} + 1}^{(0,1)})^+, \overline{V}_{m_{(1,0)}}^{(1,0)}, \overline{V}_{m_{(0,1)}}^{(0,1)},v_{m_{(0,1)} + 1}^{(0,1)}), & \beta = (0,1)
\end{cases}
$$ 
$$
\overline{Z}_{\bm{m}+\beta}^+: = \begin{cases}
(\overline{X}_{m_{(1,0)}}^{(1,0)}, (x_{m_{(1,0)} + 1}^{(1,0)})^+, \overline{X}_{m_{(0,1)}}^{(0,1)}, \overline{V}_{m_{(1,0)}}^{(1,0)}, v_{m_{(1,0)} + 1}^{(1,0)}, \overline{V}_{m_{(0,1)}}^{(0,1)}), &  \beta = (1,0)\\
(\overline{X}_{m_{(1,0)}}^{(1,0)}, \overline{X}_{m_{(0,1)}}^{(0,1)}, (x_{m_{(0,1)} + 1}^{(0,1)})^+, \overline{V}_{m_{(1,0)}}^{(1,0)}, \overline{V}_{m_{(0,1)}}^{(0,1)},v_{m_{(0,1)} + 1}^{(0,1)}), & \beta = (0,1)
\end{cases}
$$ 
$$
(x_{m_\beta + 1}^\beta )^+ = x_\ell^\alpha + \epsilon_{(\alpha,\beta)} \omega_1.
$$
\end{enumerate}
\end{prop}

We now require a few lemmas in order to prove this proposition. The first lemma shows that in the simple case of two particles, one can exclude a small cylinder to obtain a non-collisional trajectory. \\
\\
\begin{lem}  \label{gallagher-cylinder}  Let the parameters $\gamma, \epsilon_0, R, \eta, \delta$ be related by \eqref{parameter-relations}. Let $\overline{y}_1, \overline{y}_2 \in \R^d$ with $|\overline{y}_1 - \overline{y}_2| > \epsilon_0$. Also let $v_1 \in B_R^d$. Then there exists a $d$-dimensional cylinder $K_\eta^d$ such that for any $v_2 \in B_R^d \setminus K_\eta^d$, $y_1 \in B_\gamma^d(\overline{y}_1)$, and $y_2 \in B_\gamma^d(\overline{y}_2)$, we have
\begin{enumerate}
\item For all $t\geq 0$, $|(y_1 - t v_1) - (y_2 - t v_2) | \geq \epsilon_0/2$. 

\item For all $t \geq \delta $, $|(y_1 - t v_1) - (y_2 - t v_2) | \geq \epsilon_0$.
\end{enumerate}
\end{lem}
\begin{proof}
The proof of this lemma can be found in Chapter 12 of \cite{GSRT13}.
\end{proof}
For the next few lemmas, we construct a "bad set" such that outside of this set, the new particle is non-collisional for its whole trajectory. We compare each pair of particles considering three cases:
\begin{itemize}
\item[1.] Compare $z^{\sigma}_i, z^{\sigma'}_j$ for $i \in \{1, \dots , m_\sigma\}$ and $j \in \{1, \dots, m_{\sigma'}\}$. We will call this the \textit{existing particle case}. 
\item[2.] Compare the adjoined particle $z_{m_\beta+1}^\beta$ with $z_{i}^\sigma$ for $(i,\sigma) \neq (\ell,\alpha)$. We will call this the \textit{existing particle and adjoined particle case}. 
\item[3.] Compare the adjoined particle $z_{m_\beta+1}^\beta$ with the particle $z_{\ell}^\alpha$ which it is adjoined close to. This case is the only case which uses the collisional laws, and hence we call it the \textit{collisional law case}.
\end{itemize}
For the first lemma, we find a bad set $\mathcal{B}_{\ell,\alpha}^{0, -}(\overline{Z}_{\bm{m}})$ such that for all time $t \geq 0$, the initial configurations in the complement do not encounter collisions.
\begin{lem} \label{2,0,-}
Let us be in the same scenario as in Proposition \ref{bad-sets}. Then there exists a set $\mathcal{B}_{\ell, \alpha}^{0, -}(\overline{Z}_{\bm{m}})$ such that for all $(\omega_1, v_{m_\beta +1}^\beta) \in (\mathbb{S}_1^{d-1} \times B_R^d)^+(\overline{v}_{\ell}^\alpha)\setminus \mathcal{B}_{\ell,\alpha}^{0, -}(\overline{Z}_{\bm{m}})$, we have $Z_{\bm{m}+\beta}(t) \in \mathring{\mathcal{D}}_{\bm{\epsilon}}^{\bm{m} + \beta},$ for all $t\geq0$.
\end{lem}
\begin{proof} This proof follows by applying the same arguments as found in \cite{GSRT13},\cite{AP19a},\cite{AP19b} to the three cases above. 

\begin{enumerate}
\item \textit{(Existing Particle Case)} We first show that for $\sigma, \sigma' \in \mathscr{T}$, $i \in \{1, \dots, m_\sigma\}$, and $j \in \{1,\dots, m_{\sigma'}\}$, we have $|x_i^{\sigma}(t) -x_j^{\sigma'}(t)| > \epsilon_{(\sigma, \sigma')}$ for all $t \geq 0.$ Since $\overline{Z}_{\bm{m}} \in G_{\bm{m}}(\epsilon_0, 0)$, we have by definition that
\begin{equation}\label{5.1.3}
|\overline{x}_i^{\sigma}(t) -\overline{x}_j^{\sigma'}(t)| > \epsilon_0\, \quad \forall t \geq 0.
\end{equation}
Therefore, by the reverse triangle inequality
\begin{align}
|x_i^{\sigma}(t) -x_j^{\sigma'}(t)| &= | x_i^{\sigma}-x_j^{\sigma'} - t( \overline{v}_i^{\sigma} - \overline{v}_j^{\sigma'}) | \nonumber \\
& \geq |\overline{x}_i^\sigma - \overline{x}_j^{\sigma'} - t( \overline{v}_i^{\sigma} - \overline{v}_j^{\sigma'}) | - | \overline{x}_i^\sigma - \overline{x}_j^{\sigma'}- (x_i^{\sigma}-x_j^{\sigma'}) |  \nonumber\\
& \geq \epsilon_0 - \gamma \geq \epsilon_{(\sigma, \sigma')}\label{5.1.6} 
\end{align}
We have used in the first inequality in \eqref{5.1.6} the equation \eqref{5.1.3} and the fact that $
X_{\bm{m}} \in B_{\gamma/2}^{d|\bm{m}| }(\overline{X}_{\bm{m}})$.
For the second inequality in \eqref{5.1.6}, we have used that $\epsilon_{(\sigma, \sigma')} \ll \gamma \ll \epsilon_0$ by the scaling \eqref{parameter-relations}. \\

\item \textit{(Existing Particle and Adjoined Particle Case)} Next, we show that for all $\sigma \in\mathscr{T}$ and $i \in \{1, \dots, m_\sigma\}$ with $(i, \sigma) \neq (\ell, \alpha)$, we have $|x_i^\sigma(t) - x_{m_\beta + 1}^\beta (t) | > \epsilon_{(\sigma, \beta)},$ for $(\omega_1, v_{m_\beta+1}^\beta)$ outside a specific set. Again using that $\overline{Z}_{\bm{m}} \in G_{\bm{m}}(\epsilon_0, 0)$, we have $|\overline{x}_i^{\sigma} -\overline{x}_\ell^{\alpha}| > \epsilon_0.$ Since $
X_{\bm{m}} \in B_{\gamma/2}^{d|\bm{m}| }(\overline{X}_{\bm{m}})$
we also have that 
$$
|x_i^\sigma - \overline{x}_i^\sigma| < \gamma/2 <\gamma
$$
$$
|\overline{x}_\ell^\alpha - x_{m_\beta +1}^\beta| =  |\overline{x}_\ell^\alpha- x_{\ell}^\alpha +\epsilon_{(\alpha, \beta)} \omega_1| \leq \frac{\gamma}{2} + \epsilon_{(\alpha, \beta)} < \gamma
$$
where in the last line we have used $\epsilon_{(\alpha, \beta)} \ll \gamma$. Applying the first part of Lemma \ref{gallagher-cylinder} with $\overline{y}_1 = \overline{x}_i^\sigma  $, $ \overline{y}_2=\overline{x}_\ell^\alpha$, $y_1 = x_i^\sigma$, and $y_2 = x_{m_\beta + 1}^\beta$, we obtain a cylinder $K_{\eta, i, \sigma}^d$ such that for any $v_{m_\beta + 1}^\beta \in B_R^d \setminus K_{\eta, i, \sigma}^d$ and any $\omega_1 \in \mathbb{S}_1^{d-1}$,
$$
|x_{i}^\sigma(t) - x_{m_\beta + 1}^\beta(t)| \geq \epsilon_0/2 > \epsilon_{(\sigma, \beta)}, \qquad \forall t\geq 0,
$$
where we have used that $\epsilon_{(\sigma, \beta)} \ll \epsilon_0$. \\

\item \textit{(Collisional Law Case)} Finally, let's show that for any $(\omega_1, v_{m_\beta+1}^\beta) \in (\mathbb{S}_1^{d-1} \times B_R^d)^+(\overline{v}_\ell^\alpha)$, we have $
|x_\ell^{\alpha}(t) - x_{m_\beta +1}^\beta(t) | > \epsilon_{(\alpha, \beta)}$ for all $t \geq 0$. First, note that 
\begin{align}
|x_\ell^{\alpha}(t) - x_{m_\beta +1}^\beta(t) |^2 & = |x_\ell^\alpha - t \overline{v}_\ell^\alpha -(x_\ell^\alpha  - \epsilon_{(\alpha, \beta)} \omega_1  -t v_{m_\beta + 1}^\beta) |^2   \\
& = | \epsilon_{(\alpha,\beta)} \omega_1 + t ( v_{m_\beta + 1}^\beta - \overline{v}_\ell^\alpha) |^2\\
& = \epsilon_{(\alpha,\beta)}^2 + t^2 |v_{m_\beta + 1}^\beta - \overline{v}_\ell^\alpha|^2 + 2t \epsilon_{(\alpha, \beta)} \omega_1 \cdot (v_{m_\beta+1}^\beta - \overline{v}_\ell^\alpha) \\
& > \epsilon_{(\alpha,\beta)}^2 \label{5.1.12}
\end{align}
where in \eqref{5.1.12} we used that $(\omega_1, v_{m_\beta+1}^\beta) \in (\mathbb{S}_1^{d-1} \times B_R^d)^+(\overline{v}_\ell^\alpha)$.  \\

\end{enumerate}
Combining cases (1) - (3) together, set $
U_{m_\beta+1, i, \sigma} : =\mathbb{S}_1^{d-1} \times K_{\eta, i, \sigma}^d$. Then the set $\mathcal{B}_{\ell, \alpha}^{0, -}(\overline{Z}_{\bm{m}}) : = \displaystyle\bigcup_{(i, \sigma) \neq (\ell, \alpha)} U_{m_\beta+1, i, \sigma}$,
satisfies the desired properties. 
\end{proof}
For the next lemma, we obtain a bad set $\mathcal{B}_{\ell,\alpha}^{\delta, - }(\overline{Z}_{\bm{m}})$ outside of which we are in a well separated configuration for the precollisional trajectory for all times $t \geq \delta$.
\begin{lem} \label{2,delta,-} Let us be in the same scenario as in Proposition \ref{bad-sets}. Then there exists a set $\mathcal{B}_{\ell,\alpha}^{\delta, - }(\overline{Z}_{\bm{m}})$ such that for $(\omega_1, v_{m_\beta + 1}^\beta) \in ( \mathbb{S}_1^{d-1} \times B_R^d)^+( \overline{v}_\ell^\alpha) \setminus \mathcal{B}_{\ell,\alpha}^{\delta, - }(\overline{Z}_{\bm{m}})$, such that $
Z_{\bm{m} + \beta} \in G_{\bm{m}+\beta}(\epsilon_0/2, \delta)$, and $\overline{Z}_{\bm{m} + \beta} \in G_{\bm{m}+\beta}(\epsilon_0, \delta)$.
\end{lem}
\begin{proof} By considering the three cases found in the proof of Lemma \ref{2,0,-}, we show 
\begin{equation}
    \mathcal{B}_{\ell,\alpha}^{\delta, -}( \overline{Z}_{\bm{m}}) : = (\mathbb{S}_{1}^{d-1} \times B_{\eta}^d(\overline{v}_\ell^\alpha)) \cup \mathcal{B}_{\ell,\alpha}^{0,-}(\overline{Z}_{\bm{m}})
\end{equation}
with $\mathcal{B}_{\ell, \alpha}^{0,-}$ given in Lemma \ref{2,0,-} satisfies our desired properties. 
\end{proof}
The following two Lemmas \ref{2,0,+} and \ref{2,delta,+} are the post-collisional analogues of Lemmas \ref{2,0,-} and \ref{2,delta,-}. Their proofs are similar after taking pre-images under the collisional law.

\begin{lem} \label{2,0,+} Let us be in the same scenario as in Proposition \ref{bad-sets}. Then there exists a bad set $\mathcal{B}_{\ell,\alpha}^{0, +}(\overline{Z}_{\bm{m}})$ such that for all $(\omega_1, v_{m_\beta +1}^\beta) \in (\mathbb{S}_1^{d-1} \times B_R^d)^+(\overline{v}_{\ell}^\alpha)\setminus \mathcal{B}_{\ell,\alpha}^{0, +}(\overline{Z}_{\bm{m}})$, $Z_{\bm{m}+\beta}^*(t) \in \mathring{\mathcal{D}}_{\bm{\epsilon}}^{\bm{m} + \beta}$ for all $t\geq0$.
\end{lem}
\begin{proof} By applying Lemma \ref{2,0,-} to $Z_{\bm{m}+\beta}^*$, the set
\begin{equation} \label{B,0,+ def}
\mathcal{B}_{\ell, \alpha}^{0, +}(\overline{Z}_{\bm{m}}) : = \bigcup_{(i, \sigma) \neq (\ell, \alpha)} U_{m_\beta+1, i, \sigma}^*,\qquad U_{m_\beta+1, i, \sigma} ^* : = \Big \{ (\omega_1, v_{m_\beta + 1}^\beta )  : \, \, (v_{m_\beta +1}^\beta)^* \in K_{\eta, i, \sigma}^d \Big\}
\end{equation}
satisfies the desired properties, where $(v_{m_\beta+1}^\beta)^*$ is given by the collisional law
\begin{equation*}
(v_{m_\beta+1}^\beta)^* =
v_j^\beta + \frac{2 M_{\alpha}}{M_{\alpha}  + M_\beta } ((\overline{v}_\ell^\alpha - v_j^\beta) \cdot \omega_1)\omega_1
\end{equation*}
\end{proof}

\begin{lem} \label{2,delta,+} Let us be in the same scenario as in Proposition \ref{bad-sets}. Then there exists a set $\mathcal{B}_{\ell,\alpha}^{\delta, + }(\overline{Z}_{\bm{m}})$ such that for all $(\omega_1, v_{m_\beta + 1}^\beta) \in ( \mathbb{S}_1^{d-1} \times B_R^d)^+( \overline{v}_\ell^\alpha) \setminus \mathcal{B}_{\ell,\alpha}^{\delta, +}(\overline{Z}_{\bm{m}})$ we have $Z_{\bm{m} + \beta} ^*\in G_{\bm{m}+\beta}(\epsilon_0/2, \delta)$ and $
 \overline{Z}_{\bm{m} + \beta}^* \in G_{\bm{m}+\beta}(\epsilon_0, \delta)$.
\end{lem}
\begin{proof} We apply Lemma \ref{2,delta,-} to $Z_{\bm{m} + \beta}^*$. For $\eta>0$ satisfying \eqref{parameter-relations}, define the set 
\begin{align}
\tilde{U}_{m_{\beta}+1, \ell, \alpha} ^* = \Big \{ (\omega_1, v_{m_\beta+1}^\beta ) : \, \,\big| (v_{m_\beta + 1}^\beta)^* -  (\overline{v}_{\ell}^\alpha)^*\big| < \eta. \Big\} = { \{ (\omega_1, v_{m_\beta+1}^\beta ) : \, \, {v_{m_\beta+1}^\beta} \in B_\eta({\overline{v}_\ell^\alpha})  \}}\label{m_beta+1,l,gamma}.
\end{align} 
The last equality follows from conservation of momentum and energy. One can then verify that the set $\mathcal{B}_{\ell, \alpha}^{\delta, +}( \overline{Z}_{\bm{m}}) : = \tilde{U}_{m_{\beta}+1 , \ell, \alpha}^* \cup \mathcal{B}_{\ell,\alpha}^{0,+}( \overline{Z}_{\bm{m}})$ satisfies the desired properties. 
\end{proof}
\begin{proof}[Proof of Proposition \ref{bad-sets}] Combining the sets found in Lemma \ref{2,0,-} to Lemma \ref{2,delta,+}, define 
$$
\mathcal{B}_{\ell,\alpha} ( \overline{Z}_{\bm{m}}) : = \mathcal{B}_{\ell,\alpha}^{0, -} ( \overline{Z}_{\bm{m}}) \cup\mathcal{B}_{\ell,\alpha}^{0, +} ( \overline{Z}_{\bm{m}}) \cup \mathcal{B}_{\ell,\alpha}^{\delta, -} ( \overline{Z}_{\bm{m}}) \cup \mathcal{B}_{\ell,\alpha}^{\delta, +} ( \overline{Z}_{\bm{m}}).
$$
This set satisfies all of the desired properties. 
\end{proof}

\subsection{Measure Estimates}
We will now provide a measure estimate of the set constructed in Proposition \ref{bad-sets}, with respect to the truncation parameters. We will rely on the geometric estimates for binary interactions as presented in \cite{AP19a}. However, our collision law may involve particles of different masses, which we treat using the corresponding transition maps (operators reducing the post-collisional case to the pre-collisional case) for each type of collision.
\begin{prop}\label{bad-set-measure} Let $\gamma,\epsilon_0,\epsilon_{(1,0)},\epsilon_{(0,1)}, R,\eta,\delta$ be related by \eqref{parameter-relations}. Fix $\bm{m}  = (m_{(1,0)},m_{(0,1)}) \in \N_+^2$, and $\alpha, \beta \in \mathscr{T}$, where $\mathscr{T}$ is the set of types \eqref{set of types}. Let $\overline{Z}_{\bm{m}} \in G_{\bm{m}}(\epsilon_0, 0)$ and assume that $E(\overline{Z}_{\bm{m}}) \leq R^2$, where $E(\cdot)$ is given by \eqref{energy}. Let $\mathcal{B}_{\ell,\alpha} ( \overline{Z}_{\bm{m}})$ be the set found in Proposition \ref{bad-sets}. We have 
\begin{equation}\label{bad set measure}
|\mathcal{B}_{\ell,\alpha} ( \overline{Z}_{\bm{m}}) | \lesssim (m_{(1,0)}+m_{(0,1)}) R^d \eta^{\frac{d-1}{2d+2}}
\end{equation}
where $| \cdot|$ denotes the product measure on $\mathbb{S}_1^{d-1}\times B_R^d$. 
\end{prop}
\begin{rmk}
The constants hidden in \eqref{bad set measure} only depend on dimension and the ratio of the masses of the type $(1,0)$ and type $(0,1)$ particles. 
\end{rmk}
\begin{proof} As in \cite{AP19a}, it suffices to bound each of the terms in 
\begin{equation}
 \Big[(\mathbb{S}_1^{d-1} \times B_R^d)^+ (\overline{v}_{\ell}^\alpha) \Big] \cap \Big[ \mathcal{B}_{\ell,\alpha}^{0, -} ( \overline{Z}_{\bm{m}}) \cup\mathcal{B}_{\ell,\alpha}^{0, +} ( \overline{Z}_{\bm{m}}) \cup \mathcal{B}_{\ell,\alpha}^{\delta, -} ( \overline{Z}_{\bm{m}}) \cup \mathcal{B}_{\ell,\alpha}^{\delta, +} ( \overline{Z}_{\bm{m}})\Big],
\end{equation}
where we recall \eqref{S x Br plus}. We will only prove the estimate on the $\mathcal{B}_{\ell,\alpha}^{0,+}$ term, since it is the most delicate and uses the collisional law. \\

\textbf{Estimate of $\mathcal{B}_{\ell,\alpha}^{ 0, +} ( \overline{Z}_{\bm{m}})$:} Recall the definition of $\mathcal{B}_{\ell,\alpha}^{0,+}$ as given in \eqref{B,0,+ def}. Fix $(i, \sigma) \neq (\ell, \alpha)$ and define
\begin{equation*}
\mathcal{S}^+(\overline{v}_\ell^\alpha, v_{m_\beta+1}^{\beta}) : = \{ \omega_1 \in \mathbb{S}_1^{d-1}   : \, \, \omega_1 \cdot (  v_{m_\beta+1}^{\beta}-\overline{v}_\ell^\alpha ) >0  \} = \{ \omega_1 \in \mathbb{S}_1^{d-1}   : \, \, (\omega_1 , v_{m_\beta+1}^{\beta}) \in (\mathbb{S}_1^d \times B_R^d)^+(\overline{v}_\ell^\alpha ) \} 
\end{equation*}
Using radial coordinates with integration in $v_{m_\beta+1}^\beta$ centered at $\overline{v}_\ell^\alpha$, we estimate 
\begin{equation}\label{polar-coordinate-int}
|(\mathbb{S}_1^{d-1} \times B_R^d)^+ (\overline{v}_\ell^\alpha)  \cap U_{m_\beta+1, i, \sigma}^*| \leq \int_0^{2R}\int_{ \partial B_r(\overline{v}_\ell^\alpha)}  \int_{\mathcal{S}^+(\overline{v}_\ell^\alpha, v_{m_\beta+1}^{\beta})} \1_{U_{m_\beta+1, i, \sigma}^*} d\omega_1 d\sigma_r dr. 
\end{equation}
where $\sigma_r$ is the surface measure on $\partial B_r(\overline{v}_\ell^\alpha)$. Now, fix $r \in [0,2R]$ and introduce a parameter $\theta \in (0,1)$ to decompose $\mathcal{S}^+(\overline{v}_\ell^\alpha, v_{m_\beta+1}^{\beta})$ into two parts: 
\begin{equation}
\mathcal{S}^{+}_1(\overline{v}_\ell^\alpha, v_{m_\beta+1}^{\beta}) : =  \{ \omega_1 \in \mathbb{S}_1^{d-1}   : \, \, \omega_1 \cdot (  v_{m_\beta+1}^{\beta}-\overline{v}_\ell^\alpha ) >\theta| v_{m_\beta+1}^{\beta}-\overline{v}_\ell^\alpha  | \}
\end{equation}
\begin{equation}
\mathcal{S}^{+}_2(\overline{v}_\ell^\alpha, v_{m_\beta+1}^{\beta}) : =  \{ \omega_1 \in \mathbb{S}_1^{d-1}   : \, \, \omega_1 \cdot (  v_{m_\beta+1}^{\beta}-\overline{v}_\ell^\alpha ) \leq \theta| v_{m_\beta+1}^{\beta}-\overline{v}_\ell^\alpha  | \}
\end{equation}
It is clear that $\mathcal{S}^{+}_2(\overline{v}_\ell^\alpha, v_{m_\beta+1}^{\beta})$ is contained in a spherical cap of direction $v_{m_\beta+1}^{\beta}-\overline{v}_\ell^\alpha $ and angle $\arccos \theta$. Hence by integration in spherical coordinates we have (see e.g. \cite{thesis}) 
\begin{equation}\label{S,+,2 estimate}
| \mathcal{S}^{+}_2(\overline{v}_\ell^\alpha, v_{m_\beta+1}^{\beta}) |_{\mathbb{S}_1^d} \lesssim \arcsin\theta. 
\end{equation} 
The other term $\mathcal{S}^{+}_1(\overline{v}_\ell^\alpha, v_{m_\beta+1}^{\beta})$ is more difficult to handle. Motivated by the binary transition map from \cite{AP19a}, we introduce the following transition map tailored to the masses of the particles
\begin{equation}
\mathcal{J}_{\overline{v}_\ell^\alpha, v_{m_\beta+1}^\beta}: \mathbb{S}^{d-1} \rightarrow \mathbb{S}^{d-1}, \qquad \mathcal{J}_{\overline{v}_\ell^\alpha, v_{m_\beta+1}^\beta}(\omega) = r^{-1} \left({(\overline{v}_\ell^\alpha)}^* - {(v_{m_\beta+1}^\beta)}^*\right).
\end{equation}
Here, we are defining 
\begin{equation*}
(v_{m_\beta+1}^\beta)^* =
v_j^\beta + \frac{2 M_{\alpha}}{M_{\alpha}  + M_\beta } ((\overline{v}_\ell^\alpha - v_j^\beta) \cdot \omega)\omega,\qquad (\overline{v}_{\ell}^\alpha)^* =
\overline{v}_\ell^\alpha + \frac{2 M_{\beta}}{M_{\alpha}  + M_\beta } ((\overline{v}_\ell^\alpha - v_j^\beta) \cdot \omega)\omega
\end{equation*}
and $r = | \overline{v}_\ell^\alpha - v_{m_\beta+1}^\beta|$. Set $\nu : = \mathcal{J}_{\overline{v}_\ell^\alpha, v_{m_\beta+1}^\beta}(\omega_1)$. We can check that 
\begin{equation}\label{change-of-var}
{(v_{m_\beta+1}^\beta)}^* = C_{(\alpha,\beta)} v_{m_\beta+1}^\beta + {C_{(\beta,\alpha)}} \overline{v}_\ell^\alpha - {C_{(\beta,\alpha)}}r \nu,\qquad 
{(\overline{v}_\ell^\alpha)}^*=  {C_{(\alpha,\beta)}}v_{m_\beta+1}^\beta +C_{(\beta,\alpha)} \overline{v}_\ell^\alpha+ {C_{(\alpha,\beta)}} r \nu
\end{equation}
where we are defining 
\begin{equation}
C_{(\sigma, \sigma')} : = \frac{ M_{\sigma'}}{M_{\sigma'} + M_{\sigma} }, \qquad M_{\sigma} = \text{"mass of the $\sigma$-type particle".}
\end{equation}
Using \eqref{change-of-var} and recalling the cylinder $K_{\eta,i,\sigma}$ from \eqref{B,0,+ def}, note that
\begin{equation}
{(v_{m_\beta+1}^\beta)}^* \in K^d_{\eta,i,\sigma} \quad \Leftrightarrow\quad \nu \in K_{\eta/r C_{(\beta,\alpha)}, i, \sigma}^d
\end{equation}
where $K_{\eta/r C_{(\beta,\alpha)}, i ,\sigma}^d$ is a cylinder of radius $\eta/r C_{(\beta,\alpha)}$. Using this fact, we write 
\begin{equation}
\int_{\mathcal{S}^{+}_1(\overline{v}_\ell^\alpha, v_{m_\beta+1}^{\beta})} \1_{ U_{m_\beta + 1, i, \sigma}^*}(\omega_1) d \omega_1  = \int_ {\mathcal{S}^{+}_1(\overline{v}_\ell^\alpha, v_{m_\beta+1}^{\beta})} \1_{ K_{\eta/rC_{(\beta,\alpha)}, i ,\sigma}^d} \circ \mathcal{J}_{\overline{v}_\ell^\alpha, v_{m_\beta + 1}^\beta} ( \omega_1) d \omega_1.
\end{equation}
Now, using a change of variables and a Jacobian estimate similar to Proposition 12.2 in \cite{AP19a} we estimate
\begin{equation}\label{S,+,1 estimate}
 \int_ {\mathcal{S}^{+}_1(\overline{v}_\ell^\alpha, v_{m_\beta+1}^{\beta})} \1_{ K_{\eta/r C_{(\beta,\alpha)}, i ,\sigma}^d} \circ \mathcal{J}_{\overline{v}_\ell^\alpha, v_{m_\beta + 1}^\beta} ( \omega_1) d \omega_1  \lesssim \theta^{-d} \min\Big \{ 1, \left(\frac{\eta}{rC_{(\beta,\alpha)} }\right)^{\frac{d-1}{2} } \Big\}.
\end{equation}
Now putting \eqref{polar-coordinate-int},\eqref{S,+,2 estimate}, and \eqref{S,+,1 estimate} together, we obtain 
\begin{equation}
|(\mathbb{S}_1^{d-1} \times B_R^d)^+ (\overline{v}_\ell^\alpha)  \cap U_{m_\beta+1, i, \sigma}^* | \lesssim \int_0^{2R} \int_{\partial B_{r}(\overline{v}_\ell^\alpha)} \left( \arcsin \theta +\theta^{-d} \min\left \{ 1, \left(\frac{\eta}{rC_{(\beta,\alpha)} }\right)^{\frac{d-1}{2} } \right\}\right) d \sigma dr.
\end{equation}
Estimating this integral and minimizing over allowable $\theta$ (see e.g. \cite{thesis}), we obtain
\begin{equation}
|(\mathbb{S}_1^{d-1} \times B_R^d)^+ (\overline{v}_\ell^\alpha)  \cap U_{m_\beta+1, i, \sigma}^* | \lesssim \left( \max (1, C_{(\beta,\alpha)}) \right)^{\frac{1-d}{2}}R^d \eta^{\frac{d-1}{2d+2}}. 
\end{equation}
Summing over indices $(i, \sigma) \neq ( \ell,\alpha)$, of which there are less than $m_{(1,0)} + m_{(0,1)}$, subadditivity gives the estimate.
\end{proof}

\section{Elimination of Recollisions for Mixtures}\label{Elimination of Recollisions for Mixtures}
In this section we reduce the convergence proof to comparing truncated elementary observables. We first restrict to good configurations and provide the corresponding measure estimate. 

\subsection{Restriction to good configurations}
Inductively using Proposition \ref{bad-set-measure} we are able to reduce the convergence proof to good configurations, up to a small measure set. The measure of the complement will  be negligible in the limit. Throughout this subsection $\bm{s} = (s_{(1,0)}, s_{(0,1)}) \in\mathbb{N}^2_+$ will be fixed, and $\bm{N} = (N_{(1,0)}, N_{(0,1)})$,  $\bm{\epsilon} = (\epsilon_{(1,0)},\epsilon_{(0,1)})$ will be given the Boltzmann-Grad scaling as in \eqref{mixed-boltz-grad,1}. The parameters $R,\epsilon_0,\gamma,\eta,\delta$ satisfy \eqref{parameter-relations}.

Given $\bm{m} \in\mathbb{N}_+^2$, and recalling \eqref{good configurations}, let us define the set
\begin{equation}\label{both epsilon-epsilon_0}
G_{\bm{m}}(\bm{\epsilon},\epsilon_0,\delta):=G_{\bm{m}}(\epsilon_{(1,0)},0)\cap G_{\bm{m}}(\epsilon_{(0,1)},0)\cap G_{\bm{m}}(\epsilon_0,\delta).
\end{equation}
Let us also recall from \eqref{separated space data}-\eqref{separated space data 1,2} the set $\Delta_{\bm{s}}^X(\epsilon_0)$ of well-separated spatial configurations. The following lemma can be found in \cite{thesis}.
\begin{lem}\label{initially good configurations}
 Let $\bm{s} \in \N_+^2$, $\gamma,\epsilon_0,R,\eta,\delta$ be parameters as in \eqref{parameter-relations}. Then for any $X_{{\bm{s}}}\in\Delta_{{\bm{s}}}^X(\epsilon_0)$, there is a subset of velocities $\mathcal{M}_{{\bm{s}}}(X_{{\bm{s}}})\subseteq B_R^{d|\bm{s}|}$ of measure
\begin{equation}\label{measure of initialization}
\left|\mathcal{M}_{{\bm{s}}}\left(X_{{\bm{s}}}\right)\right|_{d|\bm{s}|}\leq C_{d,\bm{s}} R^{d|\bm{s}|}\eta^{\frac{d-1}{2}},
\end{equation}
such that for any $V_{{\bm{s}}}\in B_R^{d|\bm{s}|}\setminus \mathcal{M}_{{\bm{s}}}(X_{{\bm{s}}})$, 
we have
\begin{equation}\label{initialization}
Z_{{\bm{s}}}:=(X_{{\bm{s}}},V_{{\bm{s}}})\in G_{{\bm{s}}}(\bm{\epsilon},\epsilon_0,\delta).
\end{equation}
\end{lem}

For each $\bm{s}\in \N_+^2$ and $X_{{\bm{s}}} \in\Delta_{{\bm{s}}}^X(\epsilon_0)$, let us denote $\mathcal{M}_{{\bm{s}}}^c(X_{{\bm{s}}} ):=B_R^{d|\bm{s}|}\setminus\mathcal{M}_{{\bm{s}}}(X_{{\bm{s}}})$. Consider $1\leq k\leq n$ and $\bm{\alpha},\bm{\beta} \in S_k$, where $S_k$ is given in \eqref{collisional history}. Let us recall the observables $I_{{\bm{s}},k,R,\delta}^{\bm{N}}$, $I_{{\bm{s}},k,R,\delta}^\infty$ defined in \eqref{bbgky truncated time}, \eqref{boltzmann truncated time}. We will restrict the domain of integration to velocities giving good configurations. In particular, we define
\begin{align}
\widetilde{I}_{{\bm{s}},k,R,\delta}^{{\bm{N}}}(t)(X_{{\bm{s}}})&: =\sum_{\bm{\alpha}, \bm{\beta} \in S_k} \int_{\mathcal{M}_{{\bm{s}}}^c(X_{{\bm{s}}})}\phi_{{\bm{s}}}(V_{{\bm{s}}})f_{\bm{N},(\bm{\alpha},\bm{\beta}),R,\delta}^{(\bm{s})}(X_{{\bm{s}}},V_{{\bm{s}}})\,dV_{{\bm{s}}} \label{good observables BBGKY },\\
\widetilde{I}_{{\bm{s}},k,R,\delta}^{\infty}(t)(X_{{\bm{s}}})&: =\sum_{\bm{\alpha}, \bm{\beta} \in S_k} \int_{\mathcal{M}_{{\bm{s}}}^c(X_{{\bm{s}}})}\phi_{{\bm{s}}}(V_{{\bm{s}}})f_{(\bm{\alpha},\bm{\beta}),R,\delta}^{(\bm{s})}(X_{{\bm{s}}},V_{{\bm{s}}})\,dV_{{\bm{s}}}\label{good observables Boltzmann},
\end{align}
where we recall that $f_{\bm{N}, (\bm{\alpha}, \bm{\beta}), R, \delta}^{(\bm{s})}$ and $f_{(\bm{\alpha}, \bm{\beta}), R, \delta}^{(\bm{s})}$ are defined in \eqref{fNRdk},\eqref{fRdk}. Let us apply Proposition \ref{initially good configurations} to restrict to initially good configurations.
\begin{prop}\label{restriction to initially good conf} Let $n\in\mathbb{N}$, $\bm{s} \in \N_+^2$, and $\gamma,\epsilon_0,R,\eta,\delta$ be parameters as in \eqref{parameter-relations}. Then, the following estimates hold for all $t \in [0,T]$ and uniformly in $\bm{N}$:
\begin{equation*}
\sum_{k=0}^n  \|I_{{{\bm{s}}},k,R,\delta}^{{\bm{N}}}(t)-\widetilde{I}_{{{\bm{s}}},k,R,\delta}^{{\bm{N}}}(t)\|_{L^\infty\left(\Delta_{{\bm{s}}}^X\left(\epsilon_0\right)\right)}
\leq C_{d,{\bm{s}},\mu_0,T}R^{d|\bm{s}|}\eta^{\frac{d-1}{2}}\|F_{{{\bm{N}}},0}\|_{{\bm{N}},\gamma_0,\mu_0},
\end{equation*}
\begin{equation*}
\sum_{k=0}^n \|I_{{{\bm{s}}},k,R,\delta}^{\infty}(t)-\widetilde{I}_{{{\bm{s}}},k,R,\delta}^{\infty}(t)\|_{L^\infty\left(\Delta_{{\bm{s}}}^X\left(\epsilon_0\right)\right)}
\leq C_{d,{\bm{s}},\mu_0,T}R^{d|\bm{s}|}\eta^{\frac{d-1}{2}}\|F_0\|_{\infty,\gamma_0,\mu_0}.
\end{equation*}
\end{prop}
\begin{proof}
We present the proof for the BBGKY hierarchy case only. The proof for the Boltzmann hierarchy case is similar. Let us fix $X_{{\bm{s}}}\in\Delta_{{\bm{s}}}^X(\epsilon_0)$ and $k\in\left\{1,...,n\right\}$. Applying $k$ times the Lemma \ref{BBGKY-operator-estimates}, we obtain
\begin{align}
|I_{{\bm{s}},k,R,\delta}^{{\bm{N}}}(t)(X_{{\bm{s}}})-\widetilde{I}_{{{\bm{s}}},k,R,\delta}^{{\bm{N}}}(t)(X_{{\bm{s}}})|
&\leq \sum_{\bm{\alpha},\bm{\beta} \in S_k} \int_{\mathcal{M}_{{\bm{s}}}(X_{{\bm{s}}})}|\phi_{{\bm{s}}}(V_{{\bm{s}}})f_{{{\bm{N}}}, (\bm{\alpha}, \bm{\beta}),R,\delta}^{({{\bm{s}}})}(t,X_{{\bm{s}}},V_{{\bm{s}}})|\,dV_{{\bm{s}}} \nonumber\\
&\leq \sum_{\bm{\alpha},\bm{\beta} \in S_k}\|\phi_{{\bm{s}}}\|_{L^\infty}e^{-|\bm{s}|\bm{\mu}(T)}8^{-k}\|F_{\bm{N},0}\|_{\bm{N},\gamma_0,\mu_0}\int_{\mathcal{M}_{{\bm{s}}}(X_{{\bm{s}}})}e^{-\bm{\gamma}(T)E_{{\bm{s}}}(Z_{{\bm{s}}})}\,dV_{{\bm{s}}}\nonumber\\
&\leq \sum_{\bm{\alpha},\bm{\beta} \in S_k}\|\phi_{{\bm{s}}}\|_{L^\infty}e^{-|\bm{s}|\bm{\mu}(T)}8^{-k}|\mathcal{M}_{{\bm{s}}}(X_{{\bm{s}}})|_{d|\bm{s}|}\|F_{{\bm{N},0}}\|_{{{\bm{N}}},\gamma_0,\mu_0}\label{k alpha beta term}
\end{align}
For $k=0$, recall that the map $\bm{T}_{\bm{\epsilon}}^t$ defined in \eqref{BBGKY weak defn} is an isometry on the space $\bm{X}_{\bm{\epsilon},\bm{\gamma},\bm{\mu}}^{\bm{N}}([0,T])$. An application of the triangle inequality thus implies
\begin{equation}\label{k = 0 term}
|I_{{\bm{s}},0,R,\delta}^{{\bm{N}}}(t)(X_{{\bm{s}}})-\widetilde{I}_{{\bm{s}},0,R,\delta}^{{\bm{N}}}(t)(X_{{\bm{s}}})|\leq \|\phi_{{\bm{s}}}\|_{L^\infty}e^{-|\bm{s}|\bm{\mu}(T)}|\mathcal{M}_{{\bm{s}}}(X_{{\bm{s}}})|_{d|\bm{s}|}\|F_{{\bm{N}},0}\|_{\bm{N},\gamma_0,\mu_0}.
\end{equation} 
We now sum the estimates \eqref{k alpha beta term}-\eqref{k = 0 term} over $k=0,...,n$ and apply the measure estimate of Proposition \ref{initially good configurations}.
\end{proof}
\begin{rmk}\label{no need for k=0}
Given ${\bm{s}}\in\mathbb{N}^2_+$ and $X_{{\bm{s}}}\in\Delta_{{\bm{s}}}^X(\epsilon_0)$, the definition of $\mathcal{M}_{{\bm{s}}}(X_{{\bm{s}}})$ implies that $\widetilde{I}_{{{\bm{s}}},0,R,\delta}^{{\bm{N}}}(t)(X_{{\bm{s}}})=\widetilde{I}_{{{\bm{s}}},0,R,\delta}^\infty(t)(X_{{\bm{s}}})$. Therefore, Proposition \ref{restriction to initially good conf} allows us to reduce the convergence to controlling the differences $\widetilde{I}_{{{\bm{s}}},k,R,\delta}^{{\bm{N}}}(t)-\widetilde{I}_{{{\bm{s}}},k,R,\delta}^\infty(t),$ for $k=1,...,n$, in the scaled limit.
\end{rmk}
\subsection{Reduction to elementary observables}\label{reduction to elementary observables}
In this subsection, given ${\bm{s}}\in\mathbb{N}^2$ and $1\leq k\leq n$, we express the observables $\widetilde{I}_{{{\bm{s}}},k,R,\delta}^{{\bm{N}}}(t)$ and $\widetilde{I}_{{{\bm{s}}},k,R,\delta}^\infty(t)$, defined in \eqref{good observables BBGKY }, \eqref{good observables Boltzmann}, as a superposition of elementary observables.

 For this purpose, given $\bm{\ell} : = (\ell_{(1,0)}, \ell_{(0,1)})\in\mathbb{N}^2$, and $\alpha, \beta \in \mathscr{T}$, and recalling the truncated collision operators \eqref{velocity truncation of operators}, we decompose the BBGKY hierarchy collisional operators in the following way:
 \begin{equation*}
\mathcal{C}_{{\bm{\ell}}, {\bm{\ell}} +\beta, R}^{ \alpha}=\sum_{i=1}^{\ell_\alpha}\mathcal{C}_{{\bm{\ell}}, {\bm{\ell}} +\beta, R}^{ \alpha, i, +}-\sum_{i=1}^{\ell_\alpha}\mathcal{C}_{{\bm{\ell}}, {\bm{\ell}} +\beta, R}^{ \alpha, i ,-},
\end{equation*}
where we are defining
\begin{equation}
\mathcal{C}_{{\bm{\ell}}, {\bm{\ell}} +\beta, R}^{ \alpha, i, +}g^{({\bm{\ell}}+\beta)}(Z_{{\bm{\ell}}}) =A_{{\bm{\epsilon}},{\bm{\ell}}}^{{\bm{N}},(\alpha,\beta)}\int_{\mathbb{S}_1^{d-1}\times B_R^{d}}(\omega_1\cdot (v_{\ell_\beta+1}^\beta-v_i^\alpha))_+ g^{({\bm{\ell}}+\beta)}(Z_{{\bm{\ell}} +\beta,\bm{\epsilon}}^{i, \alpha,*})\,d\omega_1\,dv_{\ell_\beta+1}^\beta,
\end{equation} 
\begin{equation}
\mathcal{C}_{{\bm{\ell}}, {\bm{\ell}} +\beta, R}^{ \alpha, i, -}g^{({\bm{\ell}}+\beta)}(Z_{{\bm{\ell}} }) =A_{{\bm{\epsilon}},{\bm{\ell}}}^{{\bm{N}},(\alpha,\beta)}\int_{\mathbb{S}_1^{d-1}\times B_R^{d}}(\omega_1\cdot (v_{\ell_\beta+1}^\beta-v_i^\alpha))_+ g^{({\bm{\ell}} +\beta)}(Z_{{\bm{\ell}} +\beta,\bm{\epsilon}}^{i, \alpha})\,d\omega_1\,dv_{\ell_\beta+1}^\beta.
\end{equation} 
Here, we let $A_{\bm{\epsilon},\bm{\ell}}^{\bm{N},(\alpha,\beta)} = (N_\beta - \ell_\beta) \epsilon_{(\alpha,\beta)}^{d-1}
$, and $Z_{\bm{\ell}+\beta,\bm{\epsilon}}^{i,\alpha,*},Z_{\bm{\ell}+\beta,\bm{\epsilon}}^{i,\alpha}$ be as defined in \eqref{Z i a +},\eqref{Z i a *}, and \eqref{Z i a -}. This process of splitting the collision operators can be viewed as isolating the types of interactions being summed over. 
Given $1\leq k\leq n$, $\bm{\alpha},\bm{\beta}\in S_k$ and recalling \eqref{beta tilde}, let us denote
\begin{align}
\mathcal{J}_{{\bm{s}},k}&=\left\{J=(j_1,...,j_k):j_i\in\left\{+,-\right\},\quad\forall i\in\left\{1,...,k\right\}\right\}\label{J_k},\\
\mathcal{M}_{{\bm{s}},k,\bm{\beta}}&=\left\{M=(m_1,...,m_k)\in\mathbb{N}^k:m_i\in\left\{1,...,s_{{\beta}_i}+\widetilde{\beta}_{i-1}^{{\beta}_i}\right\},\quad\forall i\in\left\{1,...,k\right\}\right\}\label{M_k}. \\
\mathcal{U}_{{\bm{s}},k,\bm{\beta}}&=\mathcal{J}_{{\bm{s}},k}\times\mathcal{M}_{{\bm{s}},k,\bm{\beta}}.\label{U_k}
\end{align}
Here, the number $s_{\beta_i} + \widetilde{\beta}_{i-1}^{\beta_i}$ is exactly the number of $\beta_i$ type particles in the system after adding particles of types $\beta_1, \dots \beta_{i-1}$ to the system of $\bm{s}$ particles. Under this notation, the BBGKY hierarchy observable functional $\widetilde{I}_{{{\bm{s}}},k,R,\delta}^{{\bm{N}}}(t)$ defined in \eqref{good observables BBGKY } can be expressed, for $1\leq k\leq n$, as a superposition of elementary observables
\begin{equation}\label{superposition BBGKY}
\widetilde{I}_{{{\bm{s}}},k,R,\delta}^{{\bm{N}}}(t)(X_{{\bm{s}}})=\sum_{\bm{\alpha},\bm{\beta} \in S_k} \sum_{(J,M)\in\mathcal{U}_{{\bm{s}},k,\bm{\beta}}}\left(\prod_{i=1}^kj_i\right)\widetilde{I}_{{\bm{s}},k,R,\delta}^{{\bm{N}}}(t,\bm{\alpha}, \bm{\beta}, J,M)(X_{{\bm{s}}}),
\end{equation}
where the elementary observables are defined by
\begin{equation}\label{elementary observable BBGKY}
\begin{aligned}
\widetilde{I}_{{\bm{s}},k,R,\delta}^{{\bm{N}}}(t,\bm{\alpha},\bm{\beta},J,M)(X_{{\bm{s}}})
&=\int_{\mathcal{M}_s^c(X_{{\bm{s}}})}\phi_{{\bm{s}}}(V_{{\bm{s}}})\int_{\mathcal{T}_{k,\delta}(t)}T_{{\bm{s}},\bm{\epsilon}}^{t-t_1}\mathcal{C}_{{\bm{s}}, {\bm{s}} + \widetilde{\beta}_1, R}^{\alpha_1,m_1,j_1} T_{{\bm{s}}+\widetilde{\beta}_1,\bm{\epsilon}}^{t_1-t_2}...\\
&...T_{{\bm{s}}+\widetilde{\beta}_{k-1},\bm{\epsilon}}^{t_{k-1}-t_k}\mathcal{C}_{{\bm{s}}+\widetilde{\beta}_{k-1},{\bm{s}}+\widetilde{\beta}_k,R}^{\alpha_k,m_k,j_k} T_{{\bm{s}}+\widetilde{\beta}_k,\bm{\epsilon}}^{t_k}f_{\bm{N},0}^{({\bm{s}}+\widetilde{\beta}_k)}(Z_{{\bm{s}}})\,dt_k...\,dt_{1}dV_{{\bm{s}}}.
\end{aligned}
\end{equation}
Similarly, given $\bm{\ell}=(\ell_{(1,0)}, \ell_{(0,1)})\in\mathbb{N}_+^2$, $\alpha,\beta \in \mathscr{T}$, and recalling the truncated Boltzmann collision operator \eqref{velocity truncation of operators}, we decompose the Boltzmann collisional operators as:
 \begin{equation*}
\mathscr{C}_{{\bm{\ell}}, {\bm{\ell}} +\beta, R}^{ \alpha}=\sum_{i=1}^{\ell_\alpha}\mathscr{C}_{{\bm{\ell}}, {\bm{\ell}} +\beta, R}^{ \alpha, i, +}-\sum_{i=1}^{\ell_\alpha}\mathscr{C}_{{\bm{\ell}}, {\bm{\ell}} +\beta, R}^{ \alpha, i ,-},
\end{equation*}
where we are defining
\begin{equation}
\mathscr{C}_{{\bm{\ell}}, {\bm{\ell}} +\beta, R}^{\alpha, i, +}g^{({\bm{\ell}}+\beta)}(Z_{{\bm{\ell}} }) =A_\beta^\alpha\int_{\mathbb{S}_1^{d-1}\times B_R^{d}}(\omega_1\cdot (v_{\ell_\beta+1}^\beta-v_i^\alpha))_+ g^{({\bm{\ell}} +\beta)}(Z_{{\bm{\ell}} +\beta}^{i,\alpha,*})\,d\omega_1\,dv_{\ell_\beta+1}^\beta,
\end{equation} 
\begin{equation}
\mathscr{C}_{{\bm{\ell}}, {\bm{\ell}} +\beta, R}^{\alpha, i, -}g^{({\bm{\ell}}+\beta)}(Z_{{\bm{\ell}}}) =A_\beta^\alpha\int_{\mathbb{S}_1^{d-1}\times B_R^{d}}(\omega_1\cdot (v_{\ell_\beta+1}^\beta-v_i^\alpha))_+ g^{({\bm{\ell}} +\beta)}(Z_{{\bm{\ell}} +\beta}^{i,\alpha})\,d\omega_1\,dv_{\ell_\beta+1}^\beta.
\end{equation} 
Here, $A^\alpha_\beta$ is given in \eqref{A a b def} and $Z_{{\bm{\ell}} +\beta}^{i,\alpha,*},Z_{{\bm{\ell}} +\beta}^{i,\alpha}$ are given as in Definition \ref{infinite-collision-operators}. Under this notation, the Boltzmann hierarchy observable functional $\widetilde{I}_{{{\bm{s}}},k,R,\delta}^{\infty}(t)$ defined in \eqref{good observables Boltzmann} can be expressed, for $1\leq k\leq n$, as a superposition of elementary observables
\begin{equation}\label{superposition boltz}
\widetilde{I}_{{{\bm{s}}},k,R,\delta}^{\infty}(t)(X_{{\bm{s}}})=\sum_{\bm{\alpha},\bm{\beta} \in S_k} \sum_{(J,M)\in\mathcal{U}_{{\bm{s}},k,\bm{\beta}}}\left(\prod_{i=1}^k j_i\right)\widetilde{I}_{{\bm{s}},k,R,\delta}^{\infty}(t,\bm{\alpha},\bm{\beta},J,M)(X_{{\bm{s}}}),
\end{equation}
where the elementary observables are defined by
\begin{equation}\label{elementary observable boltz}
\begin{aligned}
\widetilde{I}_{{\bm{s}},k,R,\delta}^{\infty}(t,\bm{\alpha},\bm{\beta},J,M)(X_{{\bm{s}}})&=\int_{\mathcal{M}_s^c(X_{{\bm{s}}})}\phi_{{\bm{s}}}(V_{{\bm{s}}})\int_{\mathcal{T}_{k,\delta}(t)}S_{{\bm{s}}}^{t-t_1}\mathscr{C}_{{\bm{s}}, {\bm{s}} + \widetilde{\beta}_1, R}^{\alpha_1,m_1,j_1} S_{{\bm{s}}+\widetilde{\beta}_1}^{t_1-t_2}...\\
&...S_{{\bm{s}}+\widetilde{\beta}_{k-1}}^{t_{k-1}-t_k}\mathscr{C}_{{\bm{s}}+\widetilde{\beta}_{k-1},{\bm{s}}+\widetilde{\beta}_k,R}^{\alpha_k,m_k,j_k} S_{{\bm{s}}+\widetilde{\beta}_k}^{t_k}f_{0}^{({\bm{s}}+\widetilde{\beta}_k)}(Z_{{\bm{s}}})\,dt_k...\,dt_{1}dV_{{\bm{s}}}.
\end{aligned}
\end{equation}

\subsection{Boltzmann hierarchy pseudo-trajectories}\label{subsec Boltzmann pseudo}
We introduce the following notation which we will be constantly using from now on. Let ${\bm{s}} = (s_{(1,0)}, s_{(0,1)}) = (s_1, s_2) \in\mathbb{N}^2_+$, $Z_{{\bm{s}}}\in\mathbb{R}^{2d|\bm{s}|}$, $1\leq k\leq n$, $\bm{\alpha},\bm{\beta} \in S_k$ and $t\in[0,T]$. Let us recall the set $\mathcal{T}_{k}(t)$ defined in \eqref{time seq}.

Consider $(t_1,...,t_k)\in\mathcal{T}_k(t)$, $J=(j_1,...,j_k)$, $M=(m_1,...,m_k)$, $(J,M)\in\mathcal{U}_{s,k,\bm{\beta}}$ given in \eqref{U_k}. We inductively define the Boltzmann hierarchy pseudo-trajectory of $Z_{{\bm{s}}}$. Roughly speaking, the Boltzmann hierarchy pseudo-trajectory forms the configurations on which particles are adjusted during backwards in time evolution. 

For instance, assume we are given a configuration $Z_{{\bm{s}}}\in\mathbb{R}^{2d|\bm{s}|}$ 
at time $t_0=t$. The dynamics of $Z_{{\bm{s}}}(\cdot)$ evolves under backwards free flow until the time $t_1$ where the configuration $(\omega_1, v_{s_{\alpha_1} +1}^{\alpha_1})\in \mathbb{S}_1^{d-1} \times B_R^d$ is added, neglecting positions, to the $m_1$th particle of type $\alpha_1$, the adjunction being pre-collisional if $j_1=-1$ and post-collisional if $j_1=1$. We then form an $(\bm{s}+\alpha_1)$-configuration and continue this process inductively until time $t_{k+1}=0$. More precisely, we inductively construct the Boltzmann hierarchy pseudo-trajectory of $Z_{{\bm{s}}}=(X_{{\bm{s}}},V_{{\bm{s}}})\in\mathbb{R}^{2d|\bm{s}|}$
as follows:\\

{\bf{Time $t_0=t$}:} We initially define 
\begin{equation*}
Z_{{\bm{s}}}^\infty(t_{0}^-) := Z_{{\bm{s}}}.
\end{equation*}
We will use the following notation for the components of $Z_{{\bm{s}}}^\infty(t_{0}^-) $
\begin{align}
Z_{{\bm{s}}}^\infty(t_{0}^-) &=(X_{{\bm{s}}}^\infty(t_{0}^-),V_{{\bm{s}}}^\infty(t_{0}^-)).
\end{align}
Individual components of the vectors $X_{{\bm{s}}}^\infty(t_{0}^-), V_{{\bm{s}}}^\infty(t_{0}^-)$ will be written as $(x_i^\alpha)^\infty(t_0^-),$ or $(v_i^\alpha)^\infty(t_0^-)$, respectively.\\

{\bf{Time $t_i$} with $i\in\{1,...,k\}$:} Consider $i\in\left\{1,...,k\right\}$ and assume we know $(Z_{{\bm{s}}+\widetilde{\beta}_{i-1} })^\infty(t_{i-1}^-).$ We define $(Z_{{\bm{s}} + \widetilde{\beta}_{i-1}})^\infty(t_i^+)$ as follows:
\begin{equation*}
(Z_{{\bm{s}} + \widetilde{\beta}_{i-1}})^\infty(t_i^+):=\left((X_{{\bm{s}}+\widetilde{\beta}_{i-1}})^\infty(t_{i-1}^-)-(t_{i-1}-t_i)(V_{{\bm{s}}+\widetilde{\beta}_{i-1}})^\infty(t_{i-1}^-),(V_{{\bm{s}}+\widetilde{\beta}_{i-1}})^\infty(t_{i-1}^-)\right).
\end{equation*}
We also define $(Z_{{\bm{s}}+\widetilde{\beta}_{i}})^\infty(t_{i}^-)$ as follows. First, denote $Z_{\bm{s}+ \widetilde{\beta}_i}^\infty(t_i^-) : = (X_{\bm{s}+ \widetilde{\beta}_i}^\infty(t_i^-),V_{\bm{s}+ \widetilde{\beta}_i}^\infty(t_i^-))$, where 
\begin{equation}
   X_{\bm{s}+ \widetilde{\beta}_i}^\infty(t_i^-): = \begin{cases}\left((X_{s_1+\widetilde{\beta}_{i-1}^{(1,0)}}^{(1,0)})^\infty(t_i^-), (x_{s_{1}+\widetilde{\beta}_{i-1}^{(1,0)}+1}^{(1,0)})^\infty(t_i^-),(X_{s_2+\widetilde{\beta}_{i-1}^{(0,1)}}^{(0,1)})^\infty(t_i^-)\right), & \beta_i= (1,0)\\
\left((X_{s_1+\widetilde{\beta}_{i-1}^{(1,0)}}^{(1,0)})^\infty(t_i^-),(X_{s_2+\widetilde{\beta}_{i-1}^{(0,1)}}^{(0,1)})^\infty(t_i^-),(x_{s_{2}+\widetilde{\beta}_{i-1}^{(0,1)}+1}^{(0,1)})^\infty(t_i^-)\right), & \beta_i= (0,1)
\end{cases}
\end{equation}
\begin{equation}
   V_{\bm{s}+ \widetilde{\beta}_i}^\infty(t_i^-): = \begin{cases}\left((V_{s_1+\widetilde{\beta}_{i-1}^{(1,0)}}^{(1,0)})^\infty(t_i^-), (v_{s_{1}+\widetilde{\beta}_{i-1}^{(1,0)}+1}^{(1,0)})^\infty(t_i^-),(V_{s_2+\widetilde{\beta}_{i-1}^{(0,1)}}^{(0,1)})^\infty(t_i^-)\right), & \beta_i= (1,0)\\
\left((V_{s_1+\widetilde{\beta}_{i-1}^{(1,0)}}^{(1,0)})^\infty(t_i^-),(V_{s_2+\widetilde{\beta}_{i-1}^{(0,1)}}^{(0,1)})^\infty(t_i^-),(v_{s_{2}+\widetilde{\beta}_{i-1}^{(0,1)}+1}^{(0,1)})^\infty(t_i^-)\right), & \beta_i= (0,1).
\end{cases}
\end{equation}
For $\sigma \in \mathscr{T}$ and $\ell \in \{1, \dots, s_{\sigma}+\widetilde{\beta}_{i-1}^{\sigma}\}$ with $(\ell,\sigma) \neq (m_i, \alpha_i)$, we define 
\begin{equation*}
\left((x_\ell^\sigma)^\infty(t_i^-),(v_\ell^\sigma)^\infty(t_i^-)\right)
:=\left((x_\ell^\sigma)^\infty(t_i^+),(v_\ell^\sigma)^\infty(t_i^+)\right).
\end{equation*}
For the rest of the particles, we distinguish the following cases, depending on $j_i$:
\begin{itemize}
\item If $j_i=-1$:
\begin{equation*}
\begin{aligned}
\left((x_{m_i}^{\alpha_i})^\infty(t_i^-),(v_{m_i}^{\alpha_i})^\infty(t_i^-)\right)&:=\left((x_{m_{i}}^{\alpha_i})^\infty(t_i^+),(v_{m_{i}}^{\alpha_i})^\infty(t_i^+)\right),\\
\left((x_{s_{\beta_i}+\widetilde{\beta}_{i-1}^{\beta_i}+1}^{\beta_i})^\infty(t_i^-),(v_{s_{\beta_i}+\widetilde{\beta}_{i-1}^{\beta_i}+1}^{\beta_i})^\infty(t_i^-)\right)&:=\left((x_{m_{i}}^{\alpha_i})^\infty(t_i^+),v_{s_{\beta_i}+\widetilde{\beta}_{i-1}^{\beta_i}+1}^{\beta_i}\right),
\end{aligned}
\end{equation*}
\item If $j_i=1$:
\begin{equation*}
\begin{aligned}
\left((x_{m_i}^{\alpha_i})^\infty(t_i^-),(v_{m_i}^{\alpha_i})^\infty(t_i^-)\right)&:=\left((x_{m_{i}}^{\alpha_i})^\infty(t_i^+),(v_{m_{i}}^{\alpha_i})^\infty(t_i^+)^*\right),\\
\left((x_{s_{\beta_i}+\widetilde{\beta}_{i-1}^{\beta_i}+1}^{\beta_i})^\infty(t_i^-),(v_{s_{\beta_i}+\widetilde{\beta}_{i-1}^{\beta_i}+1}^{\beta_i})^\infty(t_i^-)\right)&:=\left((x_{m_{i}}^{\alpha_i})^\infty(t_i^+),(v_{s_{\beta_i}+\widetilde{\beta}_{i-1}^{\beta_i}+1}^{\beta_i})^*\right),
\end{aligned}\end{equation*}
where we are defining
\begin{equation*}
\begin{pmatrix}
(v_{m_{i}}^{\alpha_i})^\infty(t_i^+)^*\\
(v_{s_{\beta_i}+\widetilde{\beta}_{i-1}^{\beta_i}+1}^{\beta_i})^*
\end{pmatrix}=
\begin{pmatrix}
(v_{m_{i}}^{\alpha_i})^\infty(t_i^+) - \frac{2M_{\beta_i}}{M_{\alpha_i} + M_{\beta_i}} \left((v^{\alpha_i}_{m_i}- v^{\beta_i}_{s_{\beta_i}+\widetilde{\beta}_{i-1}^{\beta_i}+1} )\cdot \omega_i \right) \omega_i\\
v_{s_{\beta_i}+\widetilde{\beta}_{i-1}^{\beta_i}+1}^{\beta_i} + \frac{2M_{\alpha_i}}{M_{\alpha_i} + M_{\beta_i}} \left((v^{\alpha_i}_{m_i}- v^{\beta_i}_{s_{\beta_i}+\widetilde{\beta}_{i-1}^{\beta_i}+1} )\cdot \omega_i \right) \omega_i
\end{pmatrix}
\end{equation*}
\end{itemize}
{\bf{Time $t_{k+1}=0$}:}
We finally obtain 
$$Z_{{\bm{s}}+\widetilde{\beta}_{k}}^\infty(0^+)=Z_{{\bm{s}}+\widetilde{\beta}_{k}}^\infty(t_{k+1}^+)=\left(X_{{\bm{s}}+\widetilde{\beta}_{k}}^\infty\left(t_{k}^-\right)-t_k V_{{\bm{s}}+\widetilde{\beta}_{k}}^\infty\left(t_k^-\right),V_{{\bm{s}}+\widetilde{\beta}_{k}}^\infty\left(t_k^-\right)\right).$$
 The process is illustrated in the following diagram:
 
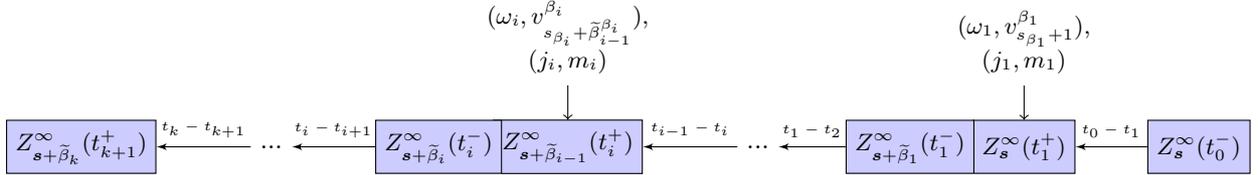
\begin{center}
\begin{tikzpicture}[node distance=2.5cm,auto,>=latex']\label{boltzmann pseudo diagram}
\node[int](0-){\small$ Z_{\bm{s}}^\infty(t_0^-)$};
\node[int,pin={[init]above:\small$\begin{matrix}({\omega}_{1},{v}_{s_{\beta_1}+1}^{\beta_1}),\\(j_1,m_1)\end{matrix}$}](1+)[left of=0-,node distance=2.3cm]{\small$Z_{\bm{s}}^\infty(t_1^+)$};
\node[int](1-)[left of=1+,node distance=1.5cm]{\small$Z_{\bm{s}+\widetilde{\beta}_1}^\infty(t_1^-)$};
\node[](intermediate1)[left of=1-,node distance=2cm]{...};
\node[int,pin={[init]above:\small$\begin{matrix}({\omega}_{i},{v}_{{s}_{\beta_i}+\widetilde{\beta}_{i-1}^{\beta_i}}^{\beta_i}),\\(j_i,m_i)\end{matrix}$}](i+)[left of=intermediate1,node distance=2.5cm]{\small$Z_{\bm{s}+\widetilde{\beta}_{i-1}}^\infty(t_i^+)$};
\node[int](i-)[left of=i+,node distance=1.7cm]{\small$Z_{\bm{s}+\widetilde{\beta}_i}^\infty(t_i^-)$};
\node[](intermediate2)[left of=i-,node distance=2.2cm]{...};
\node[int](end)[left of=intermediate2,node distance=2.5cm]{\small$Z_{\bm{s}+\widetilde{\beta}_k}^\infty(t_{k+1}^+)$};

\path[<-] (1+) edge node {\tiny$t_{0}-t_1$} (0-);
\path[<-] (intermediate1) edge node {\tiny$t_{1}-t_2$} (1-);
\path[<-] (i+) edge node {\tiny$t_{i-1}-t_i$} (intermediate1);
\path[<-] (intermediate2) edge node {\tiny$t_{i}-t_{i+1}$} (i-);
\path[<-] (end) edge node {\tiny$t_{k}-t_{k+1}$} (intermediate2);
\end{tikzpicture}
\end{center}
\begin{definition}\label{Boltzmann pseudo}
Let $Z_{{\bm{s}}}=(X_{{\bm{s}}},V_{{\bm{s}}})\in\mathbb{R}^{2d|\bm{s}|}$, $k\in\mathbb{N}$, $(t_1,...,t_k)\in\mathcal{T}_k(t)$, $\bm{\alpha},\bm{\beta}\in S_k$, $J=(j_1,...,j_k)$, $M=(m_1,...,m_k)$, $(J,M)\in\mathcal{U}_{s,k,\bm{\beta}}$ and for each $i=1,...,k$,  we consider  $(\omega_i, v_{s_{\beta_i} + \widetilde{\beta}_{i-1}^{\beta_i}+1}^{\beta_i})\in\mathbb{S}_{1}^{d-1}\times B_R^{d}.$ The sequence $\{{Z}_{{\bm{s}}+\widetilde{\beta}_{i-1}}^\infty(t_i^+)\}_{i=0,...,k+1}$ constructed above is called the Boltzmann hierarchy pseudo-trajectory of $Z_{{\bm{s}}}$.
\end{definition}

\subsection{Reduction to truncated elementary observables}\label{par_reduction to truncated}
 
We will now use the Boltzmann hierarchy pseudo-trajectory to define the BBGKY hierarchy and Boltzmann hierarchy truncated observables. The convergence proof will then be reduced to the convergence of  the corresponding truncated elementary observables. Given ${\bm{\ell}} \in\mathbb{N}^2$, recall the notation from \eqref{both epsilon-epsilon_0}:
\begin{equation*}G_{{\bm{\ell}}}(\epsilon_{(1,0)},\epsilon_{(0,1)},\epsilon_0,\delta):=G_{{\bm{\ell}}}(\epsilon_{(1,0)},0)\cap G_{{\bm{\ell}}}(\epsilon_{(0,1)},0)\cap G_{{\bm{\ell}}}(\epsilon_0,\delta).
\end{equation*}

Fix ${\bm{s}} = (s_{(1,0)}, s_{(0,1)} ) \in \N^2$, let $X_{{\bm{s}}}\in\Delta_{{\bm{s}}}^X(\epsilon_0)$, $1\leq k\leq n$, $\bm{\alpha},\bm{\beta} \in S_k$, $t\in [0,T]$, and $(J,M)\in\mathcal{U}_{s,k,\bm{\beta}}$. Also fix $(t_1, \dots, t_k) \in \mathcal{T}_{k,\delta}(t)$. By Lemma \ref{initially good configurations}, for any $V_{{\bm{s}}}\in\mathcal{M}_{{\bm{s}}}^c(X_{{\bm{s}}})$, we have 
$$Z_{{\bm{s}}}=(X_{{\bm{s}}},V_{{\bm{s}}})\in G_{{\bm{s}}} (\epsilon_{(1,0)},\epsilon_{(0,1)},\epsilon_0,\delta).$$
This implies that by construction, since $t_0 - t_1\geq \delta$, we obtain ${Z}_{{\bm{s}}}^\infty(t_1^+)\in G_{{\bm{s}}}(\epsilon_0,0)$. We will inductively apply part (c) of Proposition \ref{bad-sets}. Given $i\in\left\{1,...,k\right\}$, assume that
  \begin{equation}\label{boltz-good}
{Z}_{{\bm{s}}+\widetilde{\beta}_{i-1}}^\infty(t_{i}^+)\in G_{{\bm{s}}+\widetilde{\beta}_{i-1}}(\epsilon_0,0).
  \end{equation}
 Then, there exists a set $\mathcal{B}_{m_i,\alpha_i}\left({Z}_{{\bm{s}}+\widetilde{\beta}_{i-1}}^\infty\left(t_{i}^+\right)\right)\subseteq \mathbb{S}_1^{d-1}\times B_R^{d}$ such that: 
  \begin{equation}\label{pseudo applicable} {Z}_{{\bm{s}}+\widetilde{\beta}_{i}}^\infty(t_{i+1}^+)\in G_{{\bm{s}}+\widetilde{\beta}_{i}}(\epsilon_0,0),\quad\forall (\omega_{i},v_{s_{\beta_i} + \widetilde{\beta}_{i-1}^{\beta_i} +1}^{\beta_i})\in \mathcal{B}_{m_i,\alpha_i}^c\left({Z}_{{\bm{s}}+\widetilde{\beta}_{i-1}}^\infty\left(t_{i}^+\right)\right),
  \end{equation}
  where
  $$\mathcal{B}_{m_i,\alpha_i}^c\left({Z}_{{\bm{s}}+ \widetilde{\beta}_{i-1}}^\infty\left(t_i^+\right)\right):=(\mathbb{S}_1^{d-1}\times B_R^{d})^+\left(({v}_{m_i}^{\alpha_{i}})^\infty \left(t_i^+\right)\right)\setminus \mathcal{B}_{m_i,\alpha_i}\left({Z}_{{\bm{s}}+ \widetilde{\beta}_{i-1}}^\infty\left(t_i^+\right)\right).$$ 
  After completing this procedure, we finally obtain ${Z}_{{\bm{s}}+\widetilde{\beta}_{k}}^\infty(0^+)\in G_{{\bm{s}}+\widetilde{\beta}_{k}}(\epsilon_0,0)$.

Let us now define the truncated elementary observables. Heuristically we will truncate the domains of adjusted particles in the definition of the observables $\widetilde{I}_{{\bm{s}},k,R,\delta}^{{\bm{N}}}$, $\widetilde{I}_{{\bm{s}},k,R,\delta}^{\infty}$, defined in \eqref{good observables BBGKY }-\eqref{good observables Boltzmann}. More precisely, consider $1\leq k\leq n$, $\bm{\alpha},\bm{\beta} \in S_k$, $(J,M)\in\mathcal{U}_{{\bm{s}},k,\bm{\beta}}$ and $t\in [0,T]$. For $X_{{\bm{s}}}\in\Delta_{{\bm{s}}}^X(\epsilon_0)$, Lemma \ref{initially good configurations} implies there is a set of velocities $\mathcal{M}_{{\bm{s}}}(X_{{\bm{s}}})\subseteq B_R^{d|\bm{s}|}$ such that $Z_{{\bm{s}}}=(X_{{\bm{s}}},V_{{\bm{s}}})\in G_{{\bm{s}}}(\epsilon_{(1,0)},\epsilon_{(0,1)},\epsilon_0, \delta)$ for all $V_{{\bm{s}}}\in\mathcal{M}_{{\bm{s}}}^c(X_{{\bm{s}}})$. Now we define the BBGKY hierarchy truncated observables as:
\begin{equation}\label{truncated BBGKY}
\begin{aligned}
J_{{\bm{s}},k,R,\delta}^{\bm{N}}(t,\bm{\alpha},\bm{\beta},J,M)(X_{{\bm{s}}} )&=\int_{\mathcal{M}_{{\bm{s}}}^c(X_{{\bm{s}}})}\phi_{{\bm{s}}}(V_{{\bm{s}}})\int_{\mathcal{T}_{k,\delta}(t)}T_{{\bm{s}}}^{t-t_1}\widetilde{\mathcal{C}}_{{{\bm{s}}},{{\bm{s}}}+\widetilde{\beta}_{1},R}^{\alpha_1, m_1,j_1} T_{{\bm{s}}+\widetilde{\beta}_{1}}^{t_1-t_2}...\\
&\qquad \qquad \cdots \widetilde{\mathcal{C}}_{{\bm{s}}+\widetilde{\beta}_{k-1},{\bm{s}}+\widetilde{\beta}_{k},R}^{\alpha_k,m_k,j_k} T_{{\bm{s}}+\widetilde{\beta}_{k}}^{t_k}f^{({\bm{s}}+\widetilde{\beta}_{k})}_{\bm{N},0}(Z_{{\bm{s}}}) dt_k\cdots dt_{1}dV_{{\bm{s}}},
\end{aligned}
\end{equation}
where for each $i = 1, \dots, k$ we recall the sets \eqref{pseudo applicable} and define 
\begin{equation*}
\begin{aligned}
\widetilde{\mathcal{C}}_{{\bm{s}}+\widetilde{\beta}_{i-1},{\bm{s}}+\widetilde{\beta}_{i},R}^{\alpha_i,m_i,j_i}&\left(g_{{\bm{s}}+\widetilde{\beta}_{i}}\right)=\mathcal{C}_{{\bm{s}}+\widetilde{\beta}_{i-1},{\bm{s}}+\widetilde{\beta}_{i},R}^{\alpha_i,m_i,j_i}\left[g_{{\bm{s}}+\widetilde{\beta}_{i}}\1_{\mathcal{B}^c_{m_i,\alpha_i}\left({Z}_{{\bm{s}}+\widetilde{\beta}_{i-1}}^\infty\left(t_i^+\right)\right)}\right].
\end{aligned}
\end{equation*}

In the same spirit, for $X_{{\bm{s}}} \in\Delta_{{\bm{s}}}^X(\epsilon_0)$, we define the Boltzmann hierarchy truncated elementary observables:
\begin{equation}\label{truncated Boltzmann}
\begin{aligned}
J_{{\bm{s}},k,R,\delta}^\infty(t,\bm{\alpha},\bm{\beta},J,M)(X_{{\bm{s}}} )&=\int_{\mathcal{M}_{{\bm{s}}}^c(X_{{\bm{s}}})}\phi_{{\bm{s}}}(V_{{\bm{s}}})\int_{\mathcal{T}_{k,\delta}(t)}S_{{\bm{s}}}^{t-t_1}\widetilde{\mathscr{C}}_{{{\bm{s}}},{{\bm{s}}}+\widetilde{\beta}_{1},R}^{\alpha_1,m_1,j_1} S_{{\bm{s}}+\widetilde{\beta}_{1}}^{t_1-t_2}...\\
&\qquad \qquad \cdots \widetilde{\mathscr{C}}_{{\bm{s}}+\widetilde{\beta}_{k-1},{\bm{s}}+\widetilde{\beta}_{k},R}^{\alpha_k,m_k,j_k} S_{{\bm{s}}+\widetilde{\beta}_{k}}^{t_k}f^{({\bm{s}}+\widetilde{\beta}_{k})}_0(Z_{{\bm{s}}}) dt_k\cdots dt_{1}dV_{{\bm{s}}},
\end{aligned}
\end{equation}
where for each $i=1,...,k$, we recall the sets \eqref{pseudo applicable} and define
\begin{align*}\widetilde{\mathscr{C}}_{{\bm{s}}+\widetilde{\beta}_{i-1},{\bm{s}}+\widetilde{\beta}_{i},R}^{\alpha_i,m_i,j_i}&\left(g_{{\bm{s}}+\widetilde{\beta}_{i}}\right)=\mathscr{C}_{{\bm{s}}+\widetilde{\beta}_{i-1},{\bm{s}}+\widetilde{\beta}_{i},R}^{\alpha_i,m_i,j_i}\left[g_{{\bm{s}}+\widetilde{\beta}_{i}}\1_{\mathcal{B}^c_{m_i,\alpha_i}\left({Z}_{{\bm{s}}+\widetilde{\beta}_{i-1}}^\infty\left(t_i^+\right)\right)}\right].
\end{align*}
 
 Recalling the observables $\widetilde{I}_{{\bm{s}},k,R,\delta}^{{\bm{N}}}$, $\widetilde{I}_{{\bm{s}},k,R,\delta}^{\infty}$, defined in \eqref{good observables BBGKY }-\eqref{good observables Boltzmann} and using Proposition \ref{bad-set-measure}, we obtain the following proposition. 
 \begin{prop}\label{truncated element estimate} Let $ n\in\mathbb{N}, \bm{s} \in \N_+^2$, and $\gamma, \epsilon_0,R,\eta,\delta$ be parameters as in (\ref{parameter-relations}). Additionally, let $(\bm{N},\bm{\epsilon})$ be as in the scaling (\ref{mixed-boltz-grad,1}). Then for all $t \in [0,T]$ the following estimates hold unifomly in $\bm{N}$:
 \begin{equation*}
 \begin{aligned}\sum_{k=1}^n\sum_{\bm{\alpha},\bm{\beta} \in S_k}\sum_{(J,M)\in\mathcal{U}_{{\bm{s}},k,\bm{\beta}}}&\|\widetilde{I}_{{\bm{s}},k,R,\delta}^{{\bm{N}}} (t,\bm{\alpha},\bm{\beta},J,M)-J_{{\bm{s}},k,R,\delta}^{{\bm{N}}} (t,\bm{\alpha},\bm{\beta},J,M)\|_{L^\infty\left(\Delta_{{\bm{s}}}^X\left(\epsilon_0\right)\right)}\leq \\
 &\leq C_{d,\bm{s},\mu_0,T}^n\|\phi_{{\bm{s}}} \|_{L^\infty_{V_{{\bm{s}}}}} R^{d(|\bm{s}|+2n)}\eta^{\frac{(d-1)(d+2)}{2d+2}}\|F_{{\bm{N}},0}\|_{\bm{\epsilon},\gamma_0,\mu_0},
 \end{aligned}
 \end{equation*}
 \begin{equation*}
 \begin{aligned}\sum_{k=1}^n\sum_{\bm{\alpha},\bm{\beta} \in S_k}\sum_{(J,M)\in\mathcal{U}_{{\bm{s}},k,\bm{\beta}}}&\|\widetilde{I}_{{\bm{s}},k,R,\delta}^{\infty} (t,\bm{\alpha},\bm{\beta},J,M)-J_{{\bm{s}},k,R,\delta}^{\infty} (t,\bm{\alpha},\bm{\beta},J,M)\|_{L^\infty\left(\Delta_{{\bm{s}}}^X\left(\epsilon_0\right)\right)}\leq \\
 &\leq C_{d,\bm{s},\mu_0,T}^n\|\phi_{{\bm{s}}} \|_{L^\infty_{V_{{\bm{s}}}}} R^{d(|\bm{s}|+2n)}\eta^{\frac{(d-1)(d+2)}{2d+2}}\|F_0\|_{0,\gamma_0,\mu_0},
 \end{aligned}
 \end{equation*}
 \end{prop}
 \begin{proof}
 As usual, it suffices to prove the estimate for the BBGKY hierarchy case and the Boltzmann hierarchy case follows similarly. Fix $k\in\left\{1,...,n\right\}$, $\bm{\alpha},\bm{\beta}\in S_k$ and $(J,M)\in\mathcal{U}_{{\bm{s}},k,\bm{\beta}}$. We will bound the norm of the summand 
  \begin{equation}\label{estimated difference}
\widetilde{I}_{{\bm{s}},k,R,\delta}^{{\bm{N}}} (t,\bm{\alpha},\bm{\beta},J,M)-J_{{\bm{s}},k,R,\delta}^{{\bm{N}}} (t,\bm{\alpha},\bm{\beta},J,M),
 \end{equation}
 then use some combinatorial estimates to evaluate a bound on the whole sum. To bound this single term, note first that Cauchy-Schwartz inequality and triangle inequality imply
\begin{equation} 
|\langle\omega_1,v_1-v\rangle|\leq 2R,\quad\forall \omega_1\in\mathbb{S}_1^{d-1},\quad\forall v,v_1\in B_R^d.\label{triangle on cross binary}
 \end{equation}
Therefore we have for large $R$,
\begin{equation}\label{int of cross binary}
    \int_{\mathbb{S}^{d-1} \times B_R^d} |\langle \omega , v_1- v\rangle| d\omega dv_1 \leq C_d R^{d+1} , \quad \forall v \in B_R^d.
\end{equation}
In order to estimate the iterated integrals \eqref{estimated difference}, we must integrate over at least one of the sets $\mathcal{B}_{m_i,\alpha_i}\left( Z_{\bm{s}+\widetilde{\beta}_{i-1}}^\infty(t_i^+)\right)$ for some $i \in \{1, \dots, k\}$. By Proposition \ref{bad-set-measure} and \eqref{triangle on cross binary}, we may estimate 
\begin{equation}\label{exceptional integral}
    \int_{\mathcal{B}_{m_i,\alpha_i}\left( Z_{\bm{s}+\widetilde{\beta}_{i-1}}^\infty(t_i^+)\right)}|\langle \omega_1, v_1 - v\rangle| d\omega_1 dv_1 \leq C_d |\bm{s} + \widetilde{\beta}_{i-1}| R^{d+1}\eta^{\frac{d-1}{2d+2}}, \quad \forall v\in B_R^d.
\end{equation}
 Moreover, we have the elementary inequalities:
 \begin{align}
 \|f^{({\bm{s}} + \widetilde{\beta}_k)}_{{\bm{N}},0}\|_{L^\infty}&\leq e^{-(|\bm{s}| +k)\mu_0}\|F_{{\bm{N}},0}\|_{{\bm{\epsilon}},\gamma_0,\mu_0}\label{exclusion bad set 2 norms},\\
 \int_{\mathcal{T}_{k,\delta}(t)}\,dt_1...\,dt_k&=\leq\int_0^t\int_0^{t_1}...\int_0^{t_{k-1}}\,dt_1...\,dt_k=\frac{t^k}{k!}\leq\frac{T^k}{k!}\label{exclusion bad set 2 time},
 \end{align}
and  by Lemma \ref{initially good configurations} the estimate
 \begin{equation}\label{M_s measure}
  | \mathcal{M}_{{\bm{s}}}(X_{{\bm{s}}})| \leq C_{d,\bm{s}} R^{d|\bm{s}|} \eta^{\frac{d-1}{2}}.
 \end{equation}
 Therefore, \eqref{triangle on cross binary} -\eqref{M_s measure} imply that for some $i \in \{1, \dots, k\}$ and sufficiently large $R, n$,
 \begin{equation*}
 \begin{aligned}
 \big|&\widetilde{I}_{{\bm{s}},k,R,\delta}^{{\bm{N}}} (t,\bm{\alpha},\bm{\beta}, J,M)-J_{{\bm{s}},k,R,\delta}^{{\bm{N}}} (t,\bm{\alpha},\bm{\beta}, J,M)\big|\leq\Vert \phi_{{\bm{s}}} \Vert_{L^\infty} e^{-(|\bm{s}| +k)\mu_0}\|F_{{\bm{N}},0}\|_{{\bm{\epsilon}},\gamma_0,\mu_0}\\
 &\qquad \times C_{d,\bm{s}} R^{d|\bm{s}|} \eta^{\frac{d-1}{2}}  C_d^{k-1} R^{(d+1)(k-1)} C_d (|\bm{s} + \widetilde{\beta}_{i-1}|) R^{d+1}  \eta^{\frac{d-1}{2d+2}}\frac{T^k}{k!}\\
 & \leq C_{d,\bm{s},T,\mu_0}^k \Vert \phi_{{\bm{s}}} \Vert_{L^\infty} \|F_{{\bm{N}},0}\|_{{\bm{\epsilon}},\gamma_0,\mu_0} \frac{(|\bm{s}|+k)}{k!}R^{d(|\bm{s}| +2n)} \eta^{\frac{(d-1)(d+2)}{2d+2}}  \end{aligned}
 \end{equation*}
 Adding for all $(J,M)\in \mathcal{U}_{{\bm{s}},k,\bm{\beta}}$ we have $ 2^k|\bm{s}|(|\bm{s}|+1)...(|\bm{s}| + k)\leq 2^k(|\bm{s}| +k)^k$ contributions, thus
 \begin{equation*}
 \begin{aligned}
 &\sum_{(J,M)\in \mathcal{U}_{{\bm{s}},k,\bm{\beta}}}\|\widetilde{I}_{{\bm{s}},k,R,\delta}^{{\bm{N}}} (t,\bm{\alpha},\bm{\beta},J,M)-J_{{\bm{s}},k,R,\delta}^{{\bm{N}}} (t,\bm{\alpha},\bm{\beta},J,M)\|_{L^\infty\left(\Delta_{{\bm{s}}}^X\left(\epsilon_0\right)\right)}\leq\\
 &\leq C_{d,\bm{s} ,\mu_0,T}^k\frac{ 2^k (|\bm{s}|+ k)^{k+1}}{k!}\|\phi_{{\bm{s}}}\|_{L^\infty_{V_{{\bm{s}}}}}\|F_{{\bm{N}},0}\|_{{\bm{\epsilon}},\gamma_0,\mu_0}R^{d(|\bm{s}| + 2n)}\eta^{\frac{(d-1)(d+2)}{2d+2}} \\
 &\leq C_{d,\bm{s},\mu_0,T}^k \|\phi_{{\bm{s}}}\|_{L^\infty_{V_{{\bm{s}}}}}R^{d(|\bm{s}|+ 2n)}\eta^{\frac{(d-1)(d+2)}{2d+2}} \|F_{{\bm{N}},0}\|_{{\bm{\epsilon}},\gamma_0,\mu_0},
 \end{aligned}
 \end{equation*}
since $
\frac{2^k(|\bm{s}|+k)^{k+1}}{k!}\leq C_{\bm{s}}^k.$ Summing over $\bm{\alpha}, \bm{\beta}\in S_k$, $k=1,...,n$, we gain a factor of $
 \sum_{k=1}^n 4^k \leq n4^n \leq C^n$
in the full sum, but still obtain the required estimate.
 \end{proof}
\section{Convergence Proof}\label{Convergence Proof}
In order to conclude the convergence proof, we will estimate the differences of truncated elementary observables for the BBGKY and Boltzmann hierarchies in the scaled limit.
\subsection{BBGKY Pseudo-Trajectories and Proximity to Boltzmann Pseudo-Trajectories}
Let ${\bm{s}} = (s_{(1,0)}, s_{(0,1)}) = (s_1, s_2) \in\mathbb{N}^2$, $Z_{{\bm{s}}}\in\mathbb{R}^{2d|\bm{s}|}$, $1\leq k\leq n$, $\bm{\alpha},\bm{\beta} \in S_k$ and $t\in[0,T]$. Moreover, fix $(\bm{N},\bm{\epsilon})$ to obey the Boltzmann-Grad scaling \eqref{mixed-boltz-grad,1}. Let us recall from the set $\mathcal{T}_k(t)$ defined in \eqref{time seq}. Consider $(t_1,...,t_k)\in\mathcal{T}_k(t)$, $J=(j_1,...,j_k)$, $M=(m_1,...,m_k)$ with $(J,M)\in\mathcal{U}_{s,k,\bm{\beta}}$. We define the BBGKY hierarchy pseudo-trajectory of $Z_{{\bm{s}}}$ in an inductive manner similar to that of the Boltzmann pseudo-trajectory.  \\

{\bf{Time $t_0=t$}:} We initially define $Z_{{\bm{s}}}^{\bm{N}} (t_{0}^-) := Z_{{\bm{s}}}$. We will denote the components of $Z_{{\bm{s}}}^{\bm{N}}(t_{0}^-)$ by
\begin{align}
Z_{{\bm{s}}}^{\bm{N}}(t_{0}^-) &=(X_{{\bm{s}}}^{\bm{N}}(t_{0}^-),V_{{\bm{s}}}^{\bm{N}}(t_{0}^-)).
\end{align}
Individual components of the vectors $X_{{\bm{s}}}^{\bm{N}}(t_{0}^-), V_{{\bm{s}}}^{\bm{N}}(t_{0}^-)$ will simply be written as $(x_i^\alpha)^{\bm{N}}(t_0^-),$ or $(v_i^\alpha)^{\bm{N}}(t_0^-)$, respectively.\\

{\bf{Time $t_i$} with $i\in\{1,...,k\}$:} Consider $i\in\left\{1,...,k\right\}$ and assume we know $(Z_{{\bm{s}}+\widetilde{\beta}_{i-1} })^{\bm{N}}(t_{i-1}^-)$. We define $(Z_{{\bm{s}} + \widetilde{\beta}_{i-1}})^{\bm{N}}(t_i^+)$ as follows:
\begin{equation*}
(Z_{{\bm{s}} + \widetilde{\beta}_{i-1}})^{\bm{N}}(t_i^+):=\left((X_{{\bm{s}}+\widetilde{\beta}_{i-1}})^{\bm{N}}(t_{i-1}^-)-(t_{i-1}-t_i)(V_{{\bm{s}}+\widetilde{\beta}_{i-1}})^{\bm{N}}(t_{i-1}^-),(V_{{\bm{s}}+\widetilde{\beta}_{i-1}})^{\bm{N}}(t_{i-1}^-)\right).
\end{equation*}
We also define $(Z_{{\bm{s}}+\widetilde{\beta}_{i}})^{\bm{N}}(t_{i}^-)$ as follows. First, denote $Z_{\bm{s}+ \widetilde{\beta}_i}^{\bm{N}}(t_i^-) : = (X_{\bm{s}+ \widetilde{\beta}_i}^{\bm{N}}(t_i^-),V_{\bm{s}+ \widetilde{\beta}_i}^{\bm{N}}(t_i^-))$, where we are defining 

\begin{equation}
   X_{\bm{s}+ \widetilde{\beta}_i}^{\bm{N}}(t_i^-): = \begin{cases}\left((X_{s_1+\widetilde{\beta}_{i-1}^{(1,0)}}^{(1,0)})^{\bm{N}}(t_i^-), (x_{s_{1}+\widetilde{\beta}_{i-1}^{(1,0)}+1}^{(1,0)})^{\bm{N}}(t_i^-),(X_{s_2+\widetilde{\beta}_{i-1}^{(0,1)}}^{(0,1)})^{\bm{N}}(t_i^-)\right), & \beta_i= (1,0)\\
\left((X_{s_1+\widetilde{\beta}_{i-1}^{(1,0)}}^{(1,0)})^{\bm{N}}(t_i^-),(X_{s_2+\widetilde{\beta}_{i-1}^{(0,1)}}^{(0,1)})^{\bm{N}}(t_i^-),(x_{s_{2}+\widetilde{\beta}_{i-1}^{(0,1)}+1}^{(0,1)})^{\bm{N}}(t_i^-)\right), & \beta_i= (0,1)
\end{cases}
\end{equation}
\begin{equation}
   V_{\bm{s}+ \widetilde{\beta}_i}^{\bm{N}}(t_i^-): = \begin{cases}\left((V_{s_1+\widetilde{\beta}_{i-1}^{(1,0)}}^{(1,0)})^{\bm{N}}(t_i^-), (v_{s_{1}+\widetilde{\beta}_{i-1}^{(1,0)}+1}^{(1,0)})^{\bm{N}}(t_i^-),(V_{s_2+\widetilde{\beta}_{i-1}^{(0,1)}}^{(0,1)})^{\bm{N}}(t_i^-)\right), & \beta_i= (1,0)\\
\left((V_{s_1+\widetilde{\beta}_{i-1}^{(1,0)}}^{(1,0)})^{\bm{N}}(t_i^-),(V_{s_2+\widetilde{\beta}_{i-1}^{(0,1)}}^{(0,1)})^{\bm{N}}(t_i^-),(v_{s_{2}+\widetilde{\beta}_{i-1}^{(0,1)}+1}^{(0,1)})^{\bm{N}}(t_i^-)\right), & \beta_i= (0,1).
\end{cases}
\end{equation}
For $\sigma \in \mathscr{T}$ and $\ell \in \{1, \dots, s_{\sigma}+\widetilde{\beta}_{i-1}^{\sigma}\}$ with $(\ell,\sigma) \neq (m_i, \alpha_i)$, we define 
\begin{equation*}
\left((x_\ell^\sigma)^{\bm{N}}(t_i^-),(v_\ell^\sigma)^{\bm{N}}(t_i^-)\right)
:=\left((x_\ell^\sigma)^{\bm{N}}(t_i^+),(v_\ell^\sigma)^{\bm{N}}(t_i^+)\right).
\end{equation*}
For the rest of the particles, we distinguish the following cases, depending on $j_i$:
\begin{itemize}
\item If $j_i=-1$:
\begin{equation*}
\begin{aligned}
\left((x_{m_i}^{\alpha_i})^{\bm{N}}(t_i^-),(v_{m_i}^{\alpha_i})^{\bm{N}}(t_i^-)\right)&:=\left((x_{m_{i}}^{\alpha_i})^{\bm{N}}(t_i^+),(v_{m_{i}}^{\alpha_i})^{\bm{N}}(t_i^+)\right),\\
\left((x_{s_{\beta_i}+\widetilde{\beta}_{i-1}^{\beta_i}+1}^{\beta_i})^{\bm{N}}(t_i^-),(v_{s_{\beta_i}+\widetilde{\beta}_{i-1}^{\beta_i}+1}^{\beta_i})^{\bm{N}}(t_i^-)\right)&:=\left((x_{m_{i}}^{\alpha_i})^{\bm{N}}(t_i^+)-\epsilon_{(\alpha_i,\beta_i)}\omega_i,v_{s_{\beta_i}+\widetilde{\beta}_{i-1}^{\beta_i}+1}^{\beta_i}\right),
\end{aligned}
\end{equation*}
\item If $j_i=1$:
\begin{equation*}
\begin{aligned}
\left((x_{m_i}^{\alpha_i})^{\bm{N}}(t_i^-),(v_{m_i}^{\alpha_i})^{\bm{N}}(t_i^-)\right)&:=\left((x_{m_{i}}^{\alpha_i})^{\bm{N}}(t_i^+),(v_{m_{i}}^{\alpha_i})^{\bm{N}}(t_i^+)^*\right),\\
\left((x_{s_{\beta_i}+\widetilde{\beta}_{i-1}^{\beta_i}+1}^{\beta_i})^{\bm{N}}(t_i^-),(v_{s_{\beta_i}+\widetilde{\beta}_{i-1}^{\beta_i}+1}^{\beta_i})^{\bm{N}}(t_i^-)\right)&:=\left((x_{m_{i}}^{\alpha_i})^{\bm{N}}(t_i^+)+\epsilon_{(\alpha_i,\beta_i)}\omega_i,(v_{s_{\beta_i}+\widetilde{\beta}_{i-1}^{\beta_i}+1}^{\beta_i})^*\right),
\end{aligned}\end{equation*}
where we are defining
\begin{equation*}
\begin{pmatrix}
(v_{m_{i}}^{\alpha_i})^{\bm{N}}(t_i^+)^*\\
(v_{s_{\beta_i}+\widetilde{\beta}_{i-1}^{\beta_i}+1}^{\beta_i})^*
\end{pmatrix}=
\begin{pmatrix}
(v_{m_{i}}^{\alpha_i})^{\bm{N}}(t_i^+) - \frac{2M_{\beta_i}}{M_{\alpha_i} + M_{\beta_i}} \left((v^{\alpha_i}_{m_i}- v^{\beta_i}_{s_{\beta_i}+\widetilde{\beta}_{i-1}^{\beta_i}+1} )\cdot \omega_i \right) \omega_i\\
v_{s_{\beta_i}+\widetilde{\beta}_{i-1}^{\beta_i}+1}^{\beta_i} + \frac{2M_{\alpha_i}}{M_{\alpha_i} + M_{\beta_i}} \left((v^{\alpha_i}_{m_i}- v^{\beta_i}_{s_{\beta_i}+\widetilde{\beta}_{i-1}^{\beta_i}+1} )\cdot \omega_i \right) \omega_i
\end{pmatrix}
\end{equation*}
\end{itemize}
{\bf{Time $t_{k+1}=0$}:}
We finally obtain 
$$Z_{{\bm{s}}+\widetilde{\beta}_{k}}^{\bm{N}}(0^+)=Z_{{\bm{s}}+\widetilde{\beta}_{k}}^{\bm{N}}(t_{k+1}^+)=\left(X_{{\bm{s}}+\widetilde{\beta}_{k}}^{\bm{N}}\left(t_{k}^-\right)-t_k V_{{\bm{s}}+\widetilde{\beta}_{k}}^{\bm{N}}\left(t_k^-\right),V_{{\bm{s}}+\widetilde{\beta}_{k}}^{\bm{N}}\left(t_k^-\right)\right).$$
\begin{definition}\label{BBGKY pseudo} Let $t>0$, $Z_{{\bm{s}}} \in \R^{2d(s_1 + s_2)}$, $(t_1, \dots, t_k) \in \mathcal{T}_k(t)$, $1 \leq k \leq n$,  $\bm{\alpha}, \bm{\beta} \in S_k$, $(J,M) \in \mathcal{U}_{{\bm{s}}, k, \bm{\beta}}$, and for each $i= 1, \dots,k$ let $(\omega_i, v_{ s_{{\beta_i}} + \widetilde{\beta}_{i-1}^{\beta_i}+1}^{\beta_i})\in \mathbb{S}_1^{d-1}\times B_R^d$. We call the sequence $\{{Z}_{{\bm{s}} + \widetilde{\beta}_{i-1}}^{{\bm{N}}}( t_i^+)\}_{i=1}^k$ defined above the BBGKY pseudo-trajectory. 
\end{definition}

\begin{lem} \label{pseudo-comparison}Fix ${\bm{s}} = (s_{(1,0)}, s_{(0,1)}) = (s_1, s_2) \in \N^2$, $n \in \N_+$, and $(\bm{N},\bm{\epsilon})$. Fix $t \in [0,T]$, $(t_1, \dots, t_k) \in \mathcal{T}_k(t)$, $1 \leq k \leq n$,  $\bm{\alpha}, \bm{\beta} \in S_k$, $(J,M) \in \mathcal{U}_{{\bm{s}}, k, \bm{\beta}}$, and for each $i= 1, \dots,k$ let $(\omega_i, v_{ s_{{\beta_i}} + \widetilde{\beta}_{i-1}^{\beta_i}+1}^{\beta_i})\in \mathbb{S}_1^{d-1}\times B_R^d.$ Then for each $i = 1, \dots, k+1$, and each $\sigma \in\mathscr{T}$, we have 
\begin{equation}\label{comparison of traj inductive}
|({x}_{\ell}^\sigma)^\infty(t_i^+)-(x_\ell^\sigma)^{{\bm{N}}}(t_i^+) | \leq \sqrt{2} (i-1)\max_{\alpha\in \mathscr{T}}\epsilon_{\alpha},\qquad ({v}_{\ell}^\sigma)^\infty(t_i^+)=(v_\ell^\sigma)^{{\bm{N}}}(t_i^+)
\end{equation}
for each \footnote{Here, the number $s_{\sigma} +\widetilde{\beta}_{i-1}^\sigma$ comes from the total number of particles in the system plus the number of particles we have added in constructing the BBGKY pseudo-trajectory.}  $\ell = 1, \dots, s_{\sigma} +\widetilde{\beta}_{i-1}^\sigma$. 
In particular for $s_1 ,s_2< n$, and $i = 1, \dots, k+1$, we have 
\begin{equation}\label{total comparison of traj}
| X_{{\bm{s}} + \widetilde{\beta}_{i-1}}^{{\bm{N}}} (t_i^+) - {X}_{{\bm{s}} + \widetilde{\beta}_{i-1}}^\infty (t_i^+)| \leq \sqrt{8}n^{2} \max_{\alpha \in \mathscr{T}}{\epsilon_{\alpha}}.
\end{equation}
\end{lem}
\begin{proof} The statement \eqref{comparison of traj inductive} follows from a simple inductive argument. See \cite{GSRT13} for details. For the uniform bound \eqref{total comparison of traj}, fix $s_1, s_2 < n$, $ 1 \leq k \leq n$, and $1 \leq i \leq k+1$. Now apply \eqref{comparison of traj inductive} to calculate 
\begin{align*}
\left| X_{{\bm{s}} + \widetilde{\beta}_{i-1}}^{{\bm{N}}} (t_i^+) - {X}_{{\bm{s}} + \widetilde{\beta}_{i-1}}^\infty (t_i^+)\right|^2 & = \sum_{\sigma \in \mathscr{T}} \sum_{\ell=1}^{s_{\sigma} + \widetilde{\beta}_{i-1}^\sigma} |({x}_{\ell}^\sigma)^\infty(t_i^+)-(x_\ell^\sigma)^{{\bm{N}}}(t_i^+) |^2 \leq 8n^4 \max_{\alpha \in \mathscr{T}}(\epsilon_{\alpha})^2 
\end{align*}
Taking square roots proves \eqref{total comparison of traj}.
\end{proof}

\subsection{Truncated Observables in Terms of Pseudo-Trajectories} \label{trunc obs pseudo}
 We will now write the truncated observables coming from the Boltzmann hierarchy in terms of the Boltzmann pseudo-trajectories. By Definition \ref{Boltzmann pseudo} of the Boltzmann pseudo-trajectories, we may re-write the truncated observables given in \eqref{truncated Boltzmann} as
\begin{equation}
\begin{aligned}
J_{{\bm{s}},k,R,\delta}^\infty(t,\bm{\alpha},\bm{\beta},J,M)(X_{{\bm{s}}} )&= A^\infty_{\bm{\alpha},\bm{\beta},k}\int_{\mathcal{M}^c_{{\bm{s}}}(X_{{\bm{s}}}) } \phi_{{\bm{s}}}(V_{{\bm{s}}}) \int_{\mathcal{T}_{k,\delta}(t)} \int_{B_1^c} \cdots \int_{B_k^c} f^{({\bm{s}} + \widetilde{\beta}_k)}_0\left({Z}_{{\bm{s}}+\widetilde{\beta}_k}^\infty(0^+)\right) \\
&  \prod_{i=1}^k j_i \langle \omega_i , ({v}_{s_{\beta_i} + \widetilde{\beta}_{i-1}^{\beta_i} + 1}^{\beta_i})^\infty (t_i^+)- v_{m_i}^{\alpha_i} \rangle_+d\omega_k dv_{s_{\beta_k} + \widetilde{\beta}_{k-1}^{\beta_k}+1}^{\beta_k}\dots d\omega_1 dv_{s_{\beta_1} + 1}^{\beta_1}dt_k \dots dt_1dt dV_{{\bm{s}}}
\end{aligned}
\end{equation}
where we recall the sets given in \eqref{pseudo applicable} and define $B_i^c := \mathcal{B}^c_{m_i,\alpha_i}({Z}_{{\bm{s}}+\widetilde{\beta}_{i-1}}^\infty(t_i^+))$ and $A^\infty_{\bm{\alpha},\bm{\beta},k}: =\prod_{i=1}^k A^{\alpha_i}_{\beta_i}$. The constants $A_{\beta}^\alpha$ are determined by our scaling \eqref{mixed-boltz-grad,1}, and are explicitly given in \eqref{A a b def}. \\

Since the $\bm{\epsilon}$ particle flow may include recollisions, it is not immediately clear we can do the same expansion for the BBGKY truncated observables $J_{{\bm{s}},k,R,\delta}^{{\bm{N}}} (t,\bm{\alpha},\bm{\beta},J,M)(X_{{\bm{s}}} )$ as given in \eqref{truncated BBGKY} in terms of the BBGKY pseudo-trajectories given in Definition \ref{BBGKY pseudo}. However, due to the angle and velocity sets which we excluded in constructing $J_{{\bm{s}}, k, R, \delta}^{{\bm{N}}}(t, \bm{\alpha},\bm{\beta},J,M)(X_{{\bm{s}}})$ from $\widetilde{I}_{{\bm{s}}, k, R, \delta}^{{\bm{N}}}(t,\bm{\alpha},\bm{\beta}, J,M)(X_{{\bm{s}}})$ (as given in \eqref{elementary observable BBGKY}), we claim the relevant $\bm{\epsilon}$ flows will not experience recollisions. To show this claim, fix initial positions $X_{{\bm{s}}} \in \Delta_{{\bm{s}}}^X(\epsilon_0)$, $1 \leq k \leq n$, $(J,M) \in \mathcal{U}_{{\bm{s}}, k, \bm{\beta}}$, and $(t_1, \dots, t_k) \in \mathcal{T}_{k, \delta}(t)$ for $t \in [0,T]$. Consider $(\bm{N},\bm{\epsilon})$ that obey the scalings \eqref{mixed-boltz-grad,1}, and moreover satisfy 
\begin{equation}\label{n ll alpha}
n^{2} \max_{\alpha \in \mathscr{T}}\epsilon_\alpha \ll \gamma,
\end{equation}
where the implicit constants depend only on universal constants. Additionally assume that $s_1, s_2 < n$. Then given $V_{{\bm{s}}} \in \mathcal{M}^c_{{\bm{s}}}(X_{{\bm{s}}})$, we have by Lemma \ref{initially good configurations} that 
$$
Z_{{\bm{s}}} = ( X_{{\bm{s}}}, V_{{\bm{s}}}) \in G_{{\bm{s}}}(\epsilon_{(1,0)}, \epsilon_{(0,1)}, \epsilon_0, \delta), 
$$
where $G_{{\bm{s}}}(\epsilon_{(1,0)}, \epsilon_{(0,1)}, \epsilon_0, \delta)$ is defined in \eqref{both epsilon-epsilon_0}. Recalling the operators $\Psi_{\bm{s},\bm{\epsilon}}^{\tau}, \Phi_{\bm{s}}^\tau$ given in \eqref{global-flow-thm}, \eqref{phi-s-def}, we obtain
\begin{equation}\label{t0 - t1}
{\Psi}_{{\bm{s}}, \bm{\epsilon}}^{t_1 - t_0} (Z_{{\bm{s}}} )= \Phi_{{\bm{s}}}^{t_1 - t_0} (Z_{{\bm{s}}} )  =  Z_{{\bm{s}}}^{{\bm{N}}}(t_1^+)
\end{equation}
since $t_0 -t_1 \geq \delta$. Moreover, recall that by \eqref{pseudo applicable}, we have by construction that for each $i \in \{1, \dots, k\}$, 
$$
{Z}_{{\bm{s}}+\widetilde{\beta_{i}} }^\infty(t_{i+1}^+) \in G_{{\bm{s}}+\widetilde{\beta_{i}}}(\epsilon_0, 0), \qquad \forall (\omega_{i},v_{s_{\beta_i} + \widetilde{\beta}_{i-1}^{\beta_i} +1}^{\beta_i})\in \mathcal{B}_{m_i,\alpha_i}^c\left({Z}_{{\bm{s}}+\widetilde{\beta}_{i-1}}^\infty\left(t_{i}^+\right)\right).
$$
Since we have $s_1, s_2 < n$, we obtain by Lemma \ref{pseudo-comparison}, recalling \eqref{n ll alpha}, that 
\begin{equation}\label{X^N - X^infty}
| X_{{\bm{s}} + \widetilde{\beta}_{i-1}}^{{\bm{N}}} (t_i^+) - {X}_{{\bm{s}} + \widetilde{\beta}_{i-1}}^\infty (t_i^+)| \leq \sqrt{8}n^{2} \max_{\alpha \in \mathscr{T}} {\epsilon_{\alpha}} \leq \frac{\gamma}{2}.
\end{equation}
Now, note by parts (1.a) and (2.a) of Proposition \ref{bad-sets}, for each $i \in \{1, \dots, k\}$, we have picked the set $\mathcal{B}^c_{m_i,\alpha_i}({Z}_{{\bm{s}}+\widetilde{\beta}_{i-1}}^\infty(t_i^+))$ so that by \eqref{X^N - X^infty} we have $Z_{{\bm{s}} + \widetilde{\beta}_{i}}^{{\bm{N}}}(t_i^-)$ lies in the interior of the phase space ${\mathcal{D}}_{\bm{\epsilon}}^{{\bm{s}} + \widetilde{\beta}_{i}}$. Therefore, we have 
\begin{equation}\label{ti-1 - ti}
{\Psi}_{{\bm{s}}+\widetilde{\beta}_i, \bm{\epsilon}}^{t_{i+1} - t_{i}} (Z_{{\bm{s}}+\widetilde{\beta}_{i}}^{{\bm{N}}}(t_{i}^-))= \Phi_{{\bm{s}}+\widetilde{\beta}_i}^{t_{i+1} - t_{i}} (Z_{{\bm{s}}+\widetilde{\beta}_{i}}^{{\bm{N}}}(t_{i}^-))  =  Z_{{\bm{s}}+\widetilde{\beta}_i}^{{\bm{N}}}(t_{i+1}^+).
\end{equation}
Combining \eqref{t0 - t1} and \eqref{ti-1 - ti} together, we obtain the following expansion for the truncated observable
\begin{align*}
J_{{\bm{s}},k,R,\delta}^{{\bm{N}}} (t,\bm{\alpha},\bm{\beta},J,M)(X_{{\bm{s}}} )  & = A_{{\bm{s}},\bm{\alpha}, \bm{\beta},k}^{{\bm{N}}} \int_{\mathcal{M}^c_{{\bm{s}}}(X_{{\bm{s}}}) } \phi_{{\bm{s}}}(V_{{\bm{s}}}) \int_{\mathcal{T}_{k,\delta}} \int_{B_1^c } \cdots \int_{B_k^c} f^{({\bm{s}} + \widetilde{\beta}_k)}_{{\bm{N}},0}\left({Z}_{{\bm{s}}+\widetilde{\beta}_k}^{{\bm{N}}}(0^+)\right)  \\
& \prod_{i=1}^k j_i \langle \omega_i , (v_{s_{\beta_i} + \widetilde{\beta}_{i-1}^{\beta_i} + 1}^{\beta_i})^{{\bm{N}}} (t_i^+)- v_{m_i}^{\alpha_i} \rangle_+ d\omega_k dv_{s_{\beta_k} + \widetilde{\beta}_{k-1}^{\beta_k}+1}^{\beta_k}\dots d\omega_1 dv_{s_{\beta_1} + 1}^{\beta_1}dt_k \dots dt_1dt dV_{{\bm{s}}}
\end{align*}
where we are defining $B_i^c:= \mathcal{B}^c_{m_i,\alpha_i}({Z}_{{\bm{s}}+\widetilde{\beta}_{i-1}}^\infty(t_i^+))$ as in \eqref{pseudo applicable}. The constant $A_{{\bm{s}},\bm{\alpha}, \bm{\beta},k}^{{\bm{N}}}$ is given by the following formula:
$$
A_{{\bm{s}},\bm{\alpha}, \bm{\beta},k}^{{\bm{N}}} : = \prod_{i=1}^k A_{{\bm{s}} + \widetilde{\beta}_i, (\alpha_i,\beta_i)}^{{\bm{N}}}, \quad \text{where}\quad  A_{\bm{s}+\widetilde{\beta}_i, (\alpha_i,\beta_i) }^{\bm{N}} : = (N_{\beta_i}- s_{\beta_i} - \widetilde{\beta}_{i-1}^{\beta_i}) \epsilon_{(\alpha_i,\beta_i)}^{d-1}.$$
\begin{rmk}\label{const conv upwards}
Note that as $\bm{N}\rightarrow \infty$ and $\bm{\epsilon} \rightarrow 0$ according to the scalings \eqref{mixed-boltz-grad,1}, we have for fixed ${\bm{s}} \in \N^2$, $\bm{\alpha},\bm{\beta}\in S_k$, and $k \in \N_+$ that $A_{\bm{s},\bm{\alpha},\bm{\beta},k}^{\bm{N}} \nearrow A_{\bm{\alpha},\bm{\beta},k}^\infty$. Moreover, we have the trivial bound
\begin{equation}\label{constants bookkeeping}
0 < 1- (A^{\infty}_{\bm{\alpha},\bm{\beta},k})^{-1} A^{\bm{N}}_{\bm{s},\bm{\alpha},\bm{\beta},k} \leq C_{\bm{s}} \max_{\alpha \in \mathscr{T}} \epsilon_{\alpha}^{d-1}. 
\end{equation}
\end{rmk}
Let us now approximate the BBGKY truncated observables in terms of the Boltzmann initial data and the kernel coming from the Boltzmann pseudo-trajectories. Let $\bm{s} = (s_{(1,0)}, s_{(0,1)})\in \N_+^2$, $X_{{\bm{s}}} \in \Delta_{{\bm{s}}}^X(\epsilon_0)$, $1 \leq k \leq n$, $(J,M) \in \mathcal{U}_{{\bm{s}}, k, \bm{\beta}}$, and $(t_1, \dots, t_k) \in \mathcal{T}_{k, \delta}(t)$ for $t \in [0,T]$. Define for $\bm{\alpha},\bm{\beta} \in S_k$ the functional 
\begin{align*}
\widehat{J}_{{\bm{s}},k,R,\delta}^{{\bm{N}}} (t,\bm{\alpha},\bm{\beta},J,M)(X_{{\bm{s}}} )& := A^\infty_{\bm{\alpha},\bm{\beta},k}\int_{\mathcal{M}^c_{{\bm{s}}}(X_{{\bm{s}}}) } \phi_{{\bm{s}}}(V_{{\bm{s}}}) \int_{\mathcal{T}_{k,\delta}}\int_{B_1^c} \cdots \int_{B_k^c} f^{({\bm{s}} + \widetilde{\beta}_k)}_0\left({Z}_{{\bm{s}}+\widetilde{\beta}_k}^{{\bm{N}}}(0^+)\right)  \\
& \prod_{i=1}^k j_i \langle \omega_i , (v_{s_{\beta_i} + \widetilde{\beta}_{i-1}^{\beta_i} + 1}^{\beta_i})^{{\bm{N}}} (t_i^+)- v_{m_i}^{\alpha_i} \rangle_+ d\omega_k dv_{s_{\beta_k} + \widetilde{\beta}_{k-1}^{\beta_k}+1}^{\beta_k}\dots d\omega_1 dv_{s_{\beta_1} + 1}^{\beta_1}dt_k \dots dt_1dt dV_{{\bm{s}}}.
\end{align*}
We can now approximate the functional ${J}_{{\bm{s}},k,R,\delta}^{{\bm{N}}}$ in terms of the functional $\widehat{J}_{{\bm{s}},k,R,\delta}^{{\bm{N}}}$.

\begin{prop}\label{Jhat - J}
Let $n \in \N$, and $\bm{s} : =(s_1, s_2)\in \N_+^2$ such that $s_1, s_2< n$. Fix parameters $\gamma, \epsilon_0, R, \eta, \delta$ as in \eqref{parameter-relations} and $t \in [0,T]$. Then given $\zeta > 0$, there exists a pair $(N_1^*, N_2^*) \in \N^2$ with $N_i^*= N_i^*(\zeta, n, \gamma, \eta, \epsilon_0)$ such that for all $N_i \geq N_i^*$, i =1,2, for which $(\bm{N},\bm{\epsilon}) = (N_1, N_2, \epsilon_1, \epsilon_2)$ obey the scalings of \eqref{mixed-boltz-grad,1} we have 
\begin{align*}
\sum_{k=1}^n\sum_{\bm{\alpha}, \bm{\beta} \in S_{k}}\sum_{(J,M) \in \mathcal{U}_{{\bm{s}}, k, \bm{\beta}}}\Vert {J}_{{\bm{s}},k,R,\delta}^{{\bm{N}}} (t,\bm{\alpha},\bm{\beta},J,M)-\widehat{J}_{{\bm{s}},k,R,\delta}^{{\bm{N}}} (t,\bm{\alpha},\bm{\beta},J,M)\Vert_{L^\infty_{{\bm{s}}}(\Delta_{{\bm{s}}}^X(\epsilon_0))}&\leq C_{d,\bm{s}, \mu_0,T}^n \Vert \phi_{{\bm{s}}}\Vert_{L^\infty} R^{d(|\bm{s}| + 2n)} \zeta^2.
\end{align*}
In the case of conditioned tensorized initial data as in Theorem \ref{thm 3}, we have the upgraded bound 
\begin{align*}
    \sum_{k=1}^n\sum_{\bm{\alpha}, \bm{\beta} \in S_{k}}\sum_{(J,M) \in \mathcal{U}_{{\bm{s}}, k, \bm{\beta}}}\Vert {J}_{{\bm{s}},k,R,\delta}^{{\bm{N}}} (t,\bm{\alpha},\bm{\beta},J,M)-\widehat{J}_{{\bm{s}},k,R,\delta}^{{\bm{N}}} (t,\bm{\alpha},\bm{\beta},J,M)&\Vert_{L^\infty_{{\bm{s}}}(\Delta_{{\bm{s}}}^X(\epsilon_0))}\\
    &  \leq C_{d,\bm{s}, \mu_0,T}^n \Vert \phi_{{\bm{s}}}\Vert_{L^\infty} R^{d(|\bm{s}| + 2n)} \max_{\alpha \in \mathscr{T}}\epsilon_\alpha 
\end{align*}
\end{prop}
\begin{proof} Fix $1 \leq k \leq n$, $\bm{\alpha}, \bm{\beta} \in S_k$, and $(J,M) \in \mathcal{U}_{{\bm{s}}, k, \bm{\beta}}$. Assume that $(N_1, N_2, \epsilon_1, \epsilon_2)$ obey the scalings \eqref{mixed-boltz-grad,1}. Then for $N_i$ large enough, the scaling assumption implies that 
\begin{equation}\label{large Ni}
n^{2}\max(\epsilon_1, \epsilon_2)\ll \gamma,
\end{equation}
where the implicit constants are universal. By an argument similar to that of \cite{AP19a}, we can bound 
\begin{align*}
&  \Vert {J}_{{\bm{s}},k,R,\delta}^{{\bm{N}}} (t,\bm{\alpha},\bm{\beta},J,M)-\widehat{J}_{{\bm{s}},k,R,\delta}^{{\bm{N}}} (t,\bm{\alpha},\bm{\beta},J,M)\Vert_{L^\infty_{{\bm{s}}}(\Delta_{{\bm{s}}}^X(\epsilon_0))} \leq   \frac{ C_{d,T,\mu_0}^k}{k!}\Vert \phi_{{\bm{s}}} \Vert_{L^\infty}R^{d(|\bm{s}|+2k)} \\
& \times \Bigg(   \Vert f^{({\bm{s}} + \widetilde{\beta}_k)}_{{\bm{N}},0}- f^{({\bm{s}} + \widetilde{\beta}_k)}_0\Vert_{L^\infty (\Delta_{{\bm{s}}+ \widetilde{\beta}_{k}}(\epsilon_0/2))}  +|(A^\infty_{\bm{\alpha}, \bm{\beta}})^{-1} A_{{\bm{s}},\bm{\alpha}, \bm{\beta},k}^{{\bm{N}}}  -1| \Vert F_0\Vert_{0, \gamma_0,\mu_0}\Bigg) 
\end{align*}
Summing the above estimate over all $\bm{\alpha},\bm{\beta}, J,M,$ and $k$, we use the elementary inequality
$$
\sum_{k=1}^n \frac{ C_{d,T,\mu_0}^k}{k!} 4^k 2^k|\bm{s}|(|\bm{s}| +1)\dots(|\bm{s}| + k) \leq C_{d,T,\mu_0, \bm{s}}^n.
$$
Using Definition \ref{approx-bbgky}, Remark \ref{const conv upwards}, and in the case of conditioned tensorized initial data, the estimate \eqref{convergence of approx conditioned} finishes the proof. 
\end{proof}
Next, we compare the functionals $ \widehat{J}_{{\bm{s}},k,R,\delta}^{{\bm{N}}} (t,\bm{\alpha},\bm{\beta},J,M)$ to $ J_{{\bm{s}},k,R,\delta}^{\infty} (t,\bm{\alpha},\bm{\beta},J,M)$. The following proposition crucially uses the continuity assumption on our initial data $F_0$.

\begin{prop} \label{Jhat - Jinfty}
Let $n \in \N$, and $s_1, s_2 < n$. Fix parameters $\gamma, \epsilon_0, R, \eta, \delta$ as in \eqref{parameter-relations} and $t \in [0,T]$. Then given $\zeta > 0$, there exists a pair $(N_1^{**}, N_2^{**}) \in \N^2$ with $N_i^{**}= N_i^{**}(\zeta, n,c_1, c_2, a,b)$ such that for all $N_i \geq N_i^{**}$, i =1,2, for which $(N_1, N_2, \epsilon_1, \epsilon_2)$ obey the scalings of \eqref{mixed-boltz-grad,1} we have
\begin{align*}
\sum_{k=1}^n\sum_{\bm{\alpha}, \bm{\beta} \in S_{k}}\sum_{(J,M) \in \mathcal{U}_{{\bm{s}}, k, \bm{\beta}}}&\Vert \widehat{J}_{{\bm{s}},k,R,\delta}^{{\bm{N}}} (t,\bm{\alpha},\bm{\beta},J,M)-{J}_{{\bm{s}},k,R,\delta}^{\infty} (t,\bm{\alpha},\bm{\beta},J,M)\Vert_{L^\infty_{{\bm{s}}}(\Delta_{{\bm{s}}}^X(\epsilon_0))}\leq C_{d, \bm{s}, \mu_0,T}^n \Vert \phi_{{\bm{s}}}\Vert_{L^\infty} R^{d(|\bm{s}| + 2n)} \zeta^2
\end{align*}
For the case of conditioned, tensorized, and H\"older initial data in $\mathcal{C}^{0,\lambda}$, we have the improved estimate 
\begin{align*}
\sum_{k=1}^n\sum_{\bm{\alpha}, \bm{\beta} \in S_{k}}\sum_{(J,M) \in \mathcal{U}_{{\bm{s}}, k, \bm{\beta}}}\Vert \widehat{J}_{{\bm{s}},k,R,\delta}^{{\bm{N}}} (t,\bm{\alpha},\bm{\beta},J,M)-{J}_{{\bm{s}},k,R,\delta}^{\infty} (t,\bm{\alpha},\bm{\beta},J,M)&\Vert_{L^\infty_{{\bm{s}}}(\Delta_{{\bm{s}}}^X(\epsilon_0))}\\
& \leq C_{d, \bm{s}, \mu_0,T}^n \Vert \phi_{{\bm{s}}}\Vert_{L^\infty} R^{d(|\bm{s}| + 2n)} (\max_{\alpha \in \mathscr{T}}\epsilon_\alpha)^\lambda.
\end{align*}
\end{prop}
\begin{proof}
Let $\zeta > 0 $ be given. Fix $1 \leq k \leq n$, $\bm{\alpha},\bm{\beta} \in S_k$, and $(J,M) \in \mathcal{U}_{{\bm{s}}, k, \bm{\beta}}$. Then, since $s_1, s_2< n$, we may apply Lemma \ref{pseudo-comparison} to obtain
\begin{equation}\label{pseudo compare}
|Z_{{\bm{s}}+ \widetilde{\beta}_k}^{{\bm{N}}} (0^+)-{Z}_{{\bm{s}}+ \widetilde{\beta}_k}^\infty(0^+)| \leq \sqrt{8} n^{2} \max ( \epsilon_1, \epsilon_2).
\end{equation}
According to \eqref{mixed-boltz-grad,1}, there exists $(N_1^{**}, N_2^{**}) \in \N^2$ with $N_i^{**}= N_i^{**}(\zeta, n, c_1, c_2, a,b)$ such that for all $N_i \geq N_i^{**}$, the right hand side of \eqref{pseudo compare} is so small that by the continuity condition \eqref{unif-cont-cond} we have
$$
|f^{({\bm{s}} + \widetilde{\beta}_k)}_0(Z_{{\bm{s}}+ \widetilde{\beta}_k}^{{\bm{N}}} (0^+))-f^{({\bm{s}} + \widetilde{\beta}_k)}_0({Z}_{{\bm{s}}+ \widetilde{\beta}_k}^\infty(0^+))| \leq C^{|\bm{s}| +k -1} \zeta^2,
$$
for all $Z_{{\bm{s}}} \in \R^{2d (s_1 + s_2)}$. As in the proof of Proposition \ref{Jhat - J}, after summing this inequality over $1 \leq k \leq n$, $\bm{\alpha},\bm{\beta} \in S_k$, and $(J,M) \in \mathcal{U}_{{\bm{s}}, k, \bm{\beta}}$ proves the first part of the proposition. For the second estimate, note that for any $Z_{\bm{\ell}}, Z_{\bm{\ell}}' \in \R^{2d|\bm{\ell}|}$ we have by induction for $\bm{\ell}=(\ell_1,\ell_2) \in \N_+^2$
\begin{align}
    |g_0^{\otimes \ell_1}\otimes h_0^{\otimes \ell_2}(Z_{\bm{\ell}}) - g_0^{\otimes \ell_1}\otimes h_0^{\otimes \ell_2}(Z_{\bm{\ell}}')| \leq C_{d,\bm{\ell}} |Z_{\bm{\ell}} - Z_{\bm{\ell}}'|^\lambda.
\end{align}
Applying this estimate with \eqref{pseudo compare} we finish the proof. 
\end{proof}
\subsection{Proof of the Main Theorem}
 \begin{proof}[Proof of Theorem \ref{main-thm}]
Here we choose parameters $n \in \N$, $\delta, \eta, \gamma,\epsilon_0 >0$ and $R>1$ to show convergence in Theorem \ref{main-thm}. First, fix ${\bm{s}}=(s_1,s_2) \in \N^2$ and $\phi_{{\bm{s}}} \in C_c( \R^{d|\bm{s}| })$. Define the constant $ C_{\bm{s}, \gamma_0,\mu_0,T} > 1$ to be the maximum of all of the constants found in  Propositions \ref{reduction}, \ref{restriction to initially good conf}, \ref{truncated element estimate}, \ref{Jhat - J}, \ref{Jhat - Jinfty}. Then define the constant
 $$
C : = C_{\bm{s}, \gamma_0,\mu_0,T} \Vert \phi_{{\bm{s}}} \Vert_{L^\infty_{{\bm{s}}}} \max ( 1, \Vert F_0 \Vert_{\infty,\gamma_0,\mu_0} ) >1.
$$
Then, let $\sigma > 0$ and $0 < \zeta <1 $ be sufficiently small so that 
\begin{equation}\label{small zeta}
\zeta e^{\gamma_0 \zeta^{-1}/3} > C.
\end{equation}
We pick parameters so that for all sufficiently large $N_1, N_2$ and sufficiently small $\epsilon_1,\epsilon_2$ which obey \eqref{mixed-boltz-grad,1}, we have for all $t \in [0,T]$ that $
\Vert I_{{\bm{s}}}^{{\bm{N}}}(t) - I_{{\bm{s}}}^\infty(t) \Vert_{L^\infty ( \Delta_{{\bm{s}}}^X(\sigma))}\lesssim \zeta.$
We choose these parameters in the following order: 
\begin{itemize}
\item[$(P_1)$] Pick $n \in \N$ such that $n> \max(s_1, s_2, \log_{4}(C\zeta^{-1}) )$. This implies that $s_1, s_2 < n$ and $C4^{-n} < \zeta$.
\item[$(P_2)$] Pick $\delta > 0$ such that $\delta < \zeta C^{-(n+1)}$. This implies that $\delta C^{n+1} < \zeta$. 
\item[$(P_3)$] Pick $\eta>0$ such that $\eta < \zeta^{4/(d-1)}$. This implies that $\eta^{(d-1)/2} < \zeta^2$. 
\item[$(P_4)$] Pick $R>1$ such that $ \max(1, \sqrt{3} \gamma_0^{-1/2} \log^{1/2}(C \zeta^{-1}) )< R < \zeta^{-1/(4dn)} C^{-1/(4d)}$. This implies that $\zeta^2 R^{4dn} C^n < \zeta$ and $C e^{-\gamma_0 R^2/3} < \zeta$. 
\item[$(P_5)$] Pick $\epsilon_0>0$ such that $\epsilon_0 \ll \eta \delta$ as in \eqref{parameter-relations}, and $\epsilon_0 < \sigma$. 
\item[$(P_6)$] Pick $\gamma >0$ such that $\gamma \ll \epsilon_0$ and $\gamma \ll \epsilon_0 R^{-1} \eta$ as in \eqref{parameter-relations}.
\end{itemize}
Note that $(P_1)$ implies that $s_1, s_2<n$, so the fact that $R> 1$ implies that $
R^{d|\bm{s}|}, R^{d(|\bm{s}| + 2n)} \leq R^{4dn}.$ Moreover, since $\frac{d-1}{2} < \frac{(d-1)(d+2)}{2d+2}$ and $\eta < 1$, we have that $\eta^{(d-1)/2} > \eta^{(d-1)(d+2)/(2d+2)}$. Hence we have 
$$
C^n( R^{d|\bm{s}|} \eta^{(d-1)/2} + R^{d(|\bm{s}| + 2n)} \eta^{(d-1)(d+2)/(2d+2)}) \leq 2C^nR^{4dn}\eta^{(d-1)/2}
$$
$$
C^n R^{d(|\bm{s}| + 2n)} \zeta \leq C^n R^{4dn} \zeta.
$$
Now note that for $i =1,2$, if $N_i \geq \max(N_i^*, N_i^{**})$ and $N_i$ so large that $\max(\epsilon_1, \epsilon_2) \ll \gamma$ as in \eqref{parameter-relations}, our parameter choices $(P_1)$-$(P_6)$ imply that the conditions on our parameters from Propositions \ref{reduction}, \ref{restriction to initially good conf}, \ref{truncated element estimate}, \ref{Jhat - J}, \ref{Jhat - Jinfty} hold. Moreover, we have that 
\begin{align}
\Vert I_{{\bm{s}}}^{{\bm{N}}}(t) - I_{{\bm{s}}}^\infty(t) \Vert_{L^\infty ( \Delta_{{\bm{s}}}^X(\epsilon_0))} &\leq  C(4^{-n} + e^{-\gamma_0 R^2 /3} + \delta C^n)  + \\
& C^n( R^{d|\bm{s}|} \eta^{(d-1)/2} + R^{d(|\bm{s}| + 2n)} \eta^{(d-1)(d+2)/(2d+2)}) + \nonumber \\
& C^n R^{d (|\bm{s}| + 2n)}\zeta^2\nonumber \\
& \leq 6 \zeta \label{6 zeta}
\end{align}
for all $N_i \geq \max(N_i^* , N_i^{**})$, $(N_1, N_2,\epsilon_1, \epsilon_2)$ that satisfy the mixed Boltzmann-Grad scalings \eqref{mixed-boltz-grad,1}. Because of our choice $(P_5)$, we have $\epsilon_0 < \sigma$, which implies $\Delta_{{\bm{s}}}^X(\sigma) \subset \Delta_{{\bm{s}}}^X(\epsilon_0)$. This, combined with the bound \eqref{6 zeta} concludes the proof.
\end{proof}

\begin{proof}[Proof of Theorem \ref{thm 3}]
The inclusion of $F_0$ and $F$ in the correct spaces follows from the definition of the norms, and the continuity estimate of Theorem \ref{pde-lwp}. It follows by a computation that $F$ indeed solves the Boltzmann hierarchy. By Theorem \ref{main-thm} it thus suffices to check the continuity estimate. This follows from an induction using the bound proved in Theorem \ref{pde-lwp} and our assumption that $|u_0|_{\gamma_0, \mu_0+1} \leq 1/2$. Using the upgraded estimates in Propositions \ref{Jhat - J}, \ref{Jhat - Jinfty}, we obtain as above 
\begin{align*}
\Vert I_{{\bm{s}}}^{{\bm{N}}}(t) - I_{{\bm{s}}}^\infty(t) \Vert_{L^\infty ( \Delta_{{\bm{s}}}^X(\epsilon_0))} &\leq  C(4^{-n} + e^{-\gamma_0 R^2 /3} + \delta C^n)  + \\
& C^n( R^{d|\bm{s}|} \eta^{(d-1)/2} + R^{d(|\bm{s}| + 2n)} \eta^{(d-1)(d+2)/(2d+2)}) +  \\
& C^n R^{d (|\bm{s}| + 2n)} \max_{\alpha \in \mathscr{T}} (\epsilon_\alpha)^\lambda.
\end{align*}
Picking parameters $n, \delta, \eta, R, \epsilon_0$ and $\gamma$ in a similar manner to $(P_1) - (P_6)$ above, we obtain the claimed convergence rate in $\epsilon = \max_{\alpha \in \mathscr{T}}\epsilon_\alpha$. For details of a related calculation, see e.g. \cite{AP19b}.
\end{proof}
\nocite{*}
\bibliography{mixture}

\end{document}